\documentclass[aap]{imsart}

\usepackage{amsmath,amsthm,amssymb,amsfonts,epsfig,color,graphicx,enumerate,enumitem,accents}
\usepackage{rotating,multirow,color,enumerate,url,mathrsfs,mathrsfs}
\usepackage{arydshln}
\usepackage{bbm}
\usepackage{cases}
\usepackage{esint}
\usepackage{gensymb}
\usepackage{epstopdf}
\usepackage{comment}
\usepackage{leftidx}
\usepackage{graphicx}
\usepackage{mathtools}
\usepackage{subfig}
\usepackage[hidelinks]{hyperref}

\usepackage{float}
\usepackage{caption}
\newtheorem{theorem}{Theorem}[section]
\newtheorem{corollary}[theorem]{Corollary}
\newtheorem{lemma}[theorem]{Lemma}
\newtheorem{proposition}[theorem]{Proposition}

\theoremstyle{definition}

\newtheorem{remark}[theorem]{Remark}
\newtheorem{example}[theorem]{Example}

\numberwithin{equation}{section}

\newcommand{\eps}{\varepsilon}

\newcommand{\R}{\mathbb{R}}

\newcommand{\Rd}{{\R^d}}

\newcommand{\ov}{\overline}

\newcommand{\argmin}{\operatorname*{arg\,min}}

\newcommand{\opnorm}[1]{{\left\vert\kern-0.25ex\left\vert\kern-0.25ex\left\vert #1 
		\right\vert\kern-0.25ex\right\vert\kern-0.25ex\right\vert}}

\newcommand{\argu}{{\,\cdot\,}}

\newcommand{\Mes}{{\mathcal{M}}}

\newcommand{\cost}{{\mathsf{c}}}

\newcommand{\pairing}[1]{{\left \langle #1 \right \rangle}}
\newcommand{\norm}[1]{\Arrowvert #1 \Arrowvert}
\newcommand{\abs}[1]{{\left \lvert #1 \right \rvert}}

\newcommand{\FF}{{G}}

\newcommand{\mres}{\mathbin{\vrule height 1.6ex depth 0pt width
		0.13ex\vrule height 0.13ex depth 0pt width 1.3ex}}

\newcommand{\one}{{{\bf 1}
		\kern-0,28em \rm l}}

\DeclareMathOperator{\spt}{spt}

\newcommand{\ba}{b}

\begin{document}

	\begin{frontmatter}
		\title{Bi-martingale optimal transport and its applications}
		%\title{A sample article title with some additional note\thanksref{T1}}
		\runtitle{Bi-martingale optimal transport}
		%\thankstext{T1}{A sample of additional note to the title.}
		
		\begin{aug}
			%%%%%%%%%%%%%%%%%%%%%%%%%%%%%%%%%%%%%%%%%%%%%%%
			%% Only one address is permitted per author. %%
			%% Only division, organization and e-mail is %%
			%% included in the address.                  %%
			%% Additional information such as            %%
			%% identifying the corresponding author must %%
			%% be included in in the Acknowledgments     %%
			%% section if necessary.                     %%
			\author{Karol Bołbotowski}
			
			\address{University of Warsaw, Faculty of Mathematics, Informatics and Mechanics}

		\end{aug}
		
		\begin{abstract} 	
			We introduce a new non-linear optimal transport formulation for a pair of probability measures on $\Rd$ 
			sharing a common barycentre, in which admissible transference plans satisfy two martingale-type constraints. This bi-martingale framework underlies and interconnects several variational problems on the space of probability measures. For the quadratic cost, it provides an optimal transport interpretation of the second Zolotarev distance on $\mathcal{P}_2(\Rd)$. For a broader class of convex costs, it leads to optimization problems under convex order constraints, encompassing in particular the Zolotarev projection onto the cone of dominating probability measures. As a main application, we construct a 
			$\Gamma$-convergent bi-martingale approximation of the classical martingale optimal transport problem. This scheme robustly accommodates deviations from convex order between the marginal distributions and overcomes 
			the well-known instability of MOT with respect to variations of the marginals in higher dimensions.
		\end{abstract}
		
		

		\end{frontmatter}

	%    Abstract is required.
	
	\maketitle

%	\vskip2cm
	
	%%%%%%%%

\section{Introduction}

\subsection{Selected topics on the space of probabilities $\mathcal{P}(\Rd)$}

Let us open the introduction with a recapitulation of several issues revolving around the space of probability measures on $\Rd$ such as optimal mass transport, metrics, and stochastic ordering. The purpose of this paper is to propose a new optimization framework that bridges these topics and provides a springboard for  computational methods.

\bigskip

\textit{Optimal transport and its martingale variant} \nopagebreak

\medskip

With $\mu$ and $\nu$ being probability measures on closed  subsets of $\Rd$, $X$ and $Y$ respectively, the Monge--Kantorovich formulation of optimal transport  consists in finding a plan of minimum  cost,
\begin{align*}
	\qquad  \qquad  \qquad \quad
	\inf\left\{ \int_X \! \int_Y c (x,y)\, \gamma(dxdy) \, : \, 	\gamma \in \Gamma(\mu,\nu)   \right\} \qquad  \qquad  \qquad   (\mathrm{OT}).
\end{align*}
Above $\Gamma(\mu,\nu)$ is a subset of probabilities on $X \times Y$ whose marginal distributions are $\mu$ and $\nu$. Optimal transport has countless applications and is a topic of ever growing interest, see the monographs \cite{villani,ambrosio2008,santambrogio2015}. Choosing the cost $c(x,y) = \abs{x-y}^p$ for $p \in [1,+\infty)$ yields the Wasserstein metric $W_p(\mu,\nu)$ on the space of probability measures of finite $p$-th moment, further denoted by $\mathcal{P}_p(\Rd)$.

Motivated by applications in finance, the authors of \cite{beiglbock2013} have introduced the variation of $(\mathrm{OT})$ called the \textit{martingale optimal transport},
\begin{align*}
	\qquad  \qquad  \qquad \quad
	\inf\left\{ \int_X \! \int_Y c (x,y)\, \gamma(dxdy) \, : \, 	\gamma \in \Gamma_{\mathrm{M}}(\mu,\nu)   \right\}  \qquad  \qquad  \quad (\mathrm{MOT}).
\end{align*}
The set $\Gamma_{\mathrm{M}}(\mu,\nu)$ consists of those transport plans $\gamma \in \Gamma(\mu,\nu)$ that satisfy the martingale condition,
\begin{equation*}
  \int_Y y \, \gamma_x(dy) = x\quad \text{for $\mu$-a.e. $x$},
\end{equation*}
where $x \mapsto \gamma_x$ is the \textit{transport kernel}, a Borel measure-valued map that constitutes the disintegration $\gamma(dxdy) = \mu(dx) \otimes \gamma_x(dy)$. As expounded in \cite{beiglbock2013}, the minimum in $(\mathrm{MOT})$ represents the lowest price of an exotic option with payoff $c$ for which no arbitrage opportunity occurs in the market.

Unlike $\Gamma(\mu,\nu)$, the set of martingale plans may be empty. Strassen theorem \cite{strassen1965} states that $\Gamma_{\mathrm{M}}(\mu,\nu)$ is non-empty if and only if $\nu$ dominates $\mu$ in the convex order, namely,
\begin{equation*}
	\int_\Rd \varphi \,d\mu \leq \int_\Rd \varphi \, d\nu \qquad\text{for all convex } \varphi:\Rd \to\R,
\end{equation*}
 written $\mu \preceq_c \nu$ for short.
A necessary condition, but by no means sufficient, is the agreement of the barycentres $[\mu] = [\nu]$, where $[\mu] := \int x \,\mu(dx)$.

\bigskip

\textit{Zolotarev distance as an alternative for Wasserstein} \nopagebreak

\medskip

In the paper \cite{zolotarev1978} Vladimir M. Zolotarev  introduced a family of metrics on the space of probabilities on $\Rd$. The Zolotarev distance of order $s= 1,2,\ldots$ reads  for $\mu,\nu \in \mathcal{P}_s(\Rd)$, 
\begin{equation*}
	Z_s(\mu,\nu) = \sup_{ u \in C^{s-1,1}(\Rd)}\left\{ \int_\Rd u \, d(\mu-\nu) \, : \,  \norm{D^{s-1}u(x) - D^{s-1}u(y)}  \leq \abs{x-y} \ \ \forall\, x,y \in \Rd \right\}
\end{equation*}
where $C^{k,1}(\Rd)$ is the space of functions whose $k$-th derivative is Lipschitz continuous, and $\norm{\argu}$ is the relevant operator norm. Strictly speaking, $Z_s$ is an extended metric when $s \geq\nolinebreak 2$ since $Z_s(\mu,\nu) < +\infty$ implies equality of the mixed moments of order up to $s-1$. For instance, the distance $Z_2$ should be defined on $\mathcal{P}_2^\ba(\Rd) := \big\{ \mu \in \mathcal{P}_2(\Rd) \, : \, [\mu] = \ba\big\}$, where $\ba \in \Rd$ is a generic but fixed barycentre.

The work \cite{belili2000} demonstrates that Zolotarev and Wasserstein distances $Z_s$ and $W_s$ control the same topology on the subspace of probabilities in $\mathcal{P}_s(\Rd)$ with prescribed mixed moments up to order $s-1 $. Moreover, $Z_s$ is an example of an \textit{ideal metric}: it is scale-$s$-homogeneous and it is subadditive under convolution, cf. \cite[Chapter 14]{rachev1991} for more details. These properties make $Z_s$ useful in obtaining sharp convergence rates in the Central Limit Theorem \cite{rio2009,fathi2021}. 

The famous Kantorovich--Rubinstein duality result states that $Z_1 = W_1$. It has an abundance of ramifications, one of which is the OT approach for the Beckmann problem \cite{bouchitte2001,dweik2019p}, a PDE formulation that is dual to $Z_1$. Few similar results are available for $s\geq2$. One example is \cite{hanin1994} where a family of optimal \textit{transshipment} problems is considered, see also \cite{rachev2000}. These, however, may suffer from the lack of solutions, unlike the transport formulations. 

In a recent paper \cite{bolbotowski2024kantorovich}, co-written by the present author, the Kantorovich--Rubinstein result has been brought to the second-order case $s=2$. For two measures $\mu,\nu \in \mathcal{P}_2^\ba(\Rd)$ sharing the barycentre $\ba $,  the distance
\begin{equation}
	\label{eq:Z2}
	Z_2(\mu,\nu) = \max\left\{ \int_\Rd u \, d(\mu-\nu) \, : \, u \in C^{1,1}(\Rd), \ \mathrm{lip}( \nabla  u) \leq 1 \right\}
\end{equation}
was recast as a minimum with respect to transport 3-plans $\varpi \in \mathcal{P}(X \times Y \times \Rd)$ with suitable marginal and martingale constraints (cf. Section \ref{ssec:3-plans} below). The motivation behind \cite{bolbotowski2024kantorovich} comes from the optimal design of mechanical structures. The PDE dual of $Z_2$ is a second-order counterpart of the Beckmann problem, and its solutions are symmetric-tensor-valued measures that encode configurations of beams constituting the stiffest bearing system of a ceiling. A partial characterization of optimal 3-plans was proposed in \cite{ciosmak2025}.

The duality result in \cite{bolbotowski2024kantorovich} paves the way to establishing new properties of the $Z_2$ distance. It was exploited directly in the subsequent paper \cite{bolbotowski2025}, establishing the inequalities comparing the Zolotarev-2 and Wasserstein-2 metrics,
\begin{equation}
	\label{eq:rio}
	\frac{1}{4}\, W_2^2(\mu,\nu) \ \leq \ Z_2(\mu,\nu) \ \leq \ \frac{1}{2} \, (\sigma_\mu + \sigma_\nu)\, W_2(\mu,\nu) \qquad \quad \forall\, \mu,\nu \in \mathcal{P}_2^\ba(\Rd),
\end{equation}
where $\sigma_\mu, \sigma_\nu$ stand for the standard deviations. The above constants are sharp. In dimension one, the lower bound was proven in \cite{rio2009}. These estimates allow to retrieve the topological equivalence between $Z_2$ and $W_2$ established already in \cite{belili2000}. Thus, on the space $ \mathcal{P}_2^\ba(\Rd)$ the convergence in $Z_2$ enjoys the characterization that is well known for Wasserstein-2,
\begin{equation}
	\label{eq:Z2_convergence}
	Z_2(\mu,\mu_n) \to 0 \qquad \Leftrightarrow \qquad \mu_n \rightharpoonup \mu, \qquad m_2(\mu_n) \to   m_2(\mu),
\end{equation}
where $m_2(\mu) =  \int\abs{x}^2 \mu(dx)$ is the second moment, and $\mu_n \rightharpoonup \mu$ stands for the weak convergence of measures, i.e. $\int \varphi \,d\mu_n \to \int \varphi \,d\mu$ for every continuous bounded function $\varphi \in C_b(\Rd)$.

\bigskip

\textit{Optimization with stochastic order constraints} \nopagebreak

\medskip

The work \cite{dentcheva2003optimization} has initiated a new line of research on optimization in mathematical finance by incorporating constraints in the form of stochastic order between random vectors, see e.g. \cite{dentcheva2015,gutjahr2016} for further developments. In particular cases, the problem can be written in terms of the laws only, 
\begin{equation}
	\label{eq:stochastic}
	\inf \left\{\int_\Rd f \, d\rho \, : \, \rho \in \mathcal{P}_1(\Rd), \ \ \rho \succeq_{\mathrm{st}} \mu_i, \ \ i =1,\ldots,n \right\},
\end{equation}
where $f:\Rd \to \R$ is a Borel cost function, and $\mu_i$ are given reference probabilities.
The stochastic order $\succeq_{\mathrm{st}}$ means that $\int \varphi\,d\rho \geq \int \varphi \, d\mu_i$ for every $\varphi$ within a certain  class $\mathcal{F}$ of real functions  \cite{shaked2007,muller2017}. The convex order is one example, subharmonic order is another. For instance, to implement risk-averse investing strategies, $\mathcal{F}$ should be chosen as the set of increasing concave functions ($d=1$).

This paper will focus on the stochastic optimization for convex order only. The most important instance of this problem will involve two reference probabilities  and a quadratic cost,
\begin{equation}
	\label{eq:minDom}
	\mathcal{C}(\mu,\nu) := \inf\Big\{ m_2(\rho) \, : \, \rho \in \mathcal{P}_2(\Rd), \ \ \rho \succeq_c \mu, \ \ \rho \succeq_c \nu \Big\},
\end{equation}
where $m_2(\rho) = \int_\Rd \abs{z}^2 \rho(dz)$. Its significance stems from the connection with Zolotarev-2 distance, obtained as a by-product of the second-order Kantorovich--Rubinstein duality \cite{bolbotowski2024kantorovich},
\begin{equation}
	\label{eq:Z2_C}
	Z_2(\mu,\nu) = \mathcal{C}(\mu,\nu) - \frac{m_2(\mu)+ m_2(\nu)}{2}.
\end{equation}

\bigskip

\textit{Projection onto the cone of probabilities dominating for convex order} \nopagebreak

\medskip

As mentioned above, well posedness of martingale optimal transport entails convex order between the marginals. This calls for careful sampling techniques when tackling $(\mathrm{MOT})$ numerically, as the convex order that holds for original probabilities can be lost upon discretization. The authors of \cite{alfonsi2020sampling} proposed  the \textit{Wasserstein projection} approach as a post-sampling fix of convex order, see also \cite{gozlan2020,kim2024}. Assuming that the convex order is corrupted, $\mu \npreceq_c\nu$, it consists in solving the problem,
\begin{align*}
		 \min \Big\{ W_2(\nu,\rho) \, :\, \rho \in \mathcal{P}_2(\Rd), \  \rho \succeq_c \mu  \Big\}.
\end{align*}
Its solution $\hat\rho$ is the Wasserstein-2 projection of $\nu$ onto the convex cone of measures that dominate $\mu$, also called  the \textit{forward projection}. The idea is then to replace $\nu$ with $\hat\rho$ and subsequently solve $(\mathrm{MOT})$ for the ordered measures $\mu \preceq_c \hat\rho$. 
Alternatively, one can compute the \textit{backward projection}: of $\mu$ onto the backward cone $\{ \argu \preceq_c \nu\}$.

If, despite  the lack of convex order, the probabilities share their barycentre, the aforementioned equivalence between Wasserstein-2 and Zolotarev-2 metrics inspires to investigate the projection technique with  $W_2$ replaced by $Z_2$. In this paper, we will study only the forward Zolotarev projection,
\begin{equation*}
	\Pi_{\succeq_c \mu} (\nu) :=  \argmin\limits_{\rho \succeq_c \mu} Z_2(\nu,\rho).
\end{equation*}
The equality \eqref{eq:Z2_C} gives away that the $Z_2$ metric itself is intrinsically linked to convex order, a feature not \textit{a priori} shared by $W_2$. This factor will prove to play into the properties of the Zolotarev projection.

\subsection{Bi-martingale optimal transport as the unifying framework}
\label{ssec:bimatringale_intro}

The topics covered above have several  common threads, e.g. convex order. In what follows we show that they all can be encompassed by a single optimization formulation  -- a new variant of optimal transport with two martingale constraints.

For $\mu,\nu \in \mathcal{P}_1^\ba(\Rd)$
we shall say that  a plan $\gamma \in \Gamma(\mu,\nu)$ is \textit{bi-martingale}\footnote{In \cite{ciosmak2025} a special case of bi-martingale transport is considered, when $\zeta(x,y) = \pi_{V_1} (x) + \pi_{V_2} (y)$, with the projections onto mutually orthogonal and complementing spaces $V_1 \oplus V_2 = \Rd$. Existence of bi-martingales $\gamma$ with respect to such coupling maps $\zeta$ require that the pair $\mu,\nu$ is in \textit{convex-concave order} with respect to $V_1,V_2$, cf. \cite{ciosmak2025} for details.} with respect to a Borel \textit{coupling map}\footnote{Not to be confused with the usual notion of transport map $T:X \to Y$.} $\zeta: X \times Y \to \Rd$ whenever
\begin{align}
	\label{eq:two_martingales}
	\gamma_i = (\pi_i,\zeta)\# \gamma \text{ is a martingale plan for $i=1,2$},
\end{align}
where $\pi_1(x,y)  = x$, $\pi_2(x,y)  = y$ are projections, and $\#$ is the push forward of a measure. Performing disintegrations with respect to the two marginals, $\gamma(dxdy) = \mu(dx) \otimes \gamma_x(dy)= \gamma_y(dx) \otimes \nu(dy)$, the condition \eqref{eq:two_martingales} can be equivalently put as,
\begin{equation}
	\label{eq:2xmartingale_condition}
	\begin{cases}
		\int_Y \zeta(x,y) \, \gamma_x(dy) = x & \ \ \text{for $\mu$-a.e. $x$},  \\  
		  \int_X \zeta(x,y) \, \gamma_y(dx) = y & \ \ \text{for $\nu$-a.e. $y$}.
	\end{cases}
\end{equation}
The common barycentre $\ba = [\mu] = [\nu]$ is a necessary and sufficient condition for existence of bi-martingale pairs $\gamma,\zeta$. An admissible pair can always be chosen as $\gamma = \mu\otimes \nu$ and $\zeta(x,y)= x+y-\ba$. An example of a bi-martingale plan is shown in Fig. \ref{fig:bimart}.

		\begin{figure}[h]
	\centering
	\includegraphics*[trim={0cm 0cm -0cm -0cm},clip,width=0.4\textwidth]{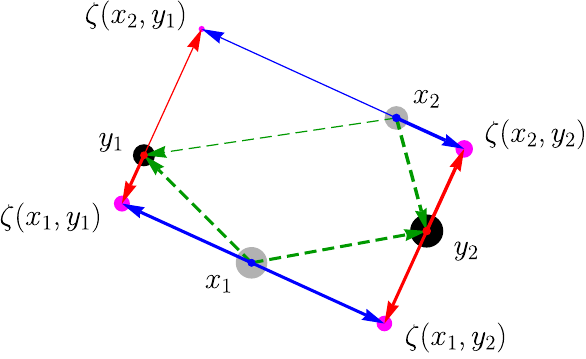}
	\caption{Example of a bi-martingale plan $\gamma$ (green dashed lines) with respect to a coupling map $\zeta$ for two-point data $\mu$ (gray) and $\nu$ (black). The martingale plans $\gamma_1 = (\pi_1,\zeta) \# \gamma$ (blue) and $\gamma_2 = (\pi_2,\zeta) \# \gamma$ (red) share the second marginal $\rho = \zeta \# \gamma$ (magenta) that dominates both $\mu$ and $\nu$ for convex order.}
	\label{fig:bimart}      
\end{figure}

The coupling map $\zeta = \zeta(x,y)$ will be involved  in the transport cost. To that end, we take a 3-point cost function that is jointly lower semi-continuous and convex in the third argument.
\begin{equation*}
	X \times Y \times \Rd \ni (x,y,z) \  \mapsto \ \cost(x,y,z) \in \R \cup \{+\infty\}.
\end{equation*}
The bi-martingale optimal transport problem consists in searching both $\gamma$ and $\zeta$,
\begin{align}
	\label{eq:M2OT_zeta}
	\inf\left\{ \int_X \! \int_Y \cost  \big(x,y,\zeta(x,y) \big) \gamma(dxdy) \ : \ 	\gamma \in \Gamma(\mu,\nu), \   \zeta \in L^1_\gamma(X\times Y;\Rd), \  \text{\eqref{eq:two_martingales} is satisfied} \right\}
\end{align}
The natural choices of the cost $\cost$ are as follows: 
\begin{enumerate}[label={(\Alph*)}]
	\item With a lower semi-continuous cost $c: X\times Y \to \R \cup \{+\infty\}$, take
	\begin{equation*}
		\cost(x,y,z) = c(x,y) + \chi_{\{y\}}(z),
	\end{equation*}
	where $ \chi_{\{y\}}(z)$ equals zero if $z=y$ and $+\infty$ otherwise. Clearly, finiteness of the cost in \eqref{eq:M2OT_zeta} enforces that $\zeta = \pi_2 = y$. For such a cost the second condition in \eqref{eq:2xmartingale_condition} is void, whilst the first one is the standard martingale condition. In this manner we are able to recast the classical $(\mathrm{MOT})$ through the bi-martingale framework.
	\item  Cost that are independent of $x,y$: for a convex function $f:\Rd \to \R \cup \{+\infty\}$  consider,
	\begin{equation*}
		\cost(x,y,z) = f(z).
	\end{equation*}
	For such a cost the bi-martingale formulation will prove equivalent to the optimal convex dominance of type \eqref{eq:stochastic} via the change of variable,
	\begin{equation*}
		\rho = \zeta \# \gamma.
	\end{equation*}
	The dominance $\rho \succeq_c \mu$, $\rho \succeq_c \nu$ can be deduced from Strassen theorem. Indeed, both martingale plans in \eqref{eq:two_martingales} have $\rho$ as their second marginal, i.e. $\gamma_1 \in \Gamma_{\mathrm{M}}(\mu, \rho)$, 	$\gamma_2 \in \Gamma_{\mathrm{M}}(\nu, \rho)$, cf. Fig. \ref{fig:bimart}.
	
	\item Inspired by the Wasserstein distance we can propose,
	\begin{equation*}
		\cost(x,y,z) = \frac{1}{p}\abs{z-x}^p + \frac{1}{p}\abs{z-y}^p, \qquad p\in\, [1,+\infty).
	\end{equation*}
	It will turn out that for $p=2$ we retrieve the metric $Z_2(\mu,\nu)$ as the minimal transport cost.
\end{enumerate}

\medskip

The bi-martingale problem \eqref{eq:M2OT_zeta} is not convex with respect to the pair $(\gamma,\zeta)$. However, a simple change of variables facilitates its convex reformulation. It is a technique known from the Benamou-Brenier \cite{benamou2000} dynamical $(\mathrm{OT})$. Consider the set of vector-valued coupling measures,
\begin{align*}
	Q(\mu,\nu) &:= \Big\{ q \in\Mes(X \times Y;\Rd ) \,:\, \pi_1 \# q = x\mu, \ \
	\pi_2 \# q= y \nu  \Big\}.
\end{align*}
We are interested in plans $\gamma$ in conjunction with absolutely continuous couplings $q$,
\begin{equation}
	\label{eq:GammaQ}
	\Gamma Q(\mu,\nu) := \Big\{ (\gamma,q) \in \Gamma(\mu,\nu)  \times Q(\mu,\nu) \, : \, q \ll \gamma \Big\}.
\end{equation}
One checks easily that, for any pair $(\gamma,q)$ in this set, $\gamma$ is a bi-martingale plan with respect to the Radon--Nikodym derivative $\zeta = \frac{dq}{d\gamma}$.
Readily, the convex formulation of the bi-martingale optimal transport reads,
\begin{align*}
	\qquad  \qquad  
	\inf\left\{ \int_X \! \int_Y \cost  \Big(x,y,\frac{d q}{d\gamma}(x,y) \Big) \gamma(dxdy) \ : \ 	(\gamma,q) \in \Gamma Q(\mu,\nu)  \right\} \qquad   (\mathrm{M^2OT}).
\end{align*}
To ensure well-posedness, we must assume that $\cost$ is superlinear in the third variable. This is crucial for providing stability of the absolute continuity assumption $ q \ll \gamma $. It is worth noting that problems focusing on optimizing vector-valued coupling measures are studied in their own right, see \cite{ciosmak2021}.

The convexity of $(\mathrm{M^2OT})$ raises the natural question of the dual problem. It bears resemblance to the dual of the classical $(\mathrm{MOT})$, where one looks for two scalar potentials $\varphi,\psi$ and a vector function $\Phi:\Rd \to \Rd$, cf. \cite{beiglbock2013}. However, the dual of $(\mathrm{M^2OT})$ is symmetric, which manifests itself in the form of another vectorial variable $\Psi$. Contrarily to the dual of $(\mathrm{MOT})$, where attainment is a highly delicate subject \cite{beiglbock2016,beiglbock2017complete,demarch2019}, the existence of a continuous quadruple $\varphi,\psi,\Phi,\Psi$ solving the dual of $(\mathrm{M^2OT})$ is well in reach under reasonable assumptions, e.g. on the smoothness of $\cost(x,y,\argu)$.

The subject of duality is intimately related to the numerical treatment of $(\mathrm{M^2OT})$. For instance, adding the entropic regularization term would lead to a variant of the Sinkhorn algorithm \cite{cuturi2013}, the state of  art in computational optimal transport. It is a block coordinate ascent method, and for $(\mathrm{M^2OT})$ the blocks of dual variables should be $(\varphi,\Phi)$ and $(\psi,\Psi)$.

In order to prevent the paper from branching out, we shall not elaborate either on duality or its numerical applications. Nonetheless, to give a flavour of the computational potential of $(\mathrm{M^2OT})$, in the example Section \ref{sec:examples} we will present its conic reformulation for discrete data, which then can be handled by  modern convex optimization solvers.

\subsection{Second-order Kantorovich--Rubinstein duality via $(\mathrm{M^2OT})$}

The first main result in this paper combines three topics touched in this introduction: Zolotarev-2 distance, optimal stochastic dominance, and the Zolotarev projection  $\Pi_{\succeq_c \mu}$. In particular, it translates the Kantorovich--Rubinstein duality for $Z_2$ developed in \cite{bolbotowski2024kantorovich} to the bi-martingale framework while providing new connection with the Zolotarev projection.
\begin{theorem}
	\label{thm:RK_reinvented}
	The following identities hold true for $\mu,\nu \in \mathcal{P}_2(\Rd)$ that share the barycentre $[\mu] = [\nu]$,
\begin{align*}
	Z_2(\mu,\nu) &= \min\left\{ \int_X \! \int_Y \!\frac{1}{2} \bigg( \Big| \frac{dq}{d\gamma}(x,y) - x \Big|^2\! \!\!+ \Big| \frac{dq}{d\gamma}(x,y) - y \Big|^2  \bigg)   \gamma(dxdy) \, : \, 	(\gamma,q) \in \Gamma Q(\mu,\nu)  \right\}, \\
	\mathcal{C}(\mu,\nu) &=  \min\left\{ \int_X \! \int_Y \Big| \frac{dq}{d\gamma}(x,y) \Big|^2   \gamma(dxdy) \, : \, 	(\gamma,q) \in \Gamma Q(\mu,\nu)  \right\},
\end{align*}
and both problems share a non-empty set of solutions.
Moreover, the following conditions are equivalent for any pair $(\gamma,q) = (\gamma,\zeta \gamma) \in \Gamma Q(\mu,\nu)$:
\begin{enumerate}[label={(\roman*)}]
	\item $(\gamma,\zeta\gamma)$ solves either of the two problems above;
	\item for any optimal potential $\ov{u}$ for the Zolotarev problem \eqref{eq:Z2} there holds $\zeta = \ov\zeta$ \ $\gamma$-a.e., where $\ov\zeta: \Rd \times \Rd \to \Rd$ is a Lipschitz continuous coupling map as below,
	\begin{equation}
			\label{eq:ovzeta}
		\ov\zeta(x,y) := \frac{x+y}{2} + \frac{\nabla \ov{u}(x) - \nabla \ov{u}(y)}{2};
	\end{equation}
		\item ${\rho} = {\zeta} \# {\gamma}$ is a convex dominant of the least second moment, namely $ {\zeta} \# {\gamma}$ solves   \eqref{eq:minDom};
		\item ${\rho} = {\zeta} \# {\gamma}$ is a Zolotarev projection of $\nu$ onto the cone $\{ \argu \succeq_c \mu\}$, namely $ {\zeta} \# {\gamma} \in \Pi_{\succeq_c \mu} (\nu)$;
		\item ${\rho} = {\zeta} \# {\gamma}$ is a Zolotarev projection of $\mu$ onto the cone $\{ \argu \succeq_c \nu\}$, namely $ {\zeta} \# {\gamma} \in \Pi_{\succeq_c \nu} (\mu)$.
\end{enumerate}
\end{theorem}

The characterization (ii) exposes the rigidity of the optimal coupling map and, at the same time, a freedom in choosing the bi-martingale  with respect to that map. For any solution $\ov{u}$ of the Zolotarev problem \eqref{eq:Z2_again} define the non-empty  set of bi-martingale plans with respect to the coupling map $\ov\zeta$,
\begin{equation}
	\label{eq:ovGamma}
	\ov\Gamma(\mu,\nu) := \left\{ \gamma \in \Gamma(\mu,\nu) \, : \, \begin{array}{ll}
		\int \ov\zeta(x,y) \, \gamma_x(dy) = x & \ \ \text{$\mu$-a.e.}  \\  
		\int \ov\zeta(x,y) \, \gamma_y(dx) = y & \ \ \text{$\nu$-a.e.}
	\end{array} \right\}.
\end{equation}
Theorem \ref{thm:RK_reinvented} states that solutions of $(\mathrm{M^2OT})$ for the costs $\cost(x,y,z) = \frac{1}{2}\big(\abs{z-x}^2+\abs{z-y}^2\big)$ or $\cost(x,y,z) = \abs{z}^2$ are exactly the pairs $\gamma \in \ov\Gamma(\mu,\nu)$, $q =\ov\zeta \gamma$. In particular, the set $\ov\Gamma(\mu,\nu)$ does not depend on the choice of the optimal potential $\ov{u}$, which may be non-unique. Once the marginals are in convex order, the set consists exactly of martingale plans, cf. Proposition \ref{prop:cvx_order_char} in the text,
\begin{equation}
	\label{eq:Gammabar_cvxorder}
	\mu \preceq_c \nu \qquad \Leftrightarrow \qquad  \ov\Gamma(\mu,\nu) = \Gamma_{\mathrm{M}}(\mu,\nu).
\end{equation}
We readily see that the equivalence (i) $\Leftrightarrow$ (ii) propagates to unordered data  the famous fact\footnote{For any $\gamma \in \Gamma_{\mathrm{M}}(\mu,\nu)$ the total quadratic cost equals $\iint \abs{x-y}^2 d\gamma = m_2(\nu) - m_2(\mu)$, and so it is invariant of the joint distribution $\gamma$.} that $(\mathrm{MOT})$ is trivial for the quadratic cost $c(x,y) = \abs{x-y}^2$, cf. the introduction of \cite{beiglbock2016}.
\smallskip

The equivalence of conditions (iii) $\Leftrightarrow$ (iv) $\Leftrightarrow$ (v) is a consequence of a new result that can be deemed of independent interest:
\begin{proposition}
	\label{prop:Zol_proj}
	Assume that $\mu,\nu \in\mathcal{P}_2(\Rd)$ share  the barycentre. Then,
	\begin{align*}
		\Pi_{\succeq_c \mu} (\nu)   = 	\Pi_{\succeq_c \nu} (\mu)  =   \argmin_{ \rho \succeq_c \mu, \  \rho \succeq_c \nu} m_2(\rho).
	\end{align*}
\end{proposition}
\noindent We see that the Zolotarev projection, either of $\nu$ onto the cone of dominants of $\mu$ or \textit{vice versa}, equals the common dominant of $\mu, \nu$ with the least second moment. Notice the emerging symmetry: while enforcing e.g. only that $\rho \succeq_c \mu$, we obtain the second dominance $\rho \succeq_c \nu$ as a result of optimization. As a by-product, we get the formulas for the normalized distance to the projections,
\begin{align}
	\label{eq:proj_alpha}
	\frac{\min_{\rho \succeq_c \mu} Z_2(\nu,\rho) }{	Z_2(\mu,\nu) }  = \frac{1- 		\alpha_{\succeq_c}(\nu\,|\,\mu)  }{2}, \qquad 
	\frac{\min_{\rho \succeq_c \nu} Z_2(\mu,\rho) }{	Z_2(\mu,\nu) }  = \frac{1+		\alpha_{\succeq_c}(\nu\,|\,\mu)  }{2},
\end{align}
where for distinct measures $\mu \neq \nu$ we have introduced the \textit{convex-order index},
\begin{align*}
		\alpha_{\succeq_c}(\nu\,|\,\mu) :=  \frac{m_2(\nu)-m_2(\mu)}{2 Z_2(\mu,\nu)} \ \in \ [-1,1].
\end{align*}
From the two foregoing equalities, it is clear that $	\alpha_{\succeq_c}(\nu\,|\,\mu) = 1$ when $\mu \preceq_c \nu$, and  $	\alpha_{\succeq_c}(\nu\,|\,\mu) = -1$ when $\mu \succeq_c \nu$. The intermediate index values convey an information on how much the convex order, in one direction or the other, is violated in the sense of the $Z_2$ metric. Observe that this information is essentially carried through the distance $Z_2(\mu,\nu)$ alone.

%By utilizing some of the elementary features of the metric $Z_2$ established in this paper 
%\begin{equation}
%	\mu_0 \preceq_c \nu_0 \qquad \Rightarrow \qquad \alpha_{\succeq_c}(\nu \, | \, \mu) \geq 1-  \frac{2Z_2(\mu,\mu_0) + 2Z_2(\nu,\nu_0)}{Z_2(\mu,\nu)}
%\end{equation}

\begin{remark}
		It must be emphasised that  Proposition \ref{prop:Zol_proj} asserts an equality between three sets which are not singletons in general, cf. Example \ref{ex:non-uniqueness}. Therefore, the term 'Zolotarev projection' is a slight abuse of language. Note that a similar terminology is used in \cite{alfonsi2020sampling,kim2024} for the forward Wasserstein projection, which in general is  not uniquely defined either.
\end{remark}

\subsection{A robust approximation of the classical martingale transport }
\label{ssec:MOT_app}

As observed above, $(\mathrm{MOT})$ can be reset through $(\mathrm{M^2OT})$ with the cost (A). However, it is a theoretical observation that is of no direct aid when tackling $(\mathrm{MOT})$, for instance numerically. For one of  the main contributions of this work we shall propose a bi-martingale approximation of $(\mathrm{MOT})$.

In order that such an approximation facilitates a robust computational method, it must be stable with respect to the two marginals. Assume that we start with marginals in convex order, $\mu \preceq_c \nu$, which are then approximated by convergent sequences $\mu_n \rightharpoonup\mu, \nu_n \rightharpoonup \nu$, e.g. as a result of sampling or deterministic discretization. Two scenarios may occur:
\begin{enumerate}[label=\arabic*)]
	\item  either $\mu_n \preceq_c \nu_n$ for every $n$ (e.g. thanks to the Wasserstein or Zolotarev projection), 
	\item   or the convex order is not preserved at the approximative level. 
\end{enumerate}
In the first case, assume that $\gamma_n \in \Gamma(\mu_n,\nu_n)$ is the exact solution of the well-posed $(\mathrm{MOT})$. Under relevant continuity assumptions on the cost $c$, one expects that every accumulation point $\gamma$ of the sequence $\gamma_n$ is optimal for $(\mathrm{MOT})$,  and that the total transport cost converges as well. Results of this kind are well established for the classical optimal transport $(\mathrm{OT})$, but for its martingale variant stability was proven in 1D only  \cite{juillet2016stability,backhoff2022stability,wiesel2023continuity}. Meanwhile, stability has been refuted for multiple dimensions in \cite{bruckerhoff2022instability} with a simple  counter-example.

It is thus clear that even in the scenario 1)  an approximate $(\mathrm{MOT})$ formulation is the only possible path to robustness, whist scenario 2) leaves no choice. Either way, the martingale condition must be  induced approximatively. The paper \cite{guo2019} addresses this challenge by  proposing the integral  constraint $\int \big|[\gamma_x] -x\big| \mu(dx) \leq \eps$. Through a careful control of $\eps=\eps_n$ the authors have established the stability result under the assumption that $c$ is Lipschitz continuous.

Our bi-martingale approximation of $(\mathrm{MOT})$ will serve as an alternative for the proposal of \cite{guo2019}. It will involve a penalty term: for $\eps >0$ the three-point cost function will read,
\begin{equation*}
	\cost(x,y,z) = c(x,y) + \frac{1}{2\eps} \abs{z}^2.
\end{equation*}
Under suitable conditions, the martingale optimal transport problem $(\mathrm{MOT})$ will be identified as the $\Gamma$-limit of the following sequence of bi-martingale problems,
\begin{align*}
	\inf\left\{ \iint_{(\Rd)^2} \bigg( c (x,y) + \frac{1}{2\eps_n} \Big(\Big|\frac{dq}{d\gamma}(x,y)\Big|^2 - \mathcal{C}(\mu_n,\nu_n) \Big) \bigg) \gamma(dxdy) \, : \, 		(\gamma,q) \in \Gamma Q(\mu_n,\nu_n)    \right\}  \\ (\mathrm{M^2 OT}_n)
\end{align*}
The constant $\mathcal{C}(\mu_n,\nu_n)$, see \eqref{eq:minDom}, affects the functional  through a vertical shift. Hence, it is not required for solving the problem -- it is merely added for the purpose of $\Gamma$-convergence analysis.
\begin{theorem}
	\label{thm:approx_MOT_intro}
	With a generic but fixed barycentre $\ba \in \Rd$, assume:
	\begin{itemize}
		\item a continuous cost ${c} \in C(\Rd \times \Rd)$ that satisfies the following growth condition: for  $p \in [0,2)$ and $D_1,D_2>0$,
			\begin{equation}
			\label{eq:quad_bound_cost}
			-D_1 \leq c(x,y) \leq D_2 (1 + \abs{x}^p + \abs{y}^p);
		\end{equation}
		\item probabilities $\mu,\nu \in \mathcal{P}^\ba_2(\Rd)$ in convex order $\mu \preceq_c \nu$;
		\item sequences $\{\mu_{n}\}, \{\nu_{n}\} \subset \mathcal{P}^\ba_2(\Rd)$ (not necessarily in convex order) that converge to, respectively, $\mu, \nu$ in the Zolotarev-2 metric  (equivalently, in the Wasserstein-2 metric);
		\item a sequence $\eps_n >0$ that converges to zero slowly enough so that,
		\begin{equation*}
			\lim_{n \to +\infty} \frac{Z_2(\mu,\mu_{n})  + Z_2(\nu,\nu_{n})}{\eps_n}=0.
		\end{equation*}
	\end{itemize}
	Then, the problems $(\mathrm{M^2 OT}_n)$ are well-posed and they $\Gamma$-converge\footnote{More accurately, the $\Gamma$-limit is the bi-martingale problem $(\mathrm{M^2 OT})$ with the cost (A), which is equivalent to $(\mathrm{MOT})$.} to $(\mathrm{MOT})$. Accordingly,  for any $(\gamma_n,q_n)$ being a sequence of minimizers for  $(\mathrm{M^2 OT}_n)$,  up to selecting a subsequence there holds,
	\begin{equation*}
		\gamma_n \rightharpoonup \gamma \in \Gamma_{\mathrm{M}}(\mu,\nu),
	\end{equation*}
	where $\gamma$  solves $(\mathrm{MOT})$. Moreover, 
	\begin{equation*}
		\lim_{n \to +\infty} \min\, (\mathrm{M^2 OT}_n) = \min\, (\mathrm{MOT}).
	\end{equation*}
\end{theorem}

\begin{remark}
	Comparison to the similar result \cite[Theorem 2.2]{guo2019} is in order:
	\begin{itemize}
		\item[(a)] The contribution \cite{guo2019} assumes Lipschitz continuity of the cost function $c$, whilst here only the growth conditions are stipulated apart from  continuity.
		\item[(b)] Our result has an extension to the limit data $\mu,\nu$ which are not in convex order, see below.
		\item[(c)]  The paper \cite{guo2019} covers the case of multiple marginals, whilst our strategy is currently confined to one-step martingale couplings only.
		\item[(d)] Our formulation requires that the approximations $\mu_n, \nu_n$ have the same barycentre $\ba =[\mu]=[\nu]$, which may be not automatically met in practice.  However, this can be easily remedied in post-processing through translation of the measures, see also \cite{alfonsi2020sampling} where this step was added to enhance computational stability.
	\end{itemize}
\end{remark}

	It is notable that the problems $(\mathrm{M^2 OT}_n)$ are symmetric with respect to the marginals $\mu,\nu$ (up to a potential dissymmetry of the cost $c$), clearly not being the case for $(\mathrm{MOT})$. In fact, the $\Gamma$-convergence  result put forth in Section \ref{sec:MOT} does not require that the marginals $\mu,\nu$ are in convex order. Assuming just a match of the barycentres $[\mu] = [\nu] = \ba$, the more general version of Theorem \ref{thm:approx_MOT_intro} follows.
	
	\begin{proposition}
		Assume the prerequisites of Theorem \ref{thm:approx_MOT_intro}, except that $\mu,\nu$ may be any probability measures in $\mathcal{P}_2^\ba(\Rd)$ (sharing the barycentre, but not necessarily in convex order). Then, up to selecting a subsequence, $\gamma_n \rightharpoonup \gamma$, where $\gamma$ is a solution of the following problem,
		\begin{align}
			\label{eq:Gamma_limit}
			\min\left\{ \int_X \! \int_Y c (x,y)\, \gamma(dxdy) \, : \, 	\gamma \in \ov\Gamma(\mu,\nu)   \right\}.
		\end{align}
	\end{proposition}

\textit{A priori}, this problem is symmetric with respect to $\mu,\nu$ as well. Only when $\mu \preceq_c \nu$, the equivalence \eqref{eq:Gammabar_cvxorder} revives the classical martingale condition $\gamma \in \Gamma_{\mathrm{M}}(\mu,\nu)$, leading to the statement of Theorem \ref{thm:approx_MOT_intro}.  The aforementioned symmetry manifests itself when the opposite order $\mu \succeq_c \nu$ holds true: the admissible set in \eqref{eq:Gamma_limit}  then encodes the reverse martingale condition, $\ov\Gamma(\mu,\nu) = (\pi_2,\pi_1) \# \Gamma_{\mathrm{M}}(\nu,\mu)$.

In case when the data are not ordered, the optimal transport formulation \eqref{eq:Gamma_limit} calls for an interpretation. One expects that, if $\nu$ is 'close' to dominating $\mu$ for  convex order, then the plans $\gamma$ in the set $\ov\Gamma(\mu,\nu)$ should be 'close' to satisfying the classical martingale condition. There are several possible ways of phrasing this rigorously. The one we choose below neatly resounds the strategy of \cite{alfonsi2020sampling}, with the Wasserstein projection replaced by the Zolotarev projection.

\begin{proposition}
	\label{prop:discrepancy}
	For every plan $\gamma \in \ov\Gamma(\mu,\nu)$ there exists a Zolotarev projection $\ov\rho \in \Pi_{\succeq_c \mu}(\nu)$ and a martingale plan $\breve\gamma  \in \Gamma_{\mathrm{M}}(\mu,\ov\rho)$ such that,
	\begin{equation*}
		\int_X W_2^2(\gamma_x,\breve\gamma_{x}) \,\mu(dx)  =  \big(1- \alpha_{\succeq_c}(\nu\,|\,\mu) \big) Z_2(\mu,\nu),
	\end{equation*}
	where $\gamma_x,\breve\gamma_{x}$ are the relevant transport kernels with respect to the marginal $\mu$. Moreover, the Zolotarev projection and the martingale plan can be computed as follows: $\breve\gamma = (\pi_1,\ov\zeta) \# \gamma$ and $\ov{\rho} = \ov\zeta \# \gamma$, where $\ov\zeta$ is expressed by \eqref{eq:ovzeta} for any optimal potential $\ov{u}$.
\end{proposition}

%The proof  the fact that $\ov\zeta(x,\argu):\Rd \to \Rd$ is a 1-Lipschitz cyclically monotone map for every $x \in \Rd$. Since the same can be said about $\ov\zeta(\argu,y)$, a symmetric result can be established for the transport kernel with respect to the marginal $\nu$, with  $ \big(1+ \alpha_{\succeq_c}(\nu\,|\,\mu) \big) Z_2(\mu,\nu)$ on the right hand side of the equality. Added together, the two equalities furnish a neat formula for the Zolotarev distance,
%\begin{equation}
%	Z_2(\mu,\nu) = \int_X \frac{1}{2} W_2^2(\gamma_x,\gamma^1_x) \,\mu(dx) + \int_Y\frac{1}{2} W_2^2(\gamma_y,\gamma^2_y) \,\nu(dy) \qquad \quad \forall\, \gamma \in \ov\Gamma(\mu,\nu),
%\end{equation}
%where the kernels form  the following disintegrations: $\gamma(dxdy) = \mu(dx) \otimes \gamma_x(dy)= \gamma_y(dx) \otimes \nu(dy)$ and $\gamma_1(dxdz) = \mu(dx) \otimes \gamma_x^1(dz)$,  $\gamma_2(dydz) = \nu(dy) \otimes \gamma_y^2(dz)$ for the induced martingale plans $\gamma_i = (\pi_i,\ov\zeta) \# \gamma$.
%

\subsection{Other results and organization of the paper}

After listing the notation in Section \ref{sec:notation}, we shall devote Section \ref{sec:M2OT} to the well-posedness of the convex minimization problem $(\mathrm{M^2OT})$ for a fairly wide class of costs $\cost(x,y,z)$. In the same section we will shortly present its reformulation to the language of 3-plans $\varpi \in \mathcal{P}(X\times Y\times \Rd)$, employed in \cite{bolbotowski2024kantorovich} for the quadratic cost.

The equivalence (i) $\Leftrightarrow$ (iii) in Theorem \ref{thm:RK_reinvented} is not exclusive to the quadratic case. It can be extended to optimal convex dominance problems for any convex superlinear cost $f:\Rd \to \R \cup \{+\infty\}$,
\begin{equation*}
	\label{eq:conv_dom_intro}
	\inf\left\{ \int_\Rd f(z) \,\rho(dz) \, : \, \rho \in \mathcal{P}_1(\Rd), \ \ \rho \succeq_c \mu, \ \ \rho \succeq_c \nu \right\}.
\end{equation*}
Section \ref{sec:conv_dom} will be devoted entirely to this problem.
Proposition \ref{prop:rhozetagamma} will state that $(\gamma,q)=(\gamma,\zeta\gamma) \in \Gamma Q(\mu,\nu)$ solves $(\mathrm{M^2OT})$ for the convex cost $c(x,y,z) = f(z)$ if and only if $\rho = \zeta \# \gamma$ is a minimizer above. This result will be preceded by the short Subsection \ref{ssec:1D} on the optimal convex dominance on the real line, where an explicit solution $\rho = \mu \vee \nu$ that is universal for all costs $f$ can be constructed.

Section \ref{sec:Zol} will focus on the quadratic case of $(\mathrm{M^2OT})$, that is the case of the equivalent costs $\cost(x,y,z) = \abs{z}^2$ and $\cost(x,y,z) = \frac{1}{2}\big(\abs{z-x}^2 + \abs{z-y}^2\big)$. In particular, it will be devoted to proving Theorem \ref{thm:RK_reinvented}, which is split throughout the section into several results: Corollary \ref{cor:RK_duality}, Proposition \ref{prop:ovGammaQ_char}, Theorem \ref{thm:Zol_proj_again}, but also the already mentioned  Proposition \ref{prop:rhozetagamma} specified for $f = \abs{\argu}^2$.

%The characterization of the convex order case $\mu \preceq_c \nu$, the equality $\ov\Gamma(\mu,\nu) = \Gamma_{\mathrm{M}}(\mu,\nu)$ in particular, is the subject of Proposition~\ref{prop:cvx_order_char}.

Section \ref{sec:MOT} will contain the proofs of the $\Gamma$-convergence results for the bi-martingale approximation $(\mathrm{M^2OT}_n)$ of the $(\mathrm{MOT})$. 
We will conclude the paper with a number of examples in Section \ref{sec:examples}. In particular, Example \ref{ex:stability} demonstrates the stabilization effect of our  approximation $(\mathrm{M^2OT}_n)$ for the sequence of data $\mu_n,\nu_n$ which in  \cite{bruckerhoff2022instability} has showcased instability of the exact $(\mathrm{MOT})$ formulation.

%The simple Example \ref{ex:non-uniqueness}  of an exact solution for the quadratic $(\mathrm{M^2OT})$ is to 
%showcase the possible non-uniqueness of the Zolotarev projection. Example \ref{ex:conv_dom} gives a flavour of the computational potential of the $(\mathrm{M^2OT})$  formulation in the context of the optimal convex dominant problem for the costs $f = \abs{\argu}^p$. Finally, Example \ref{ex:stability} demonstrates the stabilization effect of our  approximation $(\mathrm{M^2OT}_n)$ for the data $\mu_n,\nu_n$ proposed in \cite{bruckerhoff2022instability} to prove instability of the exact $(\mathrm{MOT})$ formulation.

\subsection{Additional remarks on the similarities with the weak optimal transport}

A non-linear variant of optimal transport different than $(\mathrm{M^2OT})$ has attracted considerable attention in the recent years.
Consider a cost function $\breve{c} :X \times \mathcal{P}_1(\Rd) \to \R \cup \{+\infty\}$ that is convex in the second argument. A formulation that today is known under the name of \textit{weak optimal transport} has been proposed in \cite{gozlan2017} and later in \cite{alibert2019},
\begin{align*}
	\qquad \qquad \qquad \qquad
	\inf\left\{ \int_X  \breve{c} (x,\gamma_x)\, \mu(dx) \, : \, 	\gamma \in \Gamma(\mu,\nu)   \right\} \qquad \qquad(\mathrm{WOT}).
\end{align*}
The primary difference between the formulations $(\mathrm{M^2OT})$ and $(\mathrm{WOT})$ is that the latter is inherently dissymmetric. The similarities are more apparent if the class of barycentric costs is considered in  $(\mathrm{WOT})$, i.e. costs $\breve{c}(x,p)$ that involve $[p]$.
Then, there is a resemblance between the roles played by $[\gamma_x]$  and $\frac{dq}{d\gamma}(x,y)$ in the respective problems. For instance,  if  $\breve{c}(x,p) = \int c(x,y) \,p(dy) +  \chi_{\{x\}} ([p])$, then $(\mathrm{WOT})$ problem becomes the classical martingale optimal transport problem, in a similar fashion to $(\mathrm{M^2OT})$ with the cost (A).
The most meaningful parallels can be discerned for the quadratic barycentric cost, $\breve{c}(x,p) = \abs{p-x}^2$, as we shall now outline.

An intimate connection between the backward Wasserstein projection and the quadratic barycentric variant of $(\mathrm{WOT})$ has been reported  independently in \cite{alfonsi2020sampling} and \cite{gozlan2020},
\begin{equation*}
	\min\limits_{\rho \preceq_c \nu} 	W^2_2(\mu,\rho) = \min_{\gamma \in \Gamma(\mu,\nu)} \int_X \big|[\gamma_x] - x\big|^2 \mu(dx).
\end{equation*}
If $\gamma$ solves the barycentric  $(\mathrm{WOT})$ as above, then the backward Wasserstein projection can be recast as  $\rho = T \# \mu$ for the transport map $T(x) = [\gamma_x]$. Despite the uniqueness of $\rho$, in general there is a large freedom in the choice of the optimal plan $\gamma$, cf. the characterization  \cite[Theorem 1.2]{gozlan2020}. The same result recognizes $T$ as the Brenier map $\nabla \varphi$ for a convex function $\varphi$ that is $C^{1,1}$.

Based on Theorem \ref{thm:RK_reinvented} in this paper, we can deduce links of similar flavour for the forward Zolotarev projection and the quadratic bi-martingale optimal transport,
\begin{equation*}
	\min_{\rho \succeq_c \mu} Z_2(\nu,\rho) = \min_{  	(\gamma,q) \in \Gamma Q(\mu,\nu) } \int_X \! \int_Y \frac{1}{2} \Big| \frac{dq}{d\gamma}(x,y) - y \Big|^2 \gamma(dxdy).
\end{equation*}
The $(\mathrm{M^2OT})$ problem above shares solution $(\ov\gamma,\ov{q}) = (\ov{\gamma},\ov\zeta \ov\gamma)$ with the two problems in Theorem \ref{thm:RK_reinvented}. Accordingly, $\ov\rho = \ov\zeta \# \ov\gamma$ is a forward Zolotarev projection, where the coupling map $\ov\zeta$ is expressed through the gradient of a $C^{1,1}$ potential $\ov{u}$, see \eqref{eq:ovzeta}. Moreover, the set of optimal plans $\ov\Gamma(\mu,\nu)$ offers a lot of flexibility. In particular, for ordered data $\mu \preceq_c \nu$, solutions $\ov\gamma$ are exactly martingale plans, which is also the case for the barycentric  $(\mathrm{WOT})$. The analogies with the results of \cite{gozlan2020} become even more crisp upon analysing the proof of Proposition \ref{prop:discrepancy} (Proposition \ref{prop:dispcrepancies} in the text), where the 1-Lipschitz cyclically monotone functions $\ov\zeta(x,\argu)$ serve as the Brenier maps from $\gamma_x$ to $\breve{\gamma}_x$.   

\subsection*{Acknowledgments} A large part of this paper has been  prepared during the author's postdoc at the Lagrange Mathematics and Computing Research Center in Paris. The author is grateful to Guillaume Carlier, Quentin Mérigot, and Filippo Santambrogio for their mentorship during his stay at the Lagrange Center. He is also indebted to  Guy Bouchitté for the collaboration \cite{bolbotowski2024kantorovich,bolbotowski2025} that sparked the topics studied in this paper.

\section{Notation and basic facts}
\label{sec:notation}
\noindent Let us start with the elementary definitions and notation:
\begin{itemize}
	\item $\pairing{\argu,\argu}$ is the canonical scalar product on  $\R^k$, while $\abs{x} = \sqrt{\pairing{x,x}}$ is the Euclidean norm.
	\item $\Mes(\Rd;\R^k)$ is the Banach space of vector-valued bounded Borel measures on $\Rd$.
	$\mathcal{P}(\Rd)$ is the space of probability measures, i.e. measures $\mu \in \Mes(\Rd;\R_+)$ with $\mu(\Rd)=1$.
	\item For $p > 0$ and $\mu \in \mathcal{P}(\Rd)$ we denote $m_p(\mu) = \int_\Rd \abs{x}^p \mu(dx)$, while $\mathcal{P}_p(\Rd)= \big\{ \mu \in \mathcal{P}(\Rd)\, : \, m_p(\mu) < +\infty\big\}$.
	\item For $\mu \in \mathcal{P}_1(\Rd)$ we denote $[\mu] = \int_\Rd x \,\mu(dx)$. The subset of probabilities with a fixed barycentre $\ba \in \Rd$ will be denoted by $\mathcal{P}_p^\ba(\Rd) = \big\{\mu \in \mathcal{P}_p(\Rd) \, : \, [\mu] = \ba \big\}$, where $p\geq 1$.
	\item The absolute continuity of  $\sigma \in \Mes(\Rd;\R^k)$  with respect to  $\mu \in \Mes(\Rd;\R_+)$ will be denoted  by $\sigma \ll \mu$. For such a pair of measures, $\frac{d \sigma}{d\mu}$ will be the Radon--Nikodym derivative. It is an element of  $L^1_\mu(\Rd;\R^k)$, it is an integrable function with respect to $\mu$.
	 \item By $\spt (\sigma)$ we will understand the topological support of a measure. For a closed set $X \subset \R^d$, the space $\Mes(X;\R^k)$ will consist of  measures in $\Mes(\Rd;\R^k)$ whose support is $X$. 
	\item For a Borel measurable map $T:\Rd \to \R^l$ by $T \# \sigma$ we understand the push-forward measure, i.e. $(T \# \sigma)(B) = \sigma( T^{-1}(B))$ for every Borel set $B \subset \R^l$.
	\item  For $m \leq n$, by $\pi_{i_1,\ldots,i_m}: (\Rd)^n \to (\Rd)^m$ we mean the projection map onto the coordinates $i_1 < \ldots < i_m$. The marginal measures of $\sigma \in \Mes((\Rd)^n;\R^k)$ will be denoted by $\pi_{i_1,\ldots,i_m} \# \sigma$.
	\item By $C(\Rd;\R^k)$ and $C_b(\Rd;\R^k)$ we understand the space of continuous and, respectively, bounded continuous vector functions; $C(\Rd) = C(\Rd;\R)$ for short. $C^{r,1}(\Rd)$ is the space of functions $u$ whose $r$-th derivative is Lipschitz continuous, i.e. $\mathrm{lip}(D^r u) < +\infty$.
	
	\item We say that a sequence of measures $\sigma_n \in \Mes(\Rd;\R^k)$ converges weakly to $\sigma$, denoted by $\sigma_n \rightharpoonup \sigma$, if $\int_\Rd \pairing{\Phi,d\sigma_n} \to \int_\Rd \pairing{\Phi,d\sigma}$ for any  $\Phi \in C_b(\Rd;\R^k)$.
	\item We say that a family of measures $S \subset \Mes(\R^d;\R^k)$ is tight if for every $\eps >0$ there exists a compact set $K \subset \Rd$ such that $\sup_{\sigma \in S} \abs{\sigma}(\Rd \backslash K) < \eps$, where $\abs{\sigma} \in \Mes(\Rd;\R_+)$ is the total variation measure of $\sigma$ (with respect to the Euclidean norm).
	\item For two probability measures $\mu,\nu \in \mathcal{P}_1(\Rd)$, the convex order is denoted by $\mu \preceq_c \nu$.
	\item For  $\mu \in \mathcal{P}(\R^n)$ consider the Borel measurable \textit{kernel} $\R^n \ni x \mapsto \lambda_x \in \mathcal{P}(\R^m)$, i.e. $x \mapsto \lambda_x(B)$ is Borel measurable for every Borel set $B \subset \R^m$. By  the generalized product $\gamma(dxdy) = \mu(dx) \otimes \lambda_x(dy)$ we understand the probability $\gamma \in \mathcal{P}(\R^n \times \R^m)$ such that,
	\begin{equation*}
		\iint_{\R^n \times \R^m} \varphi(x,y) \, \gamma(dxdy) = \int_{\R^n} \left( \int_{\R^m} \varphi(x,y)\, \lambda_x(dy) \right) \mu(dx) \qquad \forall\, \varphi \in C_b(\R^n \times \R^m).
	\end{equation*}
	Conversely,  any $\gamma \in \mathcal{P}(\R^n \times \R^m)$ can be disintegrated into $\gamma(dxdy) = \mu(dx) \otimes \gamma_x(dy)$ where the kernel $\gamma_x$ is uniquely determined for $\mu$-a.e. $x$, see \cite[Theorem 2.28]{ambrosio-fusco}.
	\item For $\mu, \nu \in \mathcal{P}(\Rd)$ whose supports are $X,Y$, respectively, the set of transport plans is defined as $\Gamma(\mu,\nu) = \big\{\gamma \in \mathcal{P}(X \times Y) \, : \, \pi_1 \# \gamma = \mu, \ \pi_2 \# \gamma = \nu \big\}$.
	\item If $\mu,\nu \in \mathcal{P}_1(\Rd)$, the set of martingale plans will be denoted by $\Gamma_{\mathrm{M}}(\mu,\nu) = \big\{ \gamma \in \Gamma(\mu,\nu) \, : \, [\gamma_x] = x \ \text{for $\mu$-a.e. $x$} \big\}$. A plan $\gamma \in \Gamma(\mu,\nu)$ belongs to $\Gamma_{\mathrm{M}}(\mu,\nu) $ if and only if for every bounded Borel function $\Phi:\Rd\to \Rd$,
	\begin{equation}
		\label{eq:classical_mart_cond}
		\int_X\! \int_Y \pairing{\Phi(x),y-x} \, \gamma(dxdy) =0.
	\end{equation}
\end{itemize}
\smallskip

If the barycentres of $\mu, \nu$ agree, i.e. $\mu,\nu \in \mathcal{P}_1^\ba(\Rd)$ for some $\ba \in \Rd$, then, in addition to $\Gamma(\mu,\nu)$, we will deal with the set of vector-valued couplings $Q(\mu,\nu)=\big\{ q \in\Mes(X \times Y;\Rd ) \,:\, \pi_1 \# q \nolinebreak = x\mu, \ \ \pi_2 \# q= y \nu  \big\}$. In turn, the set $\Gamma Q(\mu,\nu)$ consists of $(\gamma,q) \in \Gamma(\mu,\nu) \times Q(\mu,\nu)$ for which $q \ll \gamma$.

Assume a pair $(\gamma,q) \in \mathcal{P}(X \times Y) \times \Mes(X\times Y;\Rd)$ such that $q \ll \gamma$ and denote $\zeta = \frac{dq}{d\gamma} \in L^1_\gamma(X \times Y;\Rd)$. Then $(\gamma,q)$ is an element of $\Gamma(\mu,\nu) \times Q(\mu,\nu)$ if and only if,
\begin{equation}
	\label{eq:moment-martingale}
		\begin{cases}
		\iint \varphi(x)\, \gamma(dxdy)  = \int \varphi \,d\mu,  \\
		\iint \psi(y)\, \gamma(dxdy)  = \int \psi \,d\nu, 
	\end{cases}\quad
	\begin{cases}
		\iint\pairing{\Phi(x),\zeta(x,y)-x} \, \gamma(dxdy)  = 0,  \\
		\iint \pairing{\Psi(y),\zeta(x,y)-y} \, \gamma(dxdy)  = 0, 
	\end{cases}
\end{equation}
for any bounded Borel functions $\varphi,\psi:\Rd \to \R$ and $\Phi,\Psi:\Rd \to \Rd$. If $\mu,\nu \in \mathcal{P}_p(\Rd)$, then we can test with unbounded $\varphi,\psi$ satisfying $p$-th growth condition, e.g. $\abs{\varphi(x)} \leq C( 1 +\abs{x}^p)$.
If, moreover, $\iint \abs{\zeta}^p d \gamma < +\infty$ for $p>1$, then unbounded functions  $\Phi,\Psi$ of  growth $p' = \frac{p}{p-1}$ can be used. 

In the text, we will make use of the following fact which is an easy consequence of Strassen theorem and Jensen inequality.
If $\varphi:\Rd \to \R \cup \{+\infty\}$ is a strictly convex function, then for any ordered pair $\mu \preceq_c \nu$ such that $\int \varphi \, d\mu < + \infty$, the  equivalence below holds true,
\begin{equation}
	\label{eq:strict}
	\int_\Rd \varphi \,d\mu  = \int_\Rd \varphi \,d\nu  \qquad \Leftrightarrow \qquad  \mu = \nu. 
\end{equation}

The admissibility condition in the Zolotarev problem \eqref{eq:Z2} defining $Z_2(\mu,\nu)$ admits several equivalent forms. Indeed, it is a classical result that for a continuous function $u \in C(\Rd)$ we have,
\begin{equation}
	\label{eq:u_adm_char}
	u \in C^{1,1}(\Rd),\ \ \mathrm{lip}(\nabla u) \leq 1  \ \ \ \Leftrightarrow \ \ \ -\mathrm{Id} \preceq \nabla^2 u \preceq \mathrm{Id}
	  \ \ \  \Leftrightarrow \ \ \ \tfrac{1}{2}\abs{\argu}^2 \pm u  \ \  \text{are convex}.
\end{equation}
The second condition can be understood in the sense of distributions. Yet another equivalent condition was devised in \cite{legruyer2009}: $u$ is differentiable and,
\begin{align}
	\label{eq:3-point_ineq}
	\big({u}(x)+ \pairing{\nabla {u}(x),z-x}\big) -  \big({u}(y)  +  \pairing{\nabla {u}(y),z-y} \big) \leq \frac{1}{2} \Big( \abs{z-x}^2 + \abs{z-y}^2\Big) 
\end{align}
for every triple $(x,y,z) \in(\Rd)^3$, see \cite[Lemma 3.4]{bolbotowski2024kantorovich} for an alternative proof.

\section{Bi-martingale formulation for a general cost}
\label{sec:M2OT}

Throughout this section $\mu$ and $\nu$ will be probability measures on $\Rd$ with finite first order moments and  a common barycentre $\ba \in \Rd$, i.e. $\mu, \nu \in \mathcal{P}^\ba_1(\Rd)$. Their supports will be denoted by $X$ and $Y$.
 The cost function $\cost$ will be assumed to satisfy the following conditions: 
\begin{enumerate}
	\item[(H1)] $\cost$ is jointly lower semi-continuous on $X \times Y \times \Rd$ and bounded from below;
	\item[(H2)]  $z \mapsto \cost(x,y, z)$ is convex for every $(x,y) \in X \times Y$;
	\item[(H3)] $z \mapsto \cost(x,y, z)$ is equi-superlinear on $X \times Y$; namely,
	there exists a convex and superlinear function $h:\Rd \to \R$ such that for all $(x,y,z) \in X \times Y \times \Rd$,
	\begin{equation*}
		\cost(x,y,z) \geq  h(z).
	\end{equation*}
\end{enumerate}
The superlinerity assumption is crucial to ensure that there is a solution $(\gamma,q)$ of $(\mathrm{M^2 OT})$ for which $q \ll \gamma$, see \cite[Example 2.36]{ambrosio-fusco}. Only then the bi-martingale perspective on the couple $(\gamma,q)$ is valid. However, when the supports $X,Y$ are not compact, the condition (H3) fails for some important costs like (A) or (C) in Section \ref{ssec:bimatringale_intro}. To address that, we could instead work with a  weakened version of this condition, to which the proofs in this section could be easily adapted:
\begin{enumerate}
	\item[(H3')]  there exists a convex superlinear function $\tilde{h}$ integrable with respect to $\mu$ and $\nu$ along with scalars $\alpha,\beta \in \R$ such that for every $(x,y,z) \in X \times Y \in \Rd$ there holds,
	\begin{equation*}
		\cost(x,y,z) \geq \tilde{h}(z - \alpha x - \beta y).
	\end{equation*}
\end{enumerate}

\subsection{Existence of optimal bi-martingale plans}
\label{ssec:existence}

In this subsection we employ the direct method of the calculus of variations to show the well-posedness $(\mathrm{M^2 OT})$. Let us start by establishing the lower semi-continuity of the cost functional.

\begin{proposition}
	\label{prop:lsc}
	If  $\cost$ satisfies the assumptions (H1)-(H3), then the functional,
	\begin{align*}
		\Mes(X\times Y ; \R \times \Rd) \ni (\gamma,q) \ \mapsto \ \begin{cases}
			\int_X \! \int_Y \cost  \big(x,y,\frac{d q}{d\gamma}(x,y) \big) \gamma(dxdy) & \text{if \ \  $\gamma \geq 0$ and  $q \ll \gamma$,} \\
			+ \infty & \text{otherwise},
		\end{cases}
	\end{align*}
	is convex lower semi-continuous with respect to the topology of weak convergence.
\end{proposition}
\begin{proof}
	For every $(x,y) \in X \times Y$ and  $(t,w) \in \R \times \Rd$ define,
	\begin{equation*}
			g(x,y;t,w) := \begin{cases}
			\cost\big(x,y,\tfrac{w}{t}\big)\,t & \text{if $t>0$},\\
			0 & \text{if $t=0$, $w=0$}, \\
			+\infty & \text{if $t=0$, $w \neq 0$, or $t<0$.}
		\end{cases}
	\end{equation*}
	The function $(t,w) \mapsto g(x,y;t,w)$ is  positively 1-homogeneous and convex\footnote{Convexity and positive 1-homogeneity  is clear once $g(x,y;\argu,\argu)$ is identified as the Minkowski functional of the closed convex  set $\big\{(s,v) \in \R \times \Rd \, : \, s + \cost^*(x,y,v) \leq 0 \big\}$, where $\cost^*(x,y,\argu) = \big(\cost(x,y,\argu)\big)^*$ is the Fenchel-Legandre conjugate of $\cost$ computed with respect to the last variable.} for every $x,y$ owing to (H2). If $\vartheta$ is any finite positive Borel measure on $X \times Y$ such that $\gamma \ll \vartheta$, $q \ll \vartheta$, then the functional in the assertion can be written as,
	\begin{equation*}
		\int_X \! \int_Y g\Big(x,y; \,\frac{d \gamma}{d\vartheta}(x,y),\frac{d q}{d\vartheta}(x,y)\Big) \vartheta(dxdy).
	\end{equation*}
	The value of the integral does not depend on the choice of $\vartheta$ thanks to the 1-homogeneity. If $q = \zeta \gamma + q_s$, where $q_s \perp \gamma$, then choosing $\vartheta = \gamma + \abs{q_s}$ allows to see the equality between the functionals.
	
	Readily, the lower semi-continuity of the functional will follow from the Reshetnyak theorem (cf. \cite[Theorem 2.38]{ambrosio-fusco}) if we prove that that $g$ is lower semi-continuous jointly in all four variables. To that end, take a sequence $(x_n,y_n,t_n,w_n)$ converging to $(x_0,y_0,t_0,w_0)$. If $t_0 \neq 0$, then we have $w_n/t_n \to w_0/t_0$, and so the lower semi-continuity is a result of (H1). Once $t_0 = 0$ and $w_0 \neq 0$, we have $\liminf_{n} g(x_n,y_n;t_n,w_n) = +\infty$ by the superlinearity assumption (H3). The case when $t_0 <0$ gives the same $\liminf$. Finally, for $t_0=0$, $w_0 =0$, from (H1) we know that $\cost(x_n,y_n,w_n/t_n) \geq C$ for a real constant $C$, and thus $\liminf_{n} g(x_n,y_n;t_n,w_n) \geq 0$ in this case. 
\end{proof}

We must next address the issue of compactness. Firstly, let us recall that the set $\Gamma Q(\mu,\nu)$ defined in \eqref{eq:GammaQ} is non-empty for any probabilities $\mu,\nu \in \mathcal{P}^\ba_1(\Rd)$. Indeed, the pair
\begin{equation*}
	\gamma = \mu\otimes \nu, \qquad q = (x+y-\ba)\, \mu\otimes \nu
\end{equation*}
is universally an element of $\Gamma Q(\mu,\nu)$. Next, it is a classical result in optimal transport that the set of probabilities $\Gamma(\mu,\nu)$ is tight (see Section \ref{sec:notation} for definition), and its compactness follows by the Prokhorov theorem. This result does not extend to the set of vector measures $Q(\mu,\nu)$, which \textit{a priori} is neither tight nor even equi-bounded. We must therefore appeal to the coercivity of the cost functional stemming from (H3).

\begin{proposition}
	\label{prop:compactness}
	Assume probability measures $\mu,\nu \in \mathcal{P}^\ba_1(\Rd)$ and a cost function that satisfies the conditions (H1)-(H3). Then, for any real constant $C_0$, the set
	\begin{equation}
		\label{eq:compact_set}
		\left\{ (\gamma,q) \in \Gamma Q(\mu,\nu) \, : \, \int_X \! \int_Y  \cost  \Big(x,y,\frac{d q}{d\gamma}(x,y) \Big) \gamma(dxdy) \leq C_0 \right\}
	\end{equation}
	is compact for the topology of weak convergence in  $\Mes(\Rd \times \Rd ; \R \times \Rd)$.
\end{proposition}

\begin{corollary}
	\label{cor:existence}
	For $\mu,\nu \in \mathcal{P}^\ba_1(\Rd)$ and $\cost$ satisfying the conditions (H1)-(H3), the problem $(\mathrm{M^2 OT})$ admits a (in general non-unique) solution as long as $\inf(\mathrm{M^2 OT}) < +\infty$.
\end{corollary}

%\begin{remark}
%	Let us briefly comment on how the existence result can be replicated if the assumption (H3) is replaced by its weaker variant (H3'). First, the proof of the lower semi-continuity in Proposition \ref{prop:lsc} uses properties of $\cost$ that are local in $x,y$, and so it holds with (H3').
%	However, to get the coercivity result in Proposition \ref{prop:compactness}, we require stronger properties than superlinearity of $\cost(x,y,\argu)$ that is merely local. Utilizing convexity of the  superlinear function $\tilde{h}$ from the condition (H3'), we arrive at an estimate $\tilde{h}(z-\alpha x - \beta  y) \geq 3 \, \tilde{h}\left(\tfrac{1}{3} z \right) - \tilde{h}(\alpha x)- \tilde{h}(\beta y)$.
%	Then, thanks to the assumed integrability of $\tilde{h}$ with respect to $\mu,\nu$, this inequality can be used to adequately modify the chain of inequalities \eqref{eq:chain_h} in the proof below, so that the coercivity result remains valid.
%\end{remark}

\begin{proof}[Proof of Proposition \ref{prop:compactness}]
	Let us first show the relative compactness which, by the Prokhorov theorem, is equivalent to the tightness of the family of probabilities $\gamma$ and the tightness and equi-boundedness of the family of total variation measures $\abs{q}$, where $(\gamma,q)$ range in the set \eqref{eq:compact_set}; see \cite[Theorem 8.6.2]{bogachev2007} for the variant of the Prokhorov theorem for vector valued measures.

	The tightness of the set $\Gamma(\mu,\nu)$ is classical, see \cite{santambrogio2015,villani}. In fact, since $\mu,\nu$ have finite first order moments, the family $	\big\{ (1 +\abs{x}+\abs{y}) \, \gamma(dxdy) \, : \, \gamma \in \Gamma(\mu,\nu) \big\}$ can be shown to be
 tight as well. Let us fix a pair $(\gamma,q)$ in \eqref{eq:compact_set} and a number $\eps>0$. For the Borel function $\zeta = \frac{dq}{d\mu}$,  we define the Borel set,
	\begin{equation*}
		B_\zeta:= \Big\{ (x,y) \, : \, \abs{\zeta(x,y)} \leq 1 + \abs{x} + \abs{y}  \Big\},
	\end{equation*}
	and next we introduce the decomposition $q = q_1 + q_2$, where $q_1 = q \mres B_\zeta$ and $q_2 = q \mres B^c_\zeta$ are the relevant restrictions. By the tightness there exists a compact set $M_\eps \subset (\Rd)^2$ (independent of $(\gamma,q)$) such that $\iint_{M_\eps^c} (1 +\abs{x}+\abs{y}) \, \gamma(dxdy) < \frac{\eps}{2}$. Then, since $q = \zeta \gamma$,
	\begin{align}
		\label{eq:q1}
		\abs{q_1}(M_\eps^c) = \iint_{M_\eps^c \cap B_\zeta} \abs{\zeta(x,y)} \,\gamma(dxdy) \leq  \iint_{M_\eps^c \cap B_\zeta}  (1 +\abs{x}+\abs{y})\, \gamma(dxdy) \leq \frac{\eps}{2}.
	\end{align}
	
	To handle the part $q_2$ we will use the convex superlinear function $h:\Rd \to \R$ in the assumption (H3). Since $\cost$ is bounded from below thanks to (H1), it is not restrictive to assume that  $C_0>0$. Fix $r = r(\eps) >0$ such that $h(z) \geq \frac{2C_0}{\eps} \abs{z}$ whenever $\abs{z} \geq r(\eps)$. Such a number exists for every  $\eps>0$ by the superlinearity. Define the compact set
	$N_\eps :=  \big\{ (x,y) \, : \,  1 + \abs{x} + \abs{y} \leq r(\eps)  \big\}$ and observe that for every $(x,y) \in N_\eps^c \cap B_\zeta^c$ there holds that $\abs{\zeta(x,y)} > r(\eps)$, and, as a result, $h\big(\zeta(x,y)\big) \geq \frac{2C_0}{\eps} \, \abs{\zeta(x,y)}$. Accordingly, we have,
	\begin{align}
		\nonumber
		C_0 \geq \iint_{B_\zeta^c}  \! \cost\big(x,y,\zeta(x,y)\big)  \gamma(dxdy) \geq  \iint_{B_\zeta^c} \! h\big(\zeta(x,y)\big)  \gamma(dxdy)  \geq \iint_{N_\eps^c \cap B_\zeta^c } \! h\big(\zeta(x,y)\big)  \gamma(dxdy) &\\
		\label{eq:chain_h}
		 \qquad \geq \iint_{N_\eps^c \cap B_\zeta^c } \frac{2C_0}{\eps} \, \abs{\zeta(x,y)} \,  \gamma(dxdy) = \frac{2C_0}{\eps} \iint_{N_\eps^c \cap B_\zeta^c } d\abs{q} =  \frac{2C_0}{\eps} \abs{q_2}(N_\eps^c)&, 
	\end{align}
	which gives  $\abs{q_2}(N_\eps^c) \leq \frac{\eps}{2}$. Ultimately, by orthogonality of $q_1$ and $q_2$ we obtain for the compact set $K_\eps = M_\eps \cup N_\eps$,
	\begin{equation*}
		\abs{q}(K_\eps^c) =\abs{q_1}\big((M_\eps \cup N_\eps)^c\big) + \abs{q_2}\big((M_\eps \cup N_\eps)^c\big) \leq \abs{q_1}( M_\eps^c)+ \abs{q_2}( N_\eps^c) \leq \eps,
	\end{equation*}
	thus showing the tightness for $\abs{q}$.
	
	Let us now pass to the equi-boundedness of $\abs{q} = \abs{q_1}  + \abs{q_2}$. Equiboundedness of  $\abs{q_1}$ is clear by changing the set of integration in \eqref{eq:q1} to $B_\zeta$. For $\abs{q_2}$ we observe that, by the superlinearity and convexity of $h$, there exists a constant $C_1\geq 0$ such that $h(z) \geq  \abs{z} - C_1$. The equi-boundedness of $\abs{q_2}$ then follows from the inequality $C_0 \geq\iint_{B_\zeta^c} h\big(\zeta(x,y)\big)  \gamma(dxdy)$.
	
	We have thus showed relative compactness of the set \eqref{eq:compact_set}, and it remains to show that it is closed for the weak convergence. Note that the set can be rewritten by replacing the functional with the one in Proposition  \ref{prop:lsc} and the set $\Gamma Q(\mu,\nu)$ with $\Gamma(\mu,\nu) \times Q(\mu,\nu)$ (the functional enforces the absolute continuity $q \ll \gamma$). The latter set is closed for the weak convergence, while the lower semi-continuity result in Proposition \ref{prop:lsc} guarantees that the condition $q \ll \gamma$ is preserved at the accumulations points of weakly convergent sequences contained in \eqref{eq:compact_set}. The wished closedness of \eqref{eq:compact_set} thus follows, and the proof is complete.
\end{proof}

\subsection{Reformulation via 3-plans}
\label{ssec:3-plans}
To wrap up the section we present an equivalent formulation to $(\mathrm{M^2OT})$ where, instead of the pair of couplings $(\gamma,q)$, we look for a single probability but on the triple product $X \times Y \times \Rd$. Contrarily to $(\mathrm{M^2OT})$, it is a linear programming problem, and for the quadratic cost it has been put forth in the paper \cite{bolbotowski2024kantorovich}, a foundation for the present study,
\begin{align}
	\label{eq:3plans}
	\inf\left\{ \int_X \! \int_Y \! \int_{\Rd} \cost (x,y,z)\, \varpi(dxdydz) \ : \ 	\varpi \in \Sigma(\mu,\nu) \right\},
\end{align}
where the set of admissible 3-plans encodes the bi-martingale condition directly,
\begin{equation*}
	\Sigma(\mu,\nu) := \Big\{  \varpi \in  \mathcal{P}(X\times Y\times \Rd) \, : \,  \pi_{1,2} \# \varpi \in \Gamma(\mu,\nu), \ \  	\pi_{i,3} \#\varpi \text{ are martingale plans}   \Big\}.
\end{equation*}
This problem is more universal than $(\mathrm{M^2OT})$ as it does not entail convexity of  $\cost$ in the variable $z$.
Nonetheless, it is a mild relaxation of $(\mathrm{M^2OT})$ when $\cost(x,y,\argu)$ is convex, and the two problems are fully equivalent if this convexity is strict. 
\begin{proposition}
	\label{prop:gammaq-omega}
	For the costs $\cost(x,y,z)$ satisfying the assumptions (H1)-(H3),
	the problems $(\mathrm{M^2OT})$ and \eqref{eq:3plans} are equivalent in the following sense: $\inf (\mathrm{M^2OT})= \inf \eqref{eq:3plans}$ and,
	\begin{enumerate}[label={(\roman*)}]
		\item assuming that $(\gamma,q)=(\gamma,\zeta\gamma)$ solves $(\mathrm{M^2OT})$, the following probability solves \eqref{eq:3plans},
		\begin{equation}
			\label{eq:omega-gamma}
			\varpi = (\pi_1,\pi_2,\zeta)\# \gamma = \gamma(dxdy) \otimes \delta_{\zeta(x,y)}(dz);
		\end{equation}
		\item assuming that $\tilde{\varpi}$ solves \eqref{eq:3plans}, define the marginal $\tilde\gamma := \pi_{1,2}\# \tilde\varpi$ and perform the disintegration $\tilde\varpi(dxdydz) = \tilde\gamma(dxdy) \otimes \tilde\varpi_{x,y}(dz)$; then  the following pair solves $(\mathrm{M^2OT})$,
		\begin{equation*}
			(\tilde\gamma,\tilde{q}) = (\tilde\gamma,\tilde\zeta \tilde\gamma ), \qquad \text{where} \qquad \tilde\zeta(x,y) := [\tilde\varpi_{x,y}] = \int_{\Rd} z\, \tilde\varpi_{x,y}(dz).			
		\end{equation*} 
	\end{enumerate}
	Moreover, if $\cost(x,y,\argu)$ is strictly convex for all $(x,y) \in X \times Y$, then every solution of \eqref{eq:3plans} is of the form \eqref{eq:omega-gamma}.
\end{proposition}
\begin{proof}
	Observe that $\pi_{1,2} \#\varpi = \gamma$ and $\pi_{i,3} \#\varpi =(\pi_i,\zeta)\# \gamma$. Therefore, $\varpi$ is admissible for \eqref{eq:3plans} by \eqref{eq:two_martingales}, and
	\begin{equation*}
		\inf \eqref{eq:3plans} \leq  \int_X \! \int_Y \! \int_{\Rd} \cost (x,y,z)\, \varpi(dxdydz) =  \int_X \! \int_Y \cost  \big(x,y,\zeta(x,y) \big) \gamma(dxdy) = \inf(\mathrm{M^2OT}).
	\end{equation*}
	
	To see that $(\tilde{\gamma},\tilde{q})$ is admissible for $(\mathrm{M^2OT})$, we check the conditions in \eqref{eq:moment-martingale}. The first pair of conditions is clear, since  $\tilde{\gamma} = \pi_{1,2} \# \tilde\varpi \in \Gamma(\mu,\nu).$ To check the second pair of conditions we take a function $\Phi \in C_b(\Rd;\Rd)$ and compute,
	\begin{align*}
		 \int_X \! \int_Y \langle\Phi(x), \tilde\zeta(x,y)  -x \rangle \,\tilde{\gamma}(dxdy) 
		= \int_X \! \int_Y \Big\langle\Phi(x),\int_{\Rd} z\, \tilde\varpi_{x,y}(dz)  -x \Big\rangle \,\tilde{\gamma}(dxdy)& \\
		  =\int_X \! \int_Y \! \int_{\Rd}  \pairing{\Phi(x),z-x} \,\tilde\varpi(dxdydz) = \int_X \! \int_{\Rd}  \pairing{\Phi(x),z-x} \,\pi_{1,2} \# \tilde\varpi(dxdz)&,
	\end{align*}
	which is zero by the characterization \eqref{eq:classical_mart_cond} of the martingale condition.
	Handling the last condition in \eqref{eq:moment-martingale} is similar.
	Accordingly, using Jensen's inequality we obtain,
	\begin{align*}
		\inf  (\mathrm{M^2OT}) \leq  \int_X \! \int_Y \cost  \big(x,y,\tilde\zeta(x,y) \big) \tilde\gamma(dxdy) &=   \int_X \! \int_Y \cost  \Big(x,y,\int_{\Rd} z\, \tilde\varpi_{x,y}(dz)\Big) \tilde\gamma(dxdy) \\
		& \leq \int_X \! \int_Y \! \int_{\Rd} \! \cost(x,y,z)\,\tilde{\varpi}(dxdydz) =  \inf \eqref{eq:3plans}.
	\end{align*}
	Combining with the previous inequality we get $\inf  (\mathrm{M^2OT})  =  \inf \eqref{eq:3plans}$ together with the assertions (i),\,(ii). In particular, we obtain that the Jensen inequality above must be an equality. With $\cost(x,y,\argu)$ strictly convex this is possible only if $\tilde{\varpi}_{x,y}$ is a Dirac delta measure, which furnishes the moreover part of the statement. 
\end{proof}

\section{Application to convex dominance with minimum cost}
\label{sec:conv_dom}

In this section $\mu,\nu$ are again probabilities in $\mathcal{P}_1^\ba(\Rd)$, where $\ba \in \Rd$ is a generic barycentre. For $f:\Rd \to \R \cup \{+\infty\}$ we take any convex lower semi-continuous function that is superlinear. The problem of finding a convex dominant that minimizes the $f$-cost reads as follows,
\begin{equation}
	\label{eq:conv_dom}
	\inf\left\{ \int_\Rd f(z) \,\rho(dz) \, : \, \rho \in \mathcal{P}_1(\Rd), \ \ \rho \succeq_c \mu, \ \ \rho \succeq_c \nu \right\}.
\end{equation}
It is an infinite dimensional linear programming problem.
The plan below is to show that minimizers in \eqref{eq:conv_dom} can be found by 
solving the problem $(\mathrm{M^2OT})$ for the cost $\cost(x,y,z) = f(z)$. Beforehand, we will show that, in dimension one, an explicit formula for the optimal dominant can be proposed by appealing to the join-semilattice structure of convex order on the real line.

\subsection{Optimal convex dominance on the real line}
\label{ssec:1D}

In this short subsection we work in dimension one, $d=1$, where optimal dominance problem \eqref{eq:conv_dom} attains a closed-form solution. This stems from the join-semilattice structure of the convex order on the real line \cite{shaked2007,hirsch2012}. Namely, for each pair $\mu,\nu \in \mathcal{P}_1(\R)$ there exists a unique least upper bound $\mu \vee \nu$.

In order to construct $\mu \vee \nu$, for any measure $\mu \in \mathcal{P}_1(\R)$
define,
\begin{equation*}
	\FF_\mu(x) := \int_\R (x-\xi)_+ \mu(d\xi) \qquad \forall\, x\in \R,
\end{equation*}
where  $a_+ = \max\{a,0\}$. The function $G: \R \to [0,\infty)$ is non-negative, finite valued, and convex. One can also check that the second distributional derivative $\FF_\mu''$, being \textit{a priori} a non-negative Borel measure, is equal back to the probability $\mu$, thus portraying $\FF_\mu$ as a double primitive function of $\mu$. 

Note that, for each $x$, $(x- \argu)_+$ is a special case of convex function. Therefore, we must have $\FF_\mu(x) \leq \FF_\nu(x)$ for all $x \in \R$ whenever the convex order $\mu \preceq_c \nu$ holds true. The converse turns out to be also true provided that the barycentres match.
\begin{proposition}
	\label{prop:cvx_order_char1D}
	Let us be given two measures $\mu,\nu \in \mathcal{P}_1(\R)$ on the real line sharing the barycentre: $[\mu] = [\nu]$. Then, $\mu \preceq_c \nu$ if and only if $\FF_\mu \leq \FF_\nu$. 
\end{proposition}
\begin{proof}
	See, for instance, \cite[Theorem 3.A.1]{shaked2007} or  \cite[Proposition 2.4]{hirsch2012}.
\end{proof}

The idea for building the least upper bound $\mu \vee \nu$ takes its shape: it should be the second distributional derivative of the function $g(x) =(\FF_\mu \vee \FF_\nu)(x) :=  \max\{ \FF_\mu(x), \FF_\nu(x)\}$.
The closed formula for the optimal convex dominant in 1D can be readily given.

\begin{proposition}
	\label{prop:1D}
	Let us be given two measures $\mu,\nu \in \mathcal{P}_1(\R)$ on the real line with the same barycentre, $[\mu] = [\nu]$. Then, 
	\begin{equation}
		\label{eq:lub}
		\rho = \mu \vee \nu := (\FF_\mu \vee \FF_\nu) ''
	\end{equation}
	is a solution of the optimal convex dominant problem \eqref{eq:conv_dom} for  any convex cost function $f:\R\to\R \cup \{+\infty\}$ such that  $\inf \eqref{eq:conv_dom} < +\infty$.
	
	Moreover, if $f$ is strictly convex, then the solution $\rho = \mu \vee \nu$ is unique.
\end{proposition}
\begin{proof}
	The functions $\FF_\mu$, $\FF_\nu$ are convex, and by \cite[Proposition 2.1]{hirsch2012} we have $\lim_{x \to -\infty} {\FF_\mu}(x) = 0$ and $\lim_{x \to +\infty} (-{\FF_\mu}(x) + x)  = [\mu]$. The same is true for $\FF_\nu$. Then, $g =\FF_\mu \vee \FF_\nu$ is clearly convex, and the convergences $\lim_{x \to -\infty} {g}(x) = 0$ and $\lim_{x \to +\infty} (-{g}(x) + x)  = [\mu] = [\nu]$ are inherited from those for $\FF_\mu, \FF_\nu$. Again thanks to \cite[Proposition 2.1]{hirsch2012}, we have $\rho \in \mathcal{P}_1(\R)$, $[\rho] = [\mu] = [\nu]$, and $\FF_{\rho} =  {g}$. Since $\FF_\rho \geq \FF_\mu$ and $\FF_\rho \geq \FF_\nu$,  Proposition \ref{prop:cvx_order_char1D} renders $\rho$ admissible in \eqref{eq:conv_dom}.
	
	Let now $\breve{\rho}$ be any admissible measure for \eqref{eq:conv_dom}, in particular $[\breve{\rho}] = [\mu] = [\nu]$. Owing to Proposition \ref{prop:cvx_order_char1D} again, $\FF_{\breve{\rho}} \geq \FF_\mu$ and $\FF_{\breve{\rho}} \geq \FF_\nu$,  therefore $\FF_{\breve{\rho}} \geq {g} = \FF_{{\rho}}$. Using the same result once more, we see that $\breve{\rho} \succeq_c {\rho}$, and thus $\int_\R f d\breve{\rho}  \geq \int_\R f d{\rho}$, furnishing optimality of $\rho$. The moreover part of the statement now follows from \eqref{eq:strict}. 
\end{proof}

\subsection{Link with the bi-martingale optimal transport}

For a pair that is admissible for $(\mathrm{M^2OT})$, namely $(\gamma,q) = (\gamma,\zeta\gamma) \in \Gamma Q (\mu,\nu)$, a candidate for a convex dominant of $\mu$ and $\nu$ can be found through the push-forward,
\begin{equation}
	\label{eq:rhozetagamma_first}
	\rho = \zeta \# \gamma.
\end{equation}
In the introduction we have deduced that $\rho \succeq_c \mu$ and $\rho \succeq_c \nu$ based on Strassen theorem. For a more direct argument, we can exploit the equations \eqref{eq:2xmartingale_condition} holding for a feasible couple $(\gamma,\zeta\gamma)$. Accordingly, with the disintegration $\gamma(dxdy) = \mu(dx) \otimes \gamma_x(dy)$, we know that $ \int_Y \zeta(x,y) \,\gamma_x(dy) = x$ for $\mu$-a.e. $x$. Then, for any convex function $\varphi:\Rd \to \R$ Jensen's inequality  gives,
\begin{align*}
	\int_\Rd \varphi(z) \,\rho(dz) &= \int_X \! \int_Y \varphi\big( \zeta(x,y) \big) \, \gamma(dxdy)= \int_X \left( \int_Y \varphi\big( \zeta(x,y) \big) \, \gamma_x(dy) \right) \mu(dx) \\
	&\geq  \int_X \varphi\left( \int_Y \zeta(x,y) \,\gamma_x(dy)\right) \mu(dx) = \int_X \varphi(x) \,\mu(dx),
\end{align*}
thus establishing $\rho \succeq_c \mu$. The disintegration with respect to the other marginal, $\gamma(dxdy) =  \gamma_y(dx) \otimes \nu(dy)$, and the second equality in \eqref{eq:2xmartingale_condition} furnish $\rho \succeq_c \nu$.

With the change of variables \eqref{eq:rhozetagamma_first} in mind, we can now propose the bi-martingale reformulation of the optimal convex dominant problem \eqref{eq:conv_dom},
\begin{align}
	\label{eq:conv_dom_gammaq}
	\inf\left\{ \int_X \! \int_Y f  \Big(\frac{d q}{d\gamma}(x,y) \Big) \gamma(dxdy) \ : \ 	(\gamma,q) \in \Gamma Q(\mu,\nu)  \right\}.
\end{align}
To prove the equivalence between the two problems, one still have to show that the set of $\rho$ generated from $(\gamma,q) \in \Gamma Q(\mu,\nu)$ via \eqref{eq:rhozetagamma_first} is rich enough to produce minimizers in \eqref{eq:conv_dom}. The desired result will follow smoothly once we pass through the 3-plan formulation presented in Section \ref{ssec:3-plans}.

\begin{proposition}
\label{prop:rhozetagamma}
The equality $\inf \eqref{eq:conv_dom} = \inf \eqref{eq:conv_dom_gammaq}$ holds true, and
an admissible pair $(\gamma,q) = (\gamma,\zeta \gamma)  \in \Gamma Q(\mu,\nu)$ is a solution to \eqref{eq:conv_dom_gammaq} if and only if
\begin{equation*}
	\rho = \zeta \# \gamma
\end{equation*}
solves the optimal convex dominance problem \eqref{eq:conv_dom}.

If, in addition, $f$ is strictly convex, then to every  $\rho$ solving \eqref{eq:conv_dom} one can assign at least one pair $(\gamma,q) = (\gamma,\zeta \gamma)  \in \Gamma Q(\mu,\nu)$  such that $\rho = \zeta \# \gamma$.
\end{proposition}

Owing to the existence result for $(\mathrm{M^2OT})$ in Corollary \ref{cor:existence}, existence of an optimal dominant follows as well.

\begin{corollary}
	\label{cor:existence_dominant} For every convex superlinear function $f:\Rd \to \R \cup +\infty$, the optimal convex dominance problem \eqref{eq:conv_dom} admits a (in general non-unique) solution if $\inf \eqref{eq:conv_dom} < + \infty$.
\end{corollary}

\begin{remark}
	\label{rem:conv_dom}
	Before giving the proof, let us make several comments on the result:
	\begin{itemize}
		\item[(a)] A key consequence of Proposition \ref{prop:rhozetagamma}  can be deduced when $\mu,\nu$ are discrete. Since the discreteness automatically propagates to  plans $\gamma$, our result shows that there are solutions $\rho$  for  \eqref{eq:conv_dom} which are discrete as well (all of them are if $f$ is strictly convex). Contrarily, Proposition \ref{prop:1D}  shows that for generic data  that are absolutely continuous with respect to Lebesgue measure (even with smooth densities), solutions $\rho = (G_\mu \vee G_\nu)''$ contain atoms: maximum of two smooth functions is typically not smooth.
		\item[(b)] Related to the previous comment is the influence of  Proposition \ref{prop:rhozetagamma} on the computational perspectives for the optimal convex dominance problem \eqref{eq:conv_dom}. Indeed, if $\mu,\nu$ are discrete, then  \eqref{eq:conv_dom_gammaq} is a finite dimensional convex problem. Meanwhile, attacking \eqref{eq:conv_dom} directly would typically involve looking for $\rho$ supported on a fixed grid of points, and the accuracy of the solution would depend on its resolution, even for data as simple as in Fig.~\ref{fig:bimart}. See Section \ref{ssec:conic}, where this observation is reflected in practice.
		\item[(c)] The formula $\rho = \zeta \# \gamma$ is reminiscent of \cite[Proposition 4.2]{agueh2011barycenters}, where the Wasserstein barycentre is found as the push-forward of an optimal multi-marginal plan \cite{gangbo1998optimal} through the linear barycentre map $T(x_1,\ldots,x_n) = \sum_{i} \lambda_i x_i$. Let us briefly note that adaptation of \eqref{eq:conv_dom_gammaq} to an \textit{$n$-martingale} problem, that would underlie the optimal convex dominance for $n$ marginals, is straightforward.  
		\item[(d)] Non-uniqueness of solutions to \eqref{eq:conv_dom} will be demonstrated in Example \ref{ex:non-uniqueness} for the quadratic cost $f = \abs{\argu}^2$. It is related to the lack of displacement convexity for the forward cone $\{ \argu \succeq_c \mu\}$ as opposed to the backward cone $\{ \argu \preceq_c \mu\}$, see \cite[Example 8.2]{kim2024}.
		\item[(e)] From the definition of convex order, we can easily see that the  necessary condition for the finiteness of the total cost in \eqref{eq:conv_dom} is that $\max\{ \int f\,d\mu, \int f\,d\nu \} < +\infty$. It is also sufficient if the cost function satisfies a doubling condition. Assume that $[\mu] = [\nu] =0$ and there exists $C >0$ such that $f(2x) \leq C f(x)$ for every $x \in \Rd$. Then, we can use the fact that the convolution $\mu *\nu$ dominates both probabilities for convex order to make the relevant estimate.
		For instance, if $f = \abs{\argu}^p$ for $p>1$, existence of solutions with finite cost in \eqref{eq:conv_dom}  is equivalent to finiteness of the $p$-th moments of $\mu$ and $\nu$. 
		\item[(f)] The strict convexity assumption in the moreover part in Proposition \ref{prop:rhozetagamma} is essential. Indeed, consider the trivial data $\mu =\nu= \delta_0$ and assume that $f$ is affine on a convex set $A \ni \nolinebreak 0$. Then, any centred probability $\rho$ that is supported on $A$ solves \eqref{eq:conv_dom}, whilst $\Gamma Q(\mu,\nu) = \{(\delta_{(0,0)},0)\}$.
	\end{itemize}
\end{remark}

\begin{proof}[Proof of Proposition \ref{prop:rhozetagamma}]
	Take any pair $(\gamma,q) = (\gamma,\zeta \gamma) \in \Gamma Q$ and define $\rho = \zeta \# \gamma$. We have,
	\begin{align}
		\label{eq:chain_gammaq_rho}
		\iint f\Big( \frac{dq}{d\gamma}(x,y)\Big) \,\gamma(dxdy) = 	\iint f\big( \zeta(x,y)\big) \,\gamma(dxdy)  = \int f(z)\,\rho(dz)  \geq \inf \eqref{eq:conv_dom}.
	\end{align}
	Taking the infimum of the left hand side with respect to $(\gamma,q) \in \Gamma Q(\mu,\nu)$ furnishes the inequality $\inf \eqref{eq:conv_dom_gammaq} \geq \inf \eqref{eq:conv_dom}$.
	
	Conversely, take any admissible $\rho\in \mathcal{P}_1(\Rd)$ such that $\rho \succeq_c\mu$ and $\rho \succeq_c \nu$. By Strassen theorem, we can select martingale plans $\gamma_1 \in \Gamma_{\mathrm{M}}(\mu,\rho)$, $\gamma_2 \in \Gamma_{\mathrm{M}}(\nu,\rho)$.  Next, by virtue of the gluing lemma (cf. \cite[Lemma 7.6]{villani}), a probability  ${\varpi} \in \mathcal{P}(X \times Y \times \Rd)$ can be constructed such that $\pi_{1,3} \# {\varpi} = {\gamma}_1$,  $\pi_{2,3} \# {\varpi} = {\gamma}_2$,  and $\pi_{1,2} \# \varpi \in \Gamma(\mu,\nu)$. Therefore, clearly $\varpi \in \Sigma(\mu,\nu)$. Accordingly, Proposition \ref{prop:gammaq-omega} for the cost $\cost(x,y,z) = f(z)$ yields,
	\begin{align}
		\label{eq:chain_omega_rho}
		\inf \eqref{eq:conv_dom_gammaq} = \inf \eqref{eq:3plans} \leq \iiint f(z) \,\varpi(dxdydz) = \int f(z) \,\rho(dz),
	\end{align}
	since $\pi_3 \# \varpi = \rho$.	Minimizing the right hand side with respect to $\rho$ dominating $\mu,\nu$ gets us to the converse inequality  $\inf \eqref{eq:conv_dom_gammaq} \leq \inf \eqref{eq:conv_dom}$.
With the equality $\inf \eqref{eq:conv_dom} = \inf \eqref{eq:conv_dom_gammaq}$ at hand, the equivalence between optimality of $(\gamma,q)=(\gamma,\zeta\gamma)$ and optimality of $\rho = \zeta \# \gamma$ now follows from \eqref{eq:chain_gammaq_rho}.

Assume now that $f$ is strictly convex and $\rho$ is optimal in \eqref{eq:conv_dom}. Construct $\varpi$ as above. From  \eqref{eq:chain_omega_rho} we can deduce optimality of $\varpi$ in \eqref{eq:3plans}, and, thanks to the strict convexity of $f$, Proposition \ref{prop:gammaq-omega} guarantees the form $\varpi(dxdydz) = \gamma(dxdy) \otimes \delta_{\zeta(x,y)}(dz)$. We can now see that $\zeta \# \gamma = \pi_3 \# \varpi = \rho$, which ends the proof.
\end{proof}

\section{The quadratic case and the Zolotarev-2 distance}

\label{sec:Zol}

The results of this sections are fuelled by the main result of the paper \cite{bolbotowski2024kantorovich}, which can be referred to as the second-order Kantorovich--Rubinstein duality, where in place of the Wasserstein-1 distance (equal to the Zolotarev-1 distance) we put the Zolotarev-2 distance, i.e.
\begin{equation}
	\label{eq:Z2_again}
	Z_2(\mu,\nu) = \max\left\{ \int_\Rd u \, d(\mu-\nu) \, : \, u \in C^{1,1}(\Rd), \ \mathrm{lip}( \nabla  u) \leq 1 \right\}.
\end{equation}
In the entire Section \ref{sec:Zol}, $\mu,\nu$ will be probability distributions with equal barycentre $\ba \in \Rd$ and of finite second moments, namely $\mu,\nu \in \mathcal{P}^\ba_2(\Rd)$.  For such data, the maximum in the \textit{Zolotarev problem} \eqref{eq:Z2_again}  is always attained, cf. \cite[Proposition 2.1]{bolbotowski2024kantorovich}. Solutions are not unique. In particular, adding an affine function does not change optimality of $u$, but less trivial modifications are possible, for instance far enough from the convex hull of the union of supports of $\mu,\nu$.
%In the work \cite{bolbotowski2024kantorovich} this distance was rewritten via the 3-plan optimization problem \eqref{eq:3plans} with the cost $\cost(x,y,z) = \frac{1}{2}\abs{z-x}^2 + \frac{1}{2} \abs{z-y}^2$. We will now translate this result to the bi-martingale transport framework $(\mathrm{M^2OT})$, and provide characterization for

\subsection{Second-order Kantorovich--Rubinstein duality revisited}
\label{ssec:RK}

A convenient intermediate point on the path towards the second-order Kantorovich--Rubinstein duality is the optimal convex dominant problem for the quadratic cost $f = \abs{\argu}^2$, 
\begin{equation}
	\label{eq:C_rho}
	\mathcal{C}(\mu,\nu) := \min\Big\{m_2(\rho) \, : \, \rho \in \mathcal{P}_2(\Rd), \ \ \rho \succeq_c \mu, \ \ \rho \succeq_c \nu \Big\},
\end{equation}
where $m_2(\rho) = \int_\Rd \abs{z}^2 \rho(dz)$. Since $\mathcal{C}(\mu,\nu)$ is always finite for  $\mu,\nu \in \mathcal{P}^\ba_2(\Rd)$ (cf. Remark \ref{rem:conv_dom}(e)), solutions to \eqref{eq:C_rho} exist by Corollary \ref{cor:existence_dominant}. The dual problem entails maximizing with respect to a pair of convex potentials. Then, its non-trivial transformation founded on the fact that convexification preserves semi-concavity, has lead to the key equality of the paper \cite{bolbotowski2024kantorovich}.
\begin{theorem}[\cite{bolbotowski2024kantorovich}, Theorem 1.3]
	\label{thm:RK}
	For any pair of probabilities $\mu,\nu \in \mathcal{P}^\ba_2(\Rd)$  the following equality holds true,
	\begin{equation}
		\label{eq:CZ}
		Z_2(\mu,\nu) = \mathcal{C}(\mu,\nu) - \frac{m_2(\mu)+ m_2(\nu)}{2}.
	\end{equation}
\end{theorem}
Expressing the Zolotarev-2 distance by means of the bi-martingale transport is readily at our disposal thanks to Proposition \ref{prop:rhozetagamma} which states that,
\begin{align}
	\label{eq:C_gammaq}
	\mathcal{C}(\mu,\nu) =  \min\left\{ \int_X \! \int_Y   \Big|\frac{d q}{d\gamma}(x,y) \Big|^2 \gamma(dxdy) \ : \ 	(\gamma,q) \in \Gamma Q(\mu,\nu)  \right\}.
\end{align}
It is the $(\mathrm{M^2OT})$ problem for the cost $\cost(x,y,z) = \abs{z}^2$. Passage to the cost $\cost(x,y,z) = \frac{1}{2}\abs{z-x}^2 + \frac{1}{2} \abs{z-y}^2$, advertised as cost (B) in the introduction, is a straightforward use of the properties of the set $\Gamma Q (\mu,\nu)$. Together with \eqref{eq:C_gammaq}, the following result yields the first part of Theorem \ref{thm:RK_reinvented}.

\begin{corollary}
	\label{cor:RK_duality}
	We have the equality,
	\begin{align*}
		Z_2(\mu,\nu) &=  \min\left\{ \int_X \! \int_Y \!\frac{1}{2} \bigg( \Big| \frac{dq}{d\gamma}(x,y) - x \Big|^2\! \!\!+ \Big| \frac{dq}{d\gamma}(x,y) - y \Big|^2  \bigg)   \gamma(dxdy) \, : \, 	(\gamma,q) \in \Gamma Q(\mu,\nu)  \right\}
	\end{align*}
	and the set of minimizers is identical to the set of minimizers in \eqref{eq:C_gammaq}.
\end{corollary}
\begin{remark}
	Note that this result could  be recovered directly from \cite[Theorem 1.1]{bolbotowski2024kantorovich} by translating the 3-plan framework used therein to the one of $(\mathrm{M^2 OT})$ via Proposition \ref{prop:gammaq-omega}.
\end{remark}
\begin{proof}
	Using the fact that $\pi_1 \#\gamma = \mu$ and $\pi_1 \#q = x\mu$, we obtain,
	\begin{align}
		\nonumber
		\iint \frac{1}{2}\Big|\frac{dq}{d\gamma} - x \Big|^2 d\gamma  &=\iint  \frac{1}{2}\Big|\frac{dq}{d\gamma}(x,y)  \Big|^2 \gamma(dxdy)  -\iint \pairing{x,q(dxdy)} + \iint \frac{1}{2} \abs{x}^2 \gamma(dxdy)   \\
		\label{eq:mix_term}
		& = \iint  \frac{1}{2}\Big|\frac{dq}{d\gamma}(x,y)  \Big|^2 \gamma(dxdy) - \frac{m_2(\mu)}{2}.
	\end{align}
	Similarly, $	\iint \frac{1}{2}|\frac{dq}{d\gamma} - y |^2 d\gamma =  \iint  \frac{1}{2}|\frac{dq}{d\gamma}  |^2 d\gamma - \frac{m_2(\nu)}{2}$.  We thus see that the cost functional in the asserted equality coincides with the one in \eqref{eq:C_gammaq} up to the additive constant $\frac{1}{2} \big( m_2(\mu) + m_2(\nu) \big)$.  Readily, the equality follows by Theorem \ref{thm:RK}.
\end{proof}

%\begin{remark}[\textbf{Optimal grillage problem}] Let us briefly comment on the main motivation of the paper \cite{bolbotowski2024kantorovich} \red{(Good idea????????)}
%	\begin{equation}
%		\label{grillage_sigma}
%		Z_2(\mu,\nu) = \min\biggl\{ \int_\Rd \varrho^0(\sigma) \ : \ \sigma \in \Mes(\Rd;\R_{\mathrm{sym}}^{d \times d}), \ \dive^2 \sigma = \mu - \nu  \biggr\}.
%	\end{equation}
%	
%	\begin{align*}
%		\int_\Rd \pairing{\Theta, d \sigma_i}= \iint_{(\Rd)^2} \left( \int_0^1 t\, \Big\langle(x-z) \otimes(x-z),\Theta\big(z+t(x-z)\big) \Big\rangle \, dt \right)  \gamma_i(dxdz)
%	\end{align*}
%\end{remark}

\subsection{Characterizing optimal  bi-martingale plans in the quadratic case}
 We will characterize the set of plans $(\gamma,q)$ which are minimal for the quadratic case, that is $(\gamma,q)$ belonging to the set,
\begin{align}
		\label{eq:optGammaQ}
		\ov{\Gamma Q} (\mu,\nu)  :=& \argmin_{(\gamma,q) \in \Gamma Q(\mu,\nu)}   \int_X \! \int_Y  \Big|\frac{d {q}}{d\gamma} \Big|^2 d\gamma  \\
		\nonumber
		=&  \argmin_{(\gamma,q) \in \Gamma Q(\mu,\nu)}  \int_X \! \int_Y \frac{1}{2} \bigg( \Big| \frac{dq}{d\gamma}(x,y) - x \Big|^2 + \Big| \frac{dq}{d\gamma}(x,y) - y \Big|^2  \bigg)   \gamma(dxdy),
\end{align}
where the equality of the two sets of minimizers is due to Corollary \ref{cor:RK_duality}.
In the sequel the overbar $\ov{\argu}$ will stand for optimality in the context of the quadratic case.

To any  $\ov{u} \in C^{1,1}(\Rd)$ solving the Zolotarev problem \eqref{eq:Z2_again} for a pair $\mu,\nu \in \mathcal{P}_2^\ba(\Rd)$, we assign the Lipschitz continuous coupling map,
\begin{equation}
	\label{eq:ovzeta_again}
	\ov\zeta(x,y) := \frac{x+y}{2} + \frac{\nabla \ov{u}(x) - \nabla \ov{u}(y)}{2},
\end{equation}
and then the set of bi-martingale plans with respect to to that map (cf. \eqref{eq:ovGamma} for an equivalent definition),
\begin{equation}
	\label{eq:ovGamma_again}
	\ov\Gamma(\mu,\nu) :=\Big\{ \gamma \in \Gamma(\mu,\nu) \, : \, \ov\zeta \gamma \in Q(\mu,\nu) \Big\}.
\end{equation}
The following result proves the equivalence (i) $\Leftrightarrow$ (ii) in Theorem \ref{thm:RK_reinvented}.

%Observe that, by the admissibility of $\ov{u}$ in \eqref{eq:Z2_again}, $\mathrm{lip}(\nabla \ov{u}) \leq 1$, which renders the map $\ov\zeta$ Lipschitz continuous as well. On top of that, since $\abs{\nabla \ov{u}(x) - \nabla \ov{u}(y)} \leq \abs{x-y}$, for every pair $x,y$ the point $\ov\zeta(x,y)$ lies in the closed ball of diameter $[x,y]$.

\begin{proposition}
	\label{prop:ovGammaQ_char}
	The set of optimal bi-martingale plans in the quadratic case equals,
	\begin{equation}
		\label{eq:ovGammaQ_char}
		\ov{\Gamma Q}(\mu,\nu) = \Big\{ (\gamma,q) \,  : \,  \gamma \in \ov\Gamma(\mu,\nu),  \ \  q = \ov{\zeta} \gamma  \Big\},
	\end{equation}
	where $\ov\zeta$ and $\ov{\Gamma}(\mu,\nu)$ can be defined for any solution $\ov{u}$ of the Zolotarev problem \eqref{eq:Z2_again}. In particular, the set  $\ov{\Gamma}(\mu,\nu)$ is non-empty, and it does not depend on the choice of the possibly non-unique maximizer~$\ov{u}$.
	
	Moreover, for any $\gamma \in \ov{\Gamma}(\mu,\nu)$, 
	\begin{align}
		\label{eq:Z2_ovzeta}
		Z_2(\mu,\nu) &= \int_X \! \int_Y \frac{1}{2} \Big( \big| \ov{\zeta}(x,y) - x \big|^2 + \big| \ov{\zeta}(x,y) - y \big|^2  \Big)   \gamma(dxdy), \\
		\label{eq:Cmunu_ovzeta}
		\mathcal{C}(\mu,\nu) &= \int_X \! \int_Y   \big| \ov{\zeta}(x,y) \big|^2 \gamma(dxdy).
	\end{align}
\end{proposition}
\begin{proof}
	 First we prove the inclusion $\subset$. By the admissibility in the Zolotarev problem, we have $\mathrm{lip}(\nabla \ov{u}) \leq 1$. According to \eqref{eq:3-point_ineq}, it means that for all $(x,y,z) \in(\Rd)^3$,
	\begin{align*}
		\ov{u}(x) - \ov{u}(y) + \pairing{\nabla \ov{u}(x),z-x} -  \pairing{\nabla \ov{u}(y),z-y} \leq \frac{1}{2} \Big( \abs{z-x}^2 + \abs{z-y}^2\Big).
	\end{align*}
	One checks easily that  $z = \ov\zeta(x,y)$ is the unique minimizer of the gap in this inequality. In particular, if $z \neq  \ov\zeta(x,y)$, then the inequality must be strict. Take a pair $(\gamma,q) = (\gamma,\zeta\gamma) \in \Gamma Q(\mu,\nu)$. Testing \eqref{eq:moment-martingale} with $\varphi = \ov{u} = -\psi$ and $\Phi = \nabla \ov{u} = -\Psi$ (which is valid since $\ov{u}$  and $\nabla\ov{u}$ are of at most quadratic and linear growth, respectively, while  $\mu,\nu \in \mathcal{P}_2(\Rd)$ and $\iint \abs{\zeta}^2 d\gamma < + \infty$), we get,
	\begin{align*}
		\int \ov{u} \,d(\mu-\nu) &= \iint \Big(	\ov{u}(x) - \ov{u}(y) + \pairing{\nabla \ov{u}(x),\zeta(x,y) -x} -  \pairing{\nabla \ov{u}(y),\zeta(x,y)-y} \Big) \, \gamma(dxdy) \\
		& \leq \iint \frac{1}{2} \Big( \abs{\zeta(x,y)-x}^2 + \abs{\zeta(x,y)-y}^2\Big) \,\gamma(dxdy).
	\end{align*}
	The left hand side is equal to $Z_2(\mu,\nu)$ by maximality of $\ov{u}$. Thus, owing to Corollary \ref{cor:RK_duality},  the optimality $(\gamma,\zeta\gamma) \in \ov{\Gamma Q}(\mu,\nu)$ implies that the inequality above is an equality. This is possible only if $\zeta(x,y) = \ov\zeta(x,y)$ for $\gamma$-a.e. $(x,y)$, which proves the inclusion $\subset$.

	Conversely, take any $\gamma \in\ov{\Gamma}(\mu,\nu)$ and define $q = \ov{\zeta} \gamma$. By definition of the set  $\ov{\Gamma}(\mu,\nu)$, we have $\gamma \in \Gamma(\mu,\nu)$ and $q \in Q(\mu,\nu)$. Recalling the definition of $\ov{\zeta}$, we compute,
	\begin{align}
		\nonumber
		\iint     \big| \ov{\zeta} \big|^2 d\gamma &=  \frac{1}{2}	\iint     \big\langle x + \nabla \ov{u}(x),  \ov{\zeta}(x,y)  \big\rangle \,  \gamma(dxdy)  +  \frac{1}{2}	\iint     \big\langle y - \nabla \ov{u}(y), \ov{\zeta}(x,y)\big \rangle  \, \gamma(dxdy) \\
		\nonumber
		& =   \frac{1}{2}	\iint     \big\langle x + \nabla \ov{u}(x), q(dxdy) \big\rangle   +  \frac{1}{2}	\iint     \big\langle y - \nabla \ov{u}(y),q(dxdy)\big \rangle \\
		\label{eq:chain_eq}
		& =   \frac{1}{2}	\int     \big\langle x + \nabla \ov{u}(x), x \big\rangle \, \mu(dx)   +  \frac{1}{2}	\int     \big\langle y - \nabla \ov{u}(y),y\big \rangle \, \nu(dy),
	\end{align}
	where the last equality is owing to $q \in Q(\mu,\nu)$. We have thus showed that the value of the integral $	\iint     \big| \ov{\zeta}(x,y) \big|^2 \gamma(dxdy) $ is independent of the choice of $\gamma \in\ov{\Gamma}(\mu,\nu)$. Take now a pair $(\ov\gamma,\ov{q}) \in \ov{\Gamma Q}(\mu,\nu)$, which always exists. By the already proven inclusion $\subset$, there must hold $\ov{q} = \ov\zeta\ov\gamma$, while $\ov{\gamma}$ must be an element of $\ov{\Gamma}(\mu,\nu)$. Altogether, we obtain,
	\begin{equation*}
		 \iint  \Big|\frac{d {q}}{d\gamma} \Big|^2 d\gamma = \iint     \big| \ov{\zeta} \big|^2 d\gamma = \iint     \big| \ov{\zeta} \big|^2 d\ov\gamma = \iint  \Big|\frac{d \ov{q}}{d\ov\gamma} \Big|^2 d\ov\gamma = \mathcal{C}(\mu,\nu),
	\end{equation*} 
	where the final equality is by optimality of $(\ov{\gamma},\ov{q})$. This renders $(\gamma,q)$ an element of $\ov{\Gamma Q}(\mu,\nu)$ as well, and the inclusion $\supset$ is established. The independence of $\ov\Gamma(\mu,\nu)$ with respect to the choice of the optimal potential $\ov{u}$ now follows by the independence of $\ov{\Gamma Q}(\mu,\nu)$. 
\end{proof}

%\begin{remark}
%	\label{rem:Cmunu_ovzetaq}
%	From the proof, the chain of equalities \eqref{eq:chain_eq} in particular, one can deduce  a more flexible formula for $\mathcal{C}(\mu,\nu)$ that will prove useful in the sequel,
%	\begin{equation*}
%		\mathcal{C}(\mu,\nu) =  \int_X \! \int_Y   \pairing{ \ov{\zeta}(x,y), q(dxdy)} \qquad \quad \forall q \in Q(\mu,\nu).
%	\end{equation*}
%\end{remark}

\begin{remark}
	Unlike for the set $\ov{\Gamma}(\mu,\nu)$, the possible non-uniqueness of the optimal potential $\ov{u}$ can induce the non-uniqueness of the coupling map $\ov\zeta$, e.g. far away from the set $\spt(\mu) \otimes \spt(\nu)$. However, the characterization \eqref{eq:ovGammaQ_char} implies that $\ov\zeta$ is uniquely determined on the set $\bigcup_{ \gamma\in  \ov{\Gamma}(\mu,\nu)} \spt (\gamma)$.
\end{remark}

%By virtue of Proposition \ref{prop:gammaq-omega} the result in Proposition \ref{prop:ovGammaQ_char} can be translated to the framework of 3-plan optimization.
%\begin{corollary}
%	\label{cor:ovomega_char}
%	The set of $\varpi$ that solve the problem \eqref{eq:3plans} for the cost $\cost(x,y,z) = \abs{z}^2$ is,
%	\begin{equation}
%		\Big\{ \varpi(dxdydz) = \gamma(dxdy) \otimes \delta_{\ov\zeta(x,y)}(dz)  \, : \, \gamma \in\ov{\Gamma}(\mu,\nu)  \Big\}.
%	\end{equation}
%	The same holds true for the cost $\cost(x,y,z) = \frac{1}{2} \abs{z-x}^2 + \frac{1}{2} \abs{z-y}^2$.
%\end{corollary}

In the following, we shall show that the set $\ov{\Gamma}(\mu,\nu)$ is exactly the set of martingale plans $\Gamma_{\mathrm{M}}(\mu,\nu)$ once we have the convex order $\mu \preceq_c \nu$. This characterization will rely  on the fact that $\ov{u} = -\frac{1}{2}\abs{\argu}^2$ solves the Zolotarev problem in this case, which gives $\ov\zeta(x,y) = y= \pi_2(x,y)$.

\begin{proposition}
	\label{prop:cvx_order_char}
	For a pair  $\mu, \nu \in \mathcal{P}^\ba_2(\Rd)$, the following conditions are equivalent: 
	\begin{enumerate}[label={(\roman*)}]
		\item the convex order relation $\mu \preceq_c \nu$ holds true;
		\item $\ov{u}(x) = -\frac{1}{2} \abs{x}^2 $ solves the Zolotarev problem \eqref{eq:Z2_again}  or, equivalently,
		\begin{equation}
			\label{eq:Z2_ordered}
			Z_2(\mu,\nu) = \frac{m_2(\nu) - m_2(\mu)}{2};
		\end{equation}
		\item $\ov{\rho} = \nu$ is the (unique) solution of the minimal convex dominance problem \eqref{eq:C_rho} or, equivalently,
			\begin{equation*}
				\mathcal{C}(\mu,\nu) = m_2(\nu);
			\end{equation*}
		\item the set of optimal bi-martingale plans equals $\ov{\Gamma}(\mu,\nu) = \Gamma_{\mathrm{M}}(\mu,\nu)$ or, equivalently,
		\begin{equation}
			\label{eq:ovGammaQ_char_order}
			\ov{\Gamma Q} (\mu,\nu)  = \Big\{  (\gamma,q) \, : \, \gamma \in \Gamma_{\mathrm{M}}(\mu,\nu), \ q = y \gamma  \Big\}.
		\end{equation}
	\end{enumerate}
\end{proposition}
\begin{proof}
	To prove the implication (i) $\Rightarrow$ (ii) assume the convex order $\mu \preceq_c \nu$.
	Take a competitor $u$ in the Zolotarev problem \eqref{eq:Z2_again}. Then, 
	$\varphi = u+  \frac{1}{2} \abs{\argu}^2$ is a convex function, see \eqref{eq:u_adm_char}, and so,
	\begin{align*}
		\int_\Rd u \, d(\mu-\nu) = \frac{m_2(\nu) - m_2(\mu)}{2} + \int \varphi\, d(\mu-\nu) \leq  \frac{m_2(\nu) - m_2(\mu)}{2}.
	\end{align*}
	By taking the supremum on the left hand side with respect to admissible functions $u$ we see that $Z_2(\mu,\nu) \leq  \frac{1}{2}\big(m_2(\nu) - m_2(\mu)\big)$. Clearly, the inequality above is saturated for $\varphi = 0$, which corresponds to $u = \ov{u} = -\frac{1}{2} \abs{\argu}^2$. This gives the condition (ii).
	
	Next, we show that (ii) $\Rightarrow$ (iii). Assume that $\ov\rho$ is a solution to \eqref{eq:C_rho}, in particular $\ov\rho \succeq_c \mu$  and $\ov\rho \succeq_c \nu$.  Theorem \ref{thm:RK} combined with \eqref{eq:Z2_ordered} says that $\mathcal{C}(\mu,\nu) = m_2(\nu)$. Thus, $\int \abs{\argu}^2 d\ov\rho = \mathcal{C}(\mu,\nu) = \int \abs{\argu}^2 d\nu$ and, since $\abs{\argu}^2$ is strictly convex, from \eqref{eq:strict} we deduce that $\ov\rho=\nu$. This necessary optimality condition turns sufficient if we take into account that \eqref{eq:C_rho} always has  a solution. We have thus arrived at  (iii). 
	
%	Next, we show that (ii) $\Rightarrow$ (iii). Theorem \ref{thm:RK} combined with \eqref{eq:Z2_ordered} say  that $\mathcal{C}(\mu,\nu) = m_2(\nu)$. Take any competitor $\rho$ in  \eqref{eq:C_rho}, which satisfies $\rho \succeq_c \nu$, in particular. By Strassen theorem there exists a measurable probability kernel $y \to \lambda_y \in \mathcal{P}(\Rd)$ such that $\rho(B) = \int \lambda_y(B) \, \nu(dy) $ for every Borel set $B$, while $[\lambda_y] = y$ for $\nu$-a.e. $y$. By Jensen's inequality, we get
%	\begin{equation}
%		\int \abs{z}^2 \rho(dz) = \int \left( \int \abs{z}^2 \lambda_y(dz) \right) \nu(dy)  \geq \int  \big|[\lambda_y] \big|^2 \nu(dy) = \int \abs{y}^2 \nu(dy) = m_2(\nu) =\mathcal{C}(\mu,\nu).
%	\end{equation}
%	Accordingly, $\rho$ being a minimizer in $\eqref{eq:C_rho}$ implies that the Jensen's inequality is an equality. Due to the strict convexity of the function $\abs{\argu}^2$ this is possible only if $\lambda_y = \delta_y$ for $\nu$-a.e. $y$, which enforces that $\rho =\nu$. This necessary condition turns sufficient when we take into account that \eqref{eq:C_rho} always admits a solution. We have thus arrived at  (iii). 
	
	Showing (iii) $\Rightarrow$ (i) is straightforward: every competitor $\rho$ in  \eqref{eq:C_rho} dominates $\mu$ in the convex order, and so this also applies to $\ov\rho = \nu$ in the case when (iii) holds true. 
	
	We have so far proved the equivalences (i) $\Leftrightarrow$ (ii) $\Leftrightarrow$ (iii). The implication (iv) $\Rightarrow$ (i) is an easy consequence of Strassen theorem. Indeed, since $\ov{\Gamma Q}(\mu,\nu)$ is non-empty, (iv) implies that there exists a martingale plan from $\mu$ to $\nu$. Finally, to show that (ii) $\Rightarrow$ (iv) we exploit Proposition \ref{prop:ovGammaQ_char}. If $\ov{u}(x) = -\frac{1}{2} \abs{x}^2 $  is the solution of \eqref{eq:Z2_again}, then the formula \eqref{eq:ovzeta_again} yields $\ov\zeta(x,y) = y$. From the definition \eqref{eq:ovGamma} of the set  $\ov\Gamma(\mu,\nu)$,  we see that the first condition, i.e. $\int \ov\zeta(x,y)\,\gamma_x(dy) = x$, is now the classical martingale condition $\gamma \in \Gamma_{\mathrm{M}}(\mu,\nu)$, whilst the second condition is void: $\int \ov\zeta(x,y)\,\gamma_y(dx) = \int y \,\gamma_y(dx) = y$ for any kernel $\gamma_y(dx)$. We have thus showed that $\ov\Gamma(\mu,\nu) = \Gamma_{\mathrm{M}}(\mu,\nu)$, and the proof is complete.
\end{proof}

\subsection{Zolotarev projection onto the cone of dominating probabilities}
\label{ssec:Zol_proj}

For any $\mu \in \mathcal{P}^\ba_2(\Rd)$ we define the \textit{Zolotarev projection}, being a set-valued operator
$\Pi_{\succeq_c \mu}:\mathcal{P}_2^\ba(\Rd) \to 2^{\mathcal{P}_2^\ba(\Rd)}$,
\begin{equation}
		\label{eq:Zol_projection}
		\Pi_{\succeq_c \mu} (\nu) :=  \argmin\limits_{\rho \succeq_c \mu} Z_2(\nu,\rho).
\end{equation}
The well-posedness of the Zolotarev projection can be shown directly, and here is the outline of how to do it. If $\rho_n$ is a minimizing sequence, then the estimate $2Z_2(\nu,\rho_n) \geq m_2(\rho_n) - m_2(\nu_n)$ yields tightness of the sequence $\{\abs{\argu} \rho_n\}$. This is enough to ensure the stability of $\rho_n \succeq_c \nu$ for the weak limit points $\ov\rho$, see  e.g. \cite[Lemma A.1]{alfonsi2020sampling}. The same tightness also allows to adapt the proofs of Section \ref{ssec:existence} to obtain the lower semi-continuity $\liminf_{n} Z_2(\nu,\rho_n) \geq  Z_2(\nu,\ov\rho)$, thus proving that $\ov\rho \in \Pi_{\succeq_c \mu} (\nu) $.

We can skip the details of this reasoning since, here, the non-emptiness of $\Pi_{\succeq_c \mu} (\nu)$ will emerge as the consequence of an unexpected feature of the Zolotarev projection: swapping the roles of $\mu$ and $\nu$ does not change the resulting set, and this set equals to the common convex dominants of minimal second moment\footnote{The choice for the notation  $\mu \curlyvee \nu$ is to acknowledge that on the real line this set is equal to $ \{\mu \vee\nu \}$. It must be emphasized that convex order in $d>1$ does not enjoy the structure of semilattice, hence $\mu \curlyvee \nu$ is not the least upper bound.},
\begin{equation*}
	\mu \curlyvee \nu := \argmin_{ \rho \succeq_c \mu, \  \rho \succeq_c \nu} m_2(\rho),
\end{equation*}
which is a non-empty set.
The main step towards connecting the Zoloterev projection and optimal quadratic convex dominance consists in the following observation,
\begin{equation}
	\label{eq:triangle_eq}
Z_2(\mu,\nu) = Z_2(\mu,\ov\rho) + Z_2(\ov\rho,\nu) \qquad \qquad  \forall\, \ov{\rho} \in  \mu \curlyvee \nu.
\end{equation}
To see it, we first exploit the dominances  $\ov\rho \succeq_c \mu$, $\ov\rho \succeq_c \nu$, which, by virtue of Proposition \ref{prop:cvx_order_char},  furnish the equalities
$Z_2(\mu,\ov\rho)  = \frac{1}{2} \big( m_2(\ov\rho) - m_2(\mu)\big)$ and $Z_2(\nu,\ov\rho)  = \frac{1}{2} \big( m_2(\ov\rho) - m_2(\nu)\big)$. By definition of the set $\mu \curlyvee\nu$ we have $m_2(\ov\rho) = \mathcal{C}(\mu,\nu)$, and so the desired identity follows from Theorem \ref{thm:RK}.

The following result completes the proof of Theorem \ref{thm:RK_reinvented}.

%\begin{remark}
%	Note that the triangle inequality $\leq$ in \eqref{eq:triangle_eq} holds by the fact that $Z_2$ is a metric. Therefore, \eqref{eq:triangle_eq} means that $\ov{\rho}$ lies on the $Z_2$-geodesic connecting $\mu$ to $\nu$.
%\end{remark}

\begin{theorem}
	\label{thm:Zol_proj_again}
	For any  $\mu,\nu \in\mathcal{P}_2^\ba(\Rd)$ we have the equalities between the three sets,
	\begin{align*}
		\Pi_{\succeq_c \mu} (\nu)   = 	\Pi_{\succeq_c \nu} (\mu)  =  \mu \curlyvee \nu.
	\end{align*}
	In one dimension the projection set always consists of exactly one element being the least upper bound \eqref{eq:lub} for convex order:  $\Pi_{\succeq_c \mu} (\nu) = \{ \mu \vee \nu\}$. 
	In dimension $d>1$, $\Pi_{\succeq_c \mu} (\nu)$ is a non-empty convex set that may contain more than one element. 
	
	Moreover, for any $\ov\rho \in \mu \curlyvee \nu$ we have,
	\begin{align}
		\label{eq:min_dist}
		\min_{\rho \succeq_c \mu} Z_2(\nu,\rho) &= \frac{m_2(\ov\rho) - m_2(\nu)}{2} = \frac{Z_2(\mu,\nu) -\frac{1}{2}\big(m_2(\nu)-m_2(\mu)\big)}{2}, \\
		\label{eq:min_dist_2}
		\min_{\rho \succeq_c \nu} Z_2(\mu,\rho) &= \frac{m_2(\ov\rho) - m_2(\mu)}{2}  = \frac{Z_2(\mu,\nu) -\frac{1}{2}\big(m_2(\mu)-m_2(\nu)\big)}{2}.
	\end{align}
\end{theorem}

\begin{remark}
	The equalities \eqref{eq:proj_alpha} announced in the introduction, which involve the convex order index
	$\alpha_{\succeq_c}(\nu\,|\,\mu) =  \frac{m_2(\nu)-m_2(\mu)}{2 Z_2(\mu,\nu)}$, are now a simple corollary from the equalities \eqref{eq:min_dist}, \eqref{eq:min_dist_2}.
\end{remark}

\begin{proof}
	We will show the equalities $\Pi_{\succeq_c \mu} (\nu)  = \mu \curlyvee \nu$ and  \eqref{eq:min_dist}, and the other two will follow by symmetric arguments. The non-emptiness and convexity of $\Pi_{\succeq_c \mu} (\nu)$, as well as its multiplicity in many dimensions, will be then a consequence of Corollary \ref{cor:existence_dominant}, whilst the equality  $\Pi_{\succeq_c \mu} (\nu)   = \{ \mu \vee \nu\}$ in 1D will follow from Proposition \ref{prop:1D}.
	
	Let us choose a measure $\ov\rho \in \mu \curlyvee \nu$, whilst by $\rho $ denote any probability in $\mathcal{P}_2(\Rd)$ such that $\rho \succeq_c \mu$, i.e. $\rho$ is a competitor in \eqref{eq:Zol_projection}. Next, choose any optimal dominant $\breve \rho$ for the pair $\nu$,~$\rho$, that is $\breve \rho \in \nu \curlyvee \rho$. 
	By transitivity, we have,
	\begin{equation}
		\label{eq:transitivity}
		\breve{\rho} \succeq_c \nu, \qquad \quad \breve{\rho} \succeq_c {\rho} \succeq_c \mu.
	\end{equation}
	This renders $\breve \rho$ a competitor for the minimal common convex dominant for the pair $\mu$, $\nu$ as well. Since $\ov\rho$ was an optimal one, we must have $m_2(\breve\rho) \geq m_2(\ov\rho)$. 
	Kicking off with the identity \eqref{eq:triangle_eq} written for the pair $\nu, \rho$ and its optimal dominant $\breve\rho$, we are led to,
	\begin{align}
		\label{eq:chain_proj}
		Z_2(\nu,{\rho}) = Z_2(\nu,\breve\rho) + Z_2(\rho,\breve{\rho}) \geq Z_2(\nu,\breve{\rho}) = \tfrac{m_2(\breve{\rho}) - m_2(\nu)}{2} \geq  \tfrac{m_2(\ov{\rho}) - m_2(\nu)}{2} = Z_2(\nu,\ov{\rho}),
	\end{align}
	where we made use of the formula \eqref{eq:Z2_ordered} for the Zolotarev distance between ordered measures: first for $\nu \preceq_c \breve\rho$, and then for $\nu \preceq_c \ov\rho$. Since $\rho$ was an arbitrary competitor in the projection problem \eqref{eq:Zol_projection}, the above chain of inequalities proves that $\ov{\rho} \succeq_c \mu$ is a minimizer, i.e. $ \ov\rho \in \Pi_{\succeq_c \mu}(\nu)$.
	
	The above establishes the inclusion $\mu \curlyvee \nu \subset \Pi_{\succeq_c \mu}(\nu)$ and the first equality in \eqref{eq:min_dist}, whilst the second follows from Theorem \ref{thm:RK} and the fact that $m_2(\ov\rho) = \mathcal{C}(\mu,\nu)$.
	To get the converse inclusion, assume that $\rho \in \Pi_{\succeq_c \mu}(\nu)$. This is equivalent to saying that the two inequalities in \eqref{eq:chain_proj} are equalities. The second of those equalities implies that $m_2(\breve\rho) = m_2(\ov\rho)$, which means that $\breve\rho \in \mu \curlyvee \nu$ thanks to \eqref{eq:transitivity}. The first equality requires that  $Z_2(\rho,\breve{\rho}) = 0$, and so that $\rho = \breve\rho$. This furnishes the inclusion  $ \Pi_{\succeq_c \mu}(\nu) \subset \mu \curlyvee \nu$ and completes the proof.
\end{proof}

\section{Approximation of  the martingale optimal transport}
\label{sec:MOT}

This section focuses on employing the framework of bi-martingale optimal transport $(\mathrm{M^2 OT})$ to finding approximate solutions of the classical two-period martingale optimal transport problem \cite{beiglbock2013},
\begin{align*}
	\qquad\qquad\quad \inf\left\{ \iint_{(\Rd)^2}  c (x,y)\, \gamma(dxdy) \, : \, 	\gamma \in \Gamma_{\mathrm{M}}(\mu,\nu)   \right\}. \qquad\quad (\mathrm{MOT})
\end{align*}
In the whole section, $\mu$ and $\nu$ are elements of $\mathcal{P}^\ba_2(\Rd)$, i.e. they have finite second moments and a fixed barycentre $\ba \in \Rd$. \textit{A priori}, we do not impose that the two probabilities are in convex order. In addition, we  assume sequences  $\{\mu_n \}, \{\nu_n\} \subset \mathcal{P}^\ba_2(\Rd)$ that converge to, respectively, $\mu,\nu$ in the Zolotarev  metric  $Z_2$ (or equivalently in the Wasserstein $W_2$ metric).  For a continuous cost function $c:(\Rd)^2 \to \R$, we consider the sequence of bi-martingale optimal transport problems,
\begin{align*}
	\inf\left\{ \iint_{(\Rd)^2} \bigg( c (x,y) + \frac{1}{2\eps_n} \Big(\Big|\frac{dq}{d\gamma}(x,y)\Big|^2 - \mathcal{C}(\mu_n,\nu_n) \Big) \bigg) \gamma(dxdy) \, : \, 		(\gamma,q) \in \Gamma Q(\mu_n,\nu_n)    \right\} \\ (\mathrm{M^2 OT}_n)
\end{align*}
where $\eps_n$ is a sequence of positive numbers. In what follows, we show that $(\mathrm{M^2 OT}_n)$ $\Gamma$-converges provided that $\eps_n$ tends to zero at an adequate pace. In the case when $\mu \preceq_c \nu$, the $\Gamma$-limit problem becomes $(\mathrm{MOT})$ exactly. Otherwise, an interpretation involving  the Zolotarev projection is available.

\subsection{The $\Gamma$-convergence result}

Let $\Mes((\Rd)^2; \R \times \Rd)$ be the space of vector-valued Borel measures on $\Rd \times \Rd$ equipped with the topology of weak convergence. In order to facilitate a smooth $\Gamma$-convergence analysis, we convert $(\mathrm{M^2 OT}_n)$ to the unconstrained problem on that space, with the objective functional $F_n:\Mes((\Rd)^2; \R \times \Rd) \to \R \cup \{+\infty\}$ as follows, 
\begin{align*}
	F_n(\gamma,q) &:= \iint_{(\Rd)^2} \bigg( c  (x,y) + \frac{1}{2\eps_n} \Big(\Big|\frac{dq}{d\gamma}(x,y)\Big|^2 - \mathcal{C}(\mu_n,\nu_n) \Big) \bigg) \gamma(dxdy) + \chi_{ \Gamma Q(\mu_n,\nu_n)}(\gamma,q).
\end{align*}
The $\Gamma$-limit will be proven to equal to,
\begin{equation*}
	F(\gamma,q):= \iint_{(\Rd)^2}  c  (x,y) \,\gamma(dxdy) +  \chi_{\ov{\Gamma Q}(\mu,\nu)}(\gamma,q),
\end{equation*}
where $\ov{\Gamma Q}(\mu,\nu)$ is the set of pairs $(\gamma,q)$ that are optimal for the quadratic bi-martingale problem, see \eqref{eq:optGammaQ}. The proof of the $\Gamma$-convergence is moved to Sections \ref{ssec:liminf} and \ref{ssec:limsup}.

\begin{theorem}
	\label{thm:Gamma}
	Assume a continuous cost ${c} \in C((\Rd)^2)$, bounded from below and satisfying the $p$-growth condition \eqref{eq:quad_bound_cost} for $p<2$, and probabilities $\mu,\nu \in \mathcal{P}^\ba_2(\Rd)$  (not necessarily in convex order).
	
	Take any sequences $\mu_{n}, \nu_{n} \in \mathcal{P}^\ba_2(\Rd)$ converging   in Zolotarev metric $Z_2$ to $\mu, \nu$, respectively. Moreover, take a sequence $\eps_n >0$ that converges to zero and for which,
	\begin{equation}
		\label{eq:rate_epsn}
		\lim_{n \to +\infty} \frac{Z_2(\mu,\mu_{n})  + Z_2(\nu,\nu_{n})}{\eps_n}=0.
	\end{equation}
	Then, the sequence $F_n$\, $\Gamma$-converges to $F$ on the space $\Mes((\Rd)^2; \R \times \Rd)$.
\end{theorem}

%\begin{remark}
%	\label{rem:trivial_limsup}
%	The proof of the $\limsup$ inequality turns trivial in the case when $\mu_n \equiv \mu$, $\nu_n \equiv \nu$. Indeed, in this case the recovery sequence is simply $(\gamma_n,q_n) = (\gamma,q)$. Since  we have $	\iint |\frac{dq}{d\gamma}|^2  d\gamma = \mathcal{C}(\mu,\nu)$ whenever $F(\gamma,q) < + \infty$, there is simply $F_n(\gamma,q) = F(\gamma,q)$ in such a scenario.
%\end{remark}

%Under the condition of convex order $\mu \preceq_c \nu$, the characterization of $\ov{\Gamma Q}(\mu,\nu)$ in Proposition \ref{prop:cvx_order_char} readily allows to decode the problem $\min_{(\gamma,q)} F(\gamma,q)$ as the martingale optimal transport $(\mathrm{MOT})$. 

The characterization of $\ov{\Gamma Q}(\mu,\nu)$  in Proposition \ref{prop:ovGammaQ_char} allows to get rid of the variable $q$ in the minimization of $F$. For any solution  $\ov{u}$ of the Zolotarev problem \eqref{eq:Z2_again}, let the  coupling  map $\ov{\zeta}$ and the set $\ov\Gamma(\mu,\nu)$ read as in the definitions \eqref{eq:ovzeta_again} and \eqref{eq:ovGamma_again}. Then, minimization of $F$ is equivalent to,
\begin{align}
	\label{eq:MOT_zeta_again}
	\min\left\{ \iint_{(\Rd)^2}c (x,y)\, \gamma(dxdy) \, : \, 	\gamma \in \ov\Gamma(\mu,\nu)   \right\}.
\end{align}
Next, the set  $\ov\Gamma(\mu,\nu)$ is equal to the set of martingale plans $\Gamma_{\mathrm{M}}(\mu,\nu)$ once the data are ordered, $\mu \preceq_c \nu$, see  Proposition \ref{prop:cvx_order_char}.
Accordingly, with the compactness result for the sequence of minimizers of $F_n$ that is proven in Section \ref{ssec:liminf}, we are lead to one of  the central results of this paper.

\begin{corollary}
	\label{cor:approx_MOT}
	Under the prerequisites of the theorem,
	let $(\gamma_n,q_n)$ be a sequence of minimizers for the problems  $(\mathrm{M^2OT}_n)$, which always exists. Then, up to choosing a subsequence, $(\gamma_n,q_n) \rightharpoonup (\gamma,q) \in \Gamma Q(\mu,\nu)$ where $\gamma \in \ov\Gamma(\mu,\nu)$ solves the problem  $\eqref{eq:MOT_zeta_again}$, while $\min (\mathrm{M^2 OT}_n) \to  \min \eqref{eq:MOT_zeta_again}$
	
	In particular, if $\mu \preceq_c \nu$, then $\gamma \in \Gamma_{\mathrm{M}}(\mu,\nu)$ solves the martingale optimal transport problem  $(\mathrm{MOT})$, while $\min (\mathrm{M^2 OT}_n) \to  \min (\mathrm{MOT})$.
\end{corollary}

The set $\ov\Gamma(\mu,\nu)$ serves as a generalization of the set of martingale plans to the case where the data $\mu,\nu$ are not ordered. To give meaning of our $\Gamma$-convergence result in this wider scenario, a quantitative relation between the plans $\gamma \in\ov\Gamma(\mu,\nu)$ and the induced pair of martingale plans $\gamma_i = (\pi_i,\ov\zeta)\# \gamma$ should be elucidated. For instance, take  $\gamma_1 \in \Gamma_{\mathrm{M}}(\mu,\ov\rho)$ where, according to Theorem \ref{thm:RK_reinvented}, $\ov\rho = \ov\zeta \# \gamma \in \Pi_{\succeq_c \mu}(\nu)$ is a Zolotarev projection. The goal of the result below is to show that  discrepancy between $\gamma$ and $\gamma_1$ can be measured through the minimal Zolotarev distance between the respective second marginals $\nu$ and $\ov\rho$. Here, the discrepancy is understood as the mean $W_2$-squared error between the kernel transports computed with respect to the common first marginal $\mu$. Accordingly, if $\nu$ is close to dominate $\mu$ for convex order ($\alpha_{\succeq_c}(\nu \, | \, \mu)$ is close to one), then the admissible set $\ov\Gamma(\mu,\nu)$ consists of plans  which are close to being a martingale in the foregoing sense. A symmetric result holds true between $\gamma$ and $\gamma_2$.

\begin{proposition}
	\label{prop:dispcrepancies}
	For any  plan $\gamma \in \ov\Gamma(\mu,\nu)$ define $\ov\rho \in \Pi_{\succeq_c \mu}(\nu) = \Pi_{\succeq_c \nu}(\mu)$, $\gamma_1 \in \Gamma_{\mathrm{M}}(\mu,\ov\rho)$, and $\gamma_2 \in \Gamma_{\mathrm{M}}(\nu,\ov\rho)$ as above. Then,
\begin{align*}
	&\int_X W_2^2(\gamma_x,\gamma^1_x) \,\mu(dx)  =  \min_{\rho \succeq_c \mu}2 Z_2(\nu,\rho) = 2Z_2(\nu,\ov\rho), \\
	&	\int_Y W_2^2(\gamma_y,\gamma^2_y) \,\nu(dy)  =  \min_{\rho \succeq_c \nu}2 Z_2(\mu,\rho) = 2Z_2(\mu,\ov\rho),
\end{align*}
where the  transport kernels constitute the following disintegrations: $\gamma(dxdy) = \mu(dx) \otimes \gamma_{x}(dy)$,  $\gamma(dxdy) = \gamma_{y}(dx) \times \nu(dy)$, $\gamma_1(dxdz) = \mu(dx) \otimes \gamma^1_x(dz)$, and $\gamma_2(dydz) = \nu(dy) \otimes  \gamma^2_y(dz)$.
\end{proposition}

Thanks to the equality \eqref{eq:triangle_eq}, a neat consequence of this result is a formula for the distance $Z_2(\mu,\nu)$ that portrays it as the  mean of the two discrepancies: between any plan $\gamma \in \ov{\Gamma}(\mu,\nu)$ and the two induced martingale plans $\gamma_1$ and $\gamma_2$.

\begin{corollary}
	For any $\gamma \in \ov\Gamma(\mu,\nu)$ and $\gamma_i =(\pi_i,\ov\zeta)\# \gamma$, there holds the equality,
	\begin{equation*}
		Z_2(\mu,\nu) = \int_X \frac{1}{2} W_2^2(\gamma_x,\gamma^1_x) \,\mu(dx) + \int_Y\frac{1}{2} W_2^2(\gamma_y,\gamma^2_y) \,\nu(dy).
	\end{equation*}
\end{corollary}

\begin{proof}[Proof of Proposition \ref{prop:dispcrepancies}]
	For each $x \in \Rd$ define the map  $T_x:\Rd \to \Rd$,
\begin{equation*}
	T_x(y) = \ov\zeta(x,y) = \frac{x+\nabla \ov{u}(x)}{2} +  \frac{y-\nabla \ov{u}(y)}{2} = \nabla_y \bigg(\Big\langle  \frac{x+\nabla \ov{u}(x)}{2} ,y \Big\rangle +  \frac{1}{2}\bigg( \frac{\abs{y}^2}{2}  - \ov{u}(y)\bigg) \bigg).
\end{equation*}
By the admissibility of $\ov{u}$ in the Zolotarev problem \eqref{eq:Z2_again} we have  $\mathrm{lip}(\nabla \ov{u}) \leq 1$. Then, from \eqref{eq:u_adm_char} we deduce that the map $\frac{\abs{\argu}^2}{2}  - \ov{u}$ is convex. Therefore, $T_x$ is a cyclically monotone map for every $x$. Since it is also continuous (even 1-Lipschitz), Brenier theorem \cite{brenier1991} states that $T_x$ is an optimal  map for $L^2$-transport from $\gamma_x$ to	$\lambda_x = T_x \# \gamma_x$. To see that $\lambda_x = \gamma^1_x$ for $\mu$-a.e. $x$, for a bounded Borel function $\varphi:\Rd \times  \Rd \to \R$ we compute,
\begin{align*}
	\iint \varphi(x,z) \,\gamma_1(dxdz) = \iint \varphi\big(x,\ov\zeta(x,y)\big) \gamma(dxdy) &= \int \left( \int \varphi\big(x,T_x(y)\big) \gamma_x(dy) \right) \mu(dx)  \\
	&= \int \left( \int \varphi\big(x,z\big) \lambda_x(dz) \right) \mu(dx).
\end{align*}
Readily, the equality $\lambda_x = \gamma^1_x$ follows by the $\mu$-a.e.-uniqueness of the disintegration. The theorem of Brenier now furnishes,
\begin{align*}
	\int W_2^2(\gamma_x,\gamma^1_x) \,\mu(dx) = \int \left(\int\abs{T_x(y)-y}^2 \, \gamma_x(dy)\right)\mu(dx)  =  \iint  \abs{\ov\zeta(x,y)-y}^2\gamma(dxdy)&  \\
	 =  \iint \abs{\ov\zeta(x,y)}^2 \gamma(dxdy) - m_2(\nu)  = \mathcal{C}(\mu,\nu)-m_2(\nu) = m_2(\ov\rho) - m_2(\nu)&.
\end{align*}
To move to the second line we argue similarly as in \eqref{eq:mix_term}. Then, we use the identity \eqref{eq:Cmunu_ovzeta} in Proposition \ref{prop:ovGammaQ_char} and, subsequently, Proposition \ref{prop:rhozetagamma} for $f = \abs{\argu}^2$. Finally, we acknowledge \eqref{eq:min_dist} to arrive at the first asserted equality. The proof of the second is symmetric.
\end{proof}

\subsection{Proof of the lim\,inf inequality and compactness}
\label{ssec:liminf}

	In order to exploit the $\Gamma$-convergence result in Theorem \ref{thm:Gamma}, that is to pass to Corollary \ref{cor:approx_MOT}, we must additionally prove the weak precompactness of the sequence of minimizers $\{(\gamma_n,q_n)\}$. Only the precompactness of $\{q_n\}$ is non-trivial, and to show it we can argue similarly as in the proof of Proposition~\ref{prop:compactness}.

\begin{proof}[Proof of the compactness]
		Firstly, we need to show the tightness of  the sequence of positive measures $\{(1+ \abs{x} + \abs{y})\, \gamma_n\}$. 
		Owing to the convergences $\mu_n \to \mu$, $\nu_n \to \nu$ in the metric $Z_2$, we have convergence (and hence boundedness) of the second moment sequences $m_2(\mu_n)$, $m_2(\nu_n)$. This yields tightness of the sequences $\{\abs{x}\mu_n\}$, $\{\abs{y}\nu_n\}$. It translates to the tightness of $\{(1+ \abs{x} + \abs{y})\, \gamma_n\}$ thanks to the admissibility $\gamma_n \in \Gamma(\mu_n,\nu_n)$.
		
		Readily, the proof of  Proposition \ref{prop:compactness} can be reproduced if we are able to point to a superlinear function $h$ such that $\iint h(\frac{dq_n}{d\gamma_n}) \,d\gamma_n$ is equi-bounded. Since for any $(\gamma_n,q_n) \in \ov{\Gamma Q}(\mu_n,\nu_n) \subset \Gamma Q(\mu_n,\nu_n)$ we have $\iint \big|\frac{dq_n}{d\gamma_n}\big|^2   d\gamma_n = \mathcal{C}(\mu_n,\nu_n)$, we can first show equi-boundedness of the minimal costs, 
		\begin{align*}
			&F_n(\gamma_n,q_n) = \min\left\{ \iint \bigg( c + \frac{1}{2\eps_n} \Big(\Big|\frac{dq}{d\gamma}\Big|^2 - \mathcal{C}(\mu_n,\nu_n) \Big) \bigg) d\gamma \, : \, 		(\gamma,q) \in \Gamma Q(\mu_n,\nu_n)    \right\} \\
			& \quad \leq 	\inf\left\{ \iint c   \, d\gamma \, : \, 		(\gamma,q) \in \ov{\Gamma Q}(\mu_n,\nu_n)    \right\} \\
			&  \quad  \leq \inf\left\{ \iint \!D(1+ \abs{x}^p + \abs{y}^p)  \, d\gamma \, : \, 		(\gamma,q) \in \ov{\Gamma Q}(\mu_n,\nu_n)    \right\}  \\
			& \quad = D\big(1+m_p(\mu_n) + m_p(\nu_n)\big) \leq C_1 < +\infty,
		\end{align*}
		where we used the fact that $c$ has growth $p<2$. Multiplying the chain of inequalities by $\eps_n$ furnishes,
		\begin{align*}
			\iint \Big( \eps_n c + \frac{1}{2} \Big|\frac{dq_n}{d\gamma_n}\Big|^2 \Big)  d\gamma_n \leq \eps_n C_1 +\frac{\mathcal{C}(\mu_n,\nu_n)}{2} \leq C_2 < +\infty,
		\end{align*}
		where the boundedness of $\mathcal{C}(\mu_n,\nu_n)$ follows from \eqref{eq:CZ}. Since $c$ is assumed to be bounded from below, $\eps_n c(x,y) + \frac{1}{2}\abs{z}^2 \geq  \frac{1}{2}\abs{z}^2 - C_3$ for some finite constant $C_3$.	We can choose $h(z): = \frac{1}{2}\abs{z}^2 - C_3$ and proceed with proving precompactness of $\{q_n\}$ as in Proposition \ref{prop:compactness}.
\end{proof}

The proof of the $\liminf$ condition is fairly straightforward as it mainly uses the lower semi-continuity of the total cost functional in the $(\mathrm{M^2 OT})$ formulation.

\begin{proof}[Proof of the lim\,inf inequality]
	Take $(\gamma,q) \in \Mes((\Rd)^2; \R \times \Rd)$ and a  weakly converging sequence $\Mes((\Rd)^2; \R \times \Rd) \ni (\gamma_n,q_n) \rightharpoonup (\gamma,q)$. We can assume that  $(\gamma_n,q_n) \in \Gamma Q(\mu_n,\nu_n)$ and $q_n \ll \gamma_n$, since otherwise we can either work on a subsequence or the $\liminf$ inequality is trivial. The characterization \eqref{eq:Z2_convergence} of convergence $\mu_n \to \mu$ in the metric $Z_2$ guarantees that $\int \varphi\, d\mu_n$ and $\int \pairing{\Phi,x}\, d\mu_n$ converge to $\int \varphi\, d\mu$ and $\int \pairing{\Phi,x}\, d\mu$, respectively, for any $(\varphi,\Phi) \in C_b(\Rd;\R\times\Rd)$, and the same holds for $\nu_n,\nu$.  Accordingly, the weak convergences of $\gamma_n$ and $q_n$ furnishes $\gamma \in \Gamma(\mu,\nu)$ and $q \in Q(\mu,\nu)$, see \eqref{eq:moment-martingale}. In order to conclude that $(\gamma,q) \in \Gamma Q(\mu,\nu)$, the absolute continuity $q \ll \gamma$ will be shown in the sequel.
	
	Upon fixing  $\eps>0$, we use the fact that $(\gamma_n,q_n) \in \Gamma Q(\mu_n,\nu_n)$ to show the following chain of inequalities,
	\begin{align*}
		\liminf_{n \to +\infty} & F_n(\gamma_n,q_n) =  	\liminf_{n \to +\infty}   \iint \bigg( c  (x,y) + \frac{1}{2\eps_n} \Big(\Big|\frac{dq_n}{d\gamma_n}(x,y)\Big|^2 - \mathcal{C}(\mu_n,\nu_n) \Big) \bigg) \gamma_n(dxdy) \\
		& \geq  	\liminf_{n \to +\infty}   \iint \bigg( c  (x,y) + \frac{1}{2\eps} \Big(\Big|\frac{dq_n}{d\gamma_n}(x,y)\Big|^2 - \mathcal{C}(\mu_n,\nu_n) \Big) \bigg) \gamma_n(dxdy) \\
		&\geq  	\liminf_{n \to +\infty}   \iint \bigg( c  (x,y) + \frac{1}{2\eps} \Big(\Big|\frac{dq_n}{d\gamma_n}(x,y)\Big|^2 - \mathcal{C}(\mu,\nu) \Big) \bigg) \gamma_n(dxdy)  \\
		& \qquad + \liminf_{n \to +\infty} \frac{\mathcal{C}(\mu,\nu) -\mathcal{C}(\mu_n,\nu_n) }{2\eps}\\
		& = 	\liminf_{n \to +\infty}   \iint \bigg( c  (x,y) + \frac{1}{2\eps} \Big(\Big|\frac{dq_n}{d\gamma_n}(x,y)\Big|^2 - \mathcal{C}(\mu,\nu) \Big) \bigg) \gamma_n(dxdy) \\
		& \geq 
		 \begin{cases}
		 \iint \Big( c  (x,y) + \frac{1}{2\eps} \big(\big|\frac{dq}{d\gamma}(x,y)\big|^2 - \mathcal{C}(\mu,\nu) \big) \Big) \gamma(dxdy)  & \text{if \ \  $\gamma \geq 0$ and  $q \ll \gamma$,} \\
			+ \infty & \text{otherwise}.
		\end{cases}
	\end{align*}
	The arguments behind the subsequent inequalities are as follows:
	\begin{itemize}
		\item[-] to pass to the second line we recall that $	\iint\big|\frac{dq_n}{d\gamma_n}(x,y)\big|^2  \gamma_n(dxdy) \geq \mathcal{C}(\mu_n,\nu_n)$,  see  \eqref{eq:C_gammaq};
		\item[-] to pass to the third line we simply acknowledge that $\gamma_n$ are probabilities;
		\item[-] to pass to the forth line we acknowledge the formula \eqref{eq:CZ} to see that
		\begin{align*}
			\mathcal{C}(\mu,\nu) - \mathcal{C}(\mu_n,\nu_n) = Z_2(\mu,\nu) - Z_2(\mu_n,\nu_n) +\tfrac{m_2(\mu)-m_2(\mu_n)  }{2} +\tfrac{m_2(\nu)-m_2(\nu_n) }{2} 
		\end{align*}
		which converges to zero since $\mu_n,\nu_n$ converge to $\mu,\nu$ in the $Z_2$ metric, cf. \eqref{eq:Z2_convergence};
		\item[-] to pass to the last line we use  the lower semi-continuity result in Proposition \ref{prop:lsc}.
	\end{itemize}
	Since the case when $\liminf_{n \to +\infty}  F_n(\gamma_n,q_n) < +\infty$ is the only non-trivial scenario, from the chain of inequalities above we can deduce that indeed $q \ll \gamma$. This establishes that $(\gamma,q) \in \Gamma Q(\mu,\nu)$. With that information at hand, we know that $\iint\big|\frac{dq}{d\gamma}(x,y)\big|^2  \gamma(dxdy) \geq \mathcal{C}(\mu,\nu)$, and sending $\eps$ to zero gives,
	\begin{equation*}
		\liminf_{n \to +\infty}  F_n(\gamma_n,q_n) \geq \iint c  (x,y)  \, \gamma(dxdy)  + \begin{cases}
			0 & \text{if } \iint\big|\frac{dq}{d\gamma}(x,y)\big|^2  \gamma(dxdy) = \mathcal{C}(\mu,\nu), \\
			+\infty & \text{otherwise}.
		\end{cases}
	\end{equation*}
	Considering the definition of the set $\ov{\Gamma Q}(\mu,\nu)$, see \eqref{eq:optGammaQ}, we recognize the right hand side above as $F(\gamma,q)$ exactly. This concludes the proof of the $\liminf$ inequality.
\end{proof}

\subsection{Proof of the lim\,sup inequality}
\label{ssec:limsup}
Let us now construct a recovery sequence $(\gamma_n,q_n)$ for a pair $(\gamma,q)$. The task is trivial unless $F(\gamma,q) < +\infty$, thus we can assume that $(\gamma,q) \in \ov{\Gamma Q}(\mu,\nu)$, cf. the definition \eqref{eq:optGammaQ}.
In particular, $(\gamma,q) \in \Gamma Q(\mu,\nu)$ and, 
\begin{equation}
	\label{eq:opt_gammaq}
	\iint |\zeta(x,y)|^2  \gamma(dxdy) = \mathcal{C}(\mu,\nu), \qquad  \text{where}\quad \zeta = \frac{dq}{d\gamma}.
\end{equation}
For each $n$ take any pairs $(\hat\gamma_n,\hat{q}_n) \in \ov{\Gamma Q}(\mu,\mu_n)$ and $(\check\gamma_n,\check{q}_n) \in \ov{\Gamma Q}(\nu,\nu_n)$. Namely,  $(\hat\gamma_n,\hat{q}_n) \in \Gamma Q(\mu,\mu_n)$ and $(\check\gamma_n,\check{q}_n) \in \Gamma Q(\nu,\nu_n)$, while
\begin{align}
	\label{eq:Z2n_a}
	&\iint \frac{1}{2}\Big( \big| \hat{\zeta}_n(x,\xi)-x\big|^2 +\big|\hat{\zeta}_n(x,\xi)-\xi\big|^2 \Big) \hat\gamma_n(dxd\xi) = 	Z_2(\mu,\mu_n) , \qquad   \hat\zeta_n = \frac{d\hat{q}_n}{d\hat\gamma_n}, \\
	\label{eq:Z2n_b}
	& \iint \frac{1}{2}\Big( \big| \check{\zeta}_n(y,\eta)-y\big|^2 +\big|\check{\zeta}_n(y,\eta)-\eta\big|^2 \Big) \check\gamma_n(dyd\eta) = 	Z_2(\nu,\nu_n) , \qquad  \check\zeta_n = \frac{d\check{q}_n}{d\check\gamma_n}.
\end{align}
We will now perform a gluing of the measures $\gamma$, $\hat\gamma_n$, and $\check\gamma_n$. To that end, let us perform the disintegrations,
\begin{equation*}
	\hat{\gamma}_n(dxd\xi) = \mu(dx) \otimes \hat\gamma_{x}^n(d\xi), \qquad 	\check{\gamma}_n(dyd\eta) = \nu(dy) \otimes \check\gamma_{y}^n(d\eta),
\end{equation*}
where $x \mapsto \hat\gamma_{x}^n$ and $y \mapsto \check\gamma_{y}^n$ are Borel measurable measure-valued maps. The gluing is now achieved using the generalized product measure,
\begin{equation*}
	\lambda_n(dxd\xi dy d\eta) := \gamma(dxdy) \otimes \big(\hat\gamma_{x}^n \otimes \check\gamma_{y}^n \big)(d\xi d\eta).
\end{equation*}
Its marginals satisfy $\pi_{1,3} \# \lambda_n = \gamma$, $\pi_{1,2} \# \lambda_n = \hat\gamma_n$, $\pi_{3,4} \# \lambda_n = \check\gamma_n$, whilst $\pi_{2,4} \# \lambda_n$ belongs to $\Gamma(\mu_n,\nu_n)$. The latter is thus a candidate for the recovery sequence,
\begin{equation*}
	\gamma_n(d\xi d\eta) := \pi_{2,4} \# 	\lambda_n(dxd\xi dy d\eta).
\end{equation*}
In order to build $q_n$, we first define the map $z_n: (\Rd)^4 \to \Rd$,
\begin{equation*}
	z_n(x,\xi,y,\eta) := \zeta(x,y) + \hat{\zeta}_n(x,\xi) +\check{\zeta}_n(y,\eta) - x -y.
\end{equation*}
Finally, with $(\xi,\eta) \mapsto \lambda^n_{\xi,\eta}$ being the Borel map in the disintegration $\lambda_n = \gamma_n \otimes \lambda^n_{\xi,\eta}$, we define
\begin{align*}
	q_n := \zeta_n \gamma_n, \qquad \text{where} \quad  \zeta_n(\xi,\eta):= \iint z_n(x,\xi,y,\eta) \,\lambda^n_{\xi,\eta}(dxdy).
\end{align*}
In the next two lemmas we show that the sequence $(\gamma_n,q_n)$ is  admissible, i.e. $F_n(\gamma_n,q_n) < +\infty$, and that it  weakly converges to $(\gamma,q)$.

\begin{lemma}
	For every $n$ there holds $(\gamma_n,q_n) \in \Gamma Q(\mu_n,\nu_n)$.
\end{lemma}
\begin{proof}
	We know  that $\gamma_n \in \Gamma(\mu_n,\nu_n)$, but the condition $q_n \in Q(\mu_n,\nu_n)$ must be verified. Taking a function $\Phi \in C_b(\Rd;\Rd)$, we check the first marginal. Observe that, since the Borel functions $\zeta,\zeta_n, \hat{\zeta}_n, \check\zeta_n$ are square-integrable with respect to $\lambda_n$, all the integrals below are well defined and finite,
	\begin{align*}
		\iint \big\langle\Phi(\xi) ,q_n(d\xi d\eta)\big\rangle &=	\iint \big\langle\Phi(\xi) , \zeta_n(\xi,\eta) \big\rangle\, \gamma_n(d\xi d\eta)\\
		& = \iint \Big\langle\Phi(\xi) , \iint z_n(x,\xi,y,\eta) \,\lambda^n_{\xi,\eta}(dxdy)\Big\rangle \gamma_n(d\xi d\eta) \\
		& =  	\iiiint \big\langle\Phi(\xi) , \zeta(x,y) + \hat{\zeta}_n(x,\xi) +\check{\zeta}_n(y,\eta) - x-y\big\rangle\, \lambda_n(dxd\xi dy d\eta) \\
		& =\iint \big\langle \Phi(\xi), \hat{\zeta}_n(x,\xi)  \big\rangle \, \hat\gamma_n(dxd\xi)   + \iint\big\langle  \hat{\Phi}_n(x) , \zeta(x,y) -x \big\rangle \, \gamma(dxdy) \\
		& \qquad + \iiint \big\langle \hat{\Phi}_n(x) , \check\zeta_n(y,\eta) -y \big\rangle \, \check{\gamma}_y^n(d\eta)\, \gamma_y(dx) \, \nu(dy) \\
		&= \iint \big\langle \Phi(\xi), \hat{q}_n(dxd\xi)  \big\rangle   + \iint\big\langle  \hat{\Phi}_n(x) , \zeta(x,y) -x \big\rangle \, \gamma(dxdy) \\
		& \qquad + \iint \big\langle \tilde{\Phi}_n(y) , \check\zeta_n(y,\eta) -y \big\rangle \, \check{\gamma}_n(dyd\eta)\\ 
		&= \int \big\langle \Phi(\xi), \xi  \big\rangle \, \mu_n(d\xi),
	\end{align*}
	where  we used the disintegration $\gamma(dxdy) = \gamma_y(dx) \otimes \nu(dy)$ and we introduced the bounded Borel measurable  vector functions  $\hat{\Phi}_n(x) =  \int \Phi(\xi)\,\hat{\gamma}_x^n(d\xi)$ and   $\tilde{\Phi}_n(y) =  \int \hat{\Phi}_n(x)\, {\gamma}_y(dx)$. The last equality acknowledges that: $\hat{\zeta}_n \hat{\gamma}_n =\hat{q}_n \in Q(\mu,\mu_n)$, \ $\zeta \gamma = q \in Q(\mu,\nu)$, and $\check{\zeta}_n \check{\gamma}_n = \check{q}_n \in Q(\nu,\nu_n)$, cf. \eqref{eq:moment-martingale}. Checking the second marginal is similar.
\end{proof}

\begin{lemma}
	The sequences $\gamma_n, q_n$ weakly converge to $\gamma, q$, respectively.
\end{lemma}
\begin{proof}
		It is clear that sequences $\hat{\gamma}_n$ and $\check{\gamma}_n$ converge weakly to $(x,x) \# \mu$ and $(y,y)\# \nu$, but we need a stronger result. We shall prove that we can always extract a subsequence $(\hat{\gamma}_{n_k},\check{\gamma}_{n_k})$ such that,
	\begin{equation}
		\label{eq:lim_delta}
		\hat\gamma_{x}^{n_k} \otimes \check\gamma_{y}^{n_k}  \rightharpoonup \delta_{(x,y)} \qquad \text{for $\gamma$-a.e. $(x,y)$}.
	\end{equation}
	To that end, we use  \eqref{eq:Z2n_a},  \eqref{eq:Z2n_b}. By convexity we have $\abs{c-a}^2 +\abs{c-b}^2 \geq \frac{1}{2} \abs{a-b}^2$,  hence
	\begin{align*}
		Z_2(\mu,\mu_{n}) + Z_2(\nu,\nu_{n}) & \geq \iint \frac{1}{4} | x-\xi|^2   \hat\gamma_{n}(dxd\xi) +   \iint \frac{1}{4} |y-\eta|^2   \check\gamma_{n}(dyd\eta) \\
		& = \frac{1}{4} \iiiint  \Big(  | x-\xi|^2 + |y-\eta|^2  \Big) \lambda_{n}(dxd\xi dy d\eta) \\
		& = \frac{1}{4} \iint \left( \iint  \Big(  | x-\xi|^2 + |y-\eta|^2  \Big) (\hat\gamma_{x}^{n} \otimes \check\gamma_{y}^{n}) (d\xi d\eta) \right)  \gamma(dxdy)
	\end{align*}
	Since 	$Z_2(\mu,\mu_{n})$ and  $Z_2(\nu,\nu_{n})$ converge to zero, the function $(x,y) \mapsto \iint  \big(  | x-\xi|^2 + |y-\eta|^2  \big) \,(\hat\gamma_{x}^{n} \otimes \check\gamma_{y}^{n})(d\xi d\eta)$ converges to zero in $L^1_\gamma(\Rd \times \Rd)$. Thus, we can select a sequence of indices $n_k$ for which the latter integral converges to zero for $\gamma$-a.e. $(x,y)$. This means that $W_2(\hat\gamma_{x}^{n_k} \otimes \check\gamma_{y}^{n_k},\delta_{(x,y)}) \to 0$ in the space $\mathcal{P}_2((\Rd)^2)$ for $\gamma$-a.e. $(x,y)$,  and so, in particular, \eqref{eq:lim_delta}~holds true. With this convergence at hand, the fact that $\gamma_{n_k} \rightharpoonup \gamma$ follows easily. Indeed, for a function $\varphi \in C_b(\Rd \times \Rd;\R)$, we have,
	\begin{align*}
		&\iint \varphi(\xi,\eta)  \gamma_{n_k}(d\xi d\eta) = \iiiint \varphi(\xi,\eta)  \lambda_{n_k}(dxd\xi dy d\eta) \\
		& \qquad \qquad  = 	\iint \left( \iint\varphi(\xi,\eta) \,(\hat\gamma_{x}^{n_k} \otimes \check\gamma_{y}^{n_k}) (d\xi d\eta) \right)\! \gamma(dxdy)  \xrightarrow[k \to \infty]{} \	\iint \varphi(x,y)  \gamma(dx dy),
	\end{align*}
	where we used \eqref{eq:lim_delta} and the dominated convergence theorem. Next, for a vector function $\Phi \in C_b(\Rd \times \Rd;\Rd)$ we have,
	\begin{align*}
		\iint \big\langle\Phi(\xi,\eta),q_{n_k}(d\xi d\eta)\big\rangle  &=  	\iiiint \big\langle\Phi(\xi,\eta) , \zeta(x,y) + \hat{\zeta}_{n_k}(x,\xi) +\check{\zeta}_{n_k}(y,\eta) - x-y\big\rangle\, d\lambda_{n_k} \\
		& = \iint \bigg\langle \iint  \Phi(\xi,\eta) \,(\hat\gamma_{x}^{n_k} \otimes \check\gamma_{y}^{n_k}) (d\xi d\eta) \, , \, \zeta(x,y) \bigg\rangle    \gamma(dxdy)\\
		& \qquad +  \iiiint \big\langle\Phi(\xi,\eta), \hat{\zeta}_{n_k}(x,\xi) - x \big\rangle \lambda_{n_k}(dxd\xi dy d\eta) \\
		& \qquad \quad + \iiiint \big\langle\Phi(\xi,\eta), \check{\zeta}_{n_k}(y,\eta) - y \big\rangle \lambda_{n_k}(dxd\xi dy d\eta).
	\end{align*}
	Using \eqref{eq:lim_delta} again, we show that the first integral after the last equality sign converges to $\iint \big\langle\Phi(x,y),\zeta(x,y) \big\rangle \gamma(dxdy)$,  being equal to $\iint \big\langle\Phi(x,y), q(dxdy)\big\rangle$. Meanwhile, using \eqref{eq:Z2n_a} we estimate the second integral, starting with the Cauchy--Schwarz inequality, 
	\begin{align*}
		&\left| \iiiint \big\langle\Phi(\xi,\eta), \hat{\zeta}_{n_k}(x,\xi) - x \big\rangle \lambda_{n_k}(dxd\xi dy d\eta) \right| \\
		& \qquad  \leq  \left(\iiiint \abs{\Phi(\xi,\eta)}^2 \lambda_{n_k}(dxd\xi dy d\eta)\right)^{\frac{1}{2}} \left( \iint \big|\hat{\zeta}_{n_k}(x,\xi) - x \big|^2 \hat{\gamma}_{n_k}(dx d\xi )\right)^{\frac{1}{2}} \\
		& \qquad  \leq \sqrt{2} \, \norm{\Phi     }_{L^\infty} \Big(Z_2(\mu,\mu_{n_k})\Big)^{\frac{1}{2}},
	\end{align*}
	Similarly, the third integral can be bounded by $\sqrt{2} \, \norm{\Phi}_{L^\infty} (Z_2(\nu,\nu_{n_k}))^{\frac{1}{2}}$. Since $Z_2(\mu,\mu_{n_k}) \to 0$ and $Z_2(\nu,\nu_{n_k}) \to 0$, this establishes the convergence $q_{n_k} \rightharpoonup q$.
	
	We have showed that the desired convergence holds on a subsequence. However, a convergent sub-subsequence could be extracted from any subsequence of $(\gamma_n,q_n)$. Since the limits are always the same and equal $(\gamma,q)$, the convergence must hold for the whole sequence as well. 
\end{proof}

The $\limsup$ condition can be readily proven.

\begin{proof}[Proof of the lim\,sup inequality]
	For every $n$ we have,
	\begin{align*}
		F_n(\gamma_n,q_n) &=  \iint\bigg( c (\xi,\eta) + \frac{1}{2\eps_n} \Big(\big|\zeta_n(\xi,\eta)\big|^2 - \mathcal{C}(\mu_n,\nu_n) \Big) \bigg) \gamma_n(d\xi d\eta) \\
		&=  \iint\bigg( c (\xi,\eta) + \frac{1}{2\eps_n} \left|\iint z_n(x,\xi,y,\eta) \, \lambda_{\xi,\eta}^n(dxdy) \right|^2  \bigg) \gamma_n(d\xi d\eta)  -  \frac{1}{2\eps_n}  \mathcal{C}(\mu_n,\nu_n) \\
		& \leq \iiiint\bigg( c (\xi,\eta) + \frac{1}{2\eps_n} \big| z_n(x,\xi,y,\eta)  \big|^2 \bigg) \lambda_n(dx d\xi dy d\eta) -  \frac{1}{2\eps_n}  \mathcal{C}(\mu_n,\nu_n),
	\end{align*}
	where to pass to the last line we used Jensen's inequality. To continue, we expand the formula for $z_n$, but first we also make use of the formula \eqref{eq:CZ}  and the fact that $Z_2$ is a metric in order to make the following estimate,
	\begin{align*}
		\mathcal{C}(\mu_n,\nu_n) &= Z_2(\mu_n,\nu_n) + \tfrac{m_2(\mu_n) +m_2(\nu_n)}{2} \\
		&\geq Z_2(\mu,\nu) - Z_2(\mu,\mu_n) - Z_2(\nu,\nu_n) +\tfrac{m_2(\mu_n) +m_2(\nu_n)}{2}  \\
		&\geq \mathcal{C}(\mu,\nu) - Z_2(\mu,\mu_n) - Z_2(\nu,\nu_n)  - \tfrac{m_2(\mu) -m_2(\mu_n)}{2} - \tfrac{m_2(\nu) -m_2(\nu_n)}{2}.
	\end{align*}
	Carrying on, we obtain,
	\begin{align*}
		&F_n(\gamma_n,q_n) \leq  \iiiint\bigg( c (\xi,\eta) + \frac{1}{2\eps_n} \Big(\Big|  \zeta(x,y) + \big(\hat{\zeta}_n(x,\xi) -x\big) + \big(\check{\zeta}_n(y,\eta) -  y\big)  \Big|^2  \Big) \bigg) d\lambda_n \\
		& \qquad \qquad \qquad+ \frac{1}{2\eps_n} \bigg[ - \mathcal{C}(\mu,\nu) + Z_2(\mu,\mu_n) + Z_2(\nu,\nu_n) +\tfrac{m_2(\mu) -m_2(\mu_n)}{2} + \tfrac{m_2(\nu) -m_2(\nu_n)}{2}  \bigg] \\
		& \qquad = \iint c(\xi,\eta) \,\gamma_n(d\xi d\eta) + \frac{1}{2\eps_n} \bigg[ \iint \abs{\zeta(x,y)}^2 \gamma(dxdy) + \iint |\hat\zeta_n(x,\xi)-x|^2 \hat\gamma_n(dxd\xi)    \\
		& \qquad  \quad + \iint |\check\zeta_n(y,\eta)-y|^2 \check\gamma_n(dyd\eta) + 2 \iiiint \big\langle\zeta(x,y),\hat\zeta_n(x,\xi)-x \big\rangle \lambda_n(dxd\xi dy d\eta)   \\
		&  \qquad  \quad \quad + 2 \iiiint \big\langle\zeta(x,y),\check\zeta_n(y,\eta)-y \big\rangle \lambda_n(dxd\xi dy d\eta)  \\
		&  \qquad  \quad  \quad \quad+  2 \iiiint \big\langle\hat\zeta_n(x,\xi)-x,\check\zeta_n(y,\eta)-y \big\rangle \lambda_n(dxd\xi dy d\eta)\\
		&  \qquad  \quad \quad \quad \quad  - \mathcal{C}(\mu,\nu) + Z_2(\mu,\mu_n) + Z_2(\nu,\nu_n) +\tfrac{m_2(\mu) -m_2(\mu_n)}{2} + \tfrac{m_2(\nu) -m_2(\nu_n)}{2}  \bigg].
	\end{align*}
	Let us now compute or bound the integrals in the square brackets. From \eqref{eq:opt_gammaq} we know that $\iint \abs{\zeta(x,y)}^2 \gamma(dxdy)  = \mathcal{C}(\mu,\nu)$, while, following the steps of \eqref{eq:mix_term}, we obtain by optimality of $(\hat{\gamma}_n,\hat{\zeta}_n\hat{\gamma}_n)$ and $(\check{\gamma}_n,\check{\zeta}_n\check{\gamma}_n)$,
	\begin{align*}
		&\iint |\hat\zeta_n(x,\xi)-x|^2 \hat\gamma_n(dxd\xi) = \mathcal{C}(\mu,\mu_n) - m_2(\mu) = Z_2(\mu,\mu_n) + \tfrac{m_2(\mu_n) - m_2(\mu)}{2}, \\
		&  \iint |\check\zeta_n(y,\eta)-y|^2 \check\gamma_n(dyd\eta) =  \mathcal{C}(\nu,\nu_n) - m_2(\nu) = Z_2(\nu,\nu_n) + \tfrac{m_2(\nu_n) - m_2(\nu)}{2},
	\end{align*}
	where we also used the formula \eqref{eq:CZ}.
	Next, we find that,
	\begin{align*}
		&\iiiint \!\! \big\langle\zeta(x,y),\hat\zeta_n(x,\xi)-x \big\rangle d\lambda_n  =\!\! 	\iiiint\!\!\big\langle\zeta(x,y),\hat\zeta_n(x,\xi)-x \big\rangle \check\gamma^n_y(d\eta) \,\hat\gamma^n_x(d\xi) \gamma(dxdy) \\
		&= \!\!	\iiint\!\! \big\langle\zeta(x,y),\hat\zeta_n(x,\xi)-x \big\rangle \hat\gamma^n_x(d\xi) \gamma(dxdy) = \!\!	\iiint\!\! \big\langle\zeta(x,y),\hat\zeta_n(x,\xi)-x \big\rangle \hat\gamma^n_x(d\xi) \gamma_x(dy) \mu(dx) \\
		&  = \!	\iint \! \big\langle x,\hat\zeta_n(x,\xi)-x \big\rangle 
		\hat\gamma^n_x(d\xi)  \mu(dx) = \! \iint \! \big\langle x,\hat\zeta_n(x,\xi)-x \big\rangle \hat{\gamma}_n(dxd\xi) = 0.
	\end{align*}
	Above we used the disintegration $\gamma(dxdy) = \mu(dx) \otimes \gamma_x(dy)$, and then we passed to the last line by recalling that $\int \zeta(x,y) \, \gamma_x(dy) = x$, see \eqref{eq:2xmartingale_condition}.
	The last equality is by the fact that $\hat{\zeta}_n \hat{\gamma}_n \in Q(\mu,\mu_n)$, see \eqref{eq:moment-martingale}.  By a similar token we get,
	\begin{equation*}
		\iiiint \big\langle\zeta(x,y),\check\zeta_n(y,\eta)-y \big\rangle \lambda_n(dxd\xi dy d\eta) = 0.
	\end{equation*}
	Finally, starting with  Cauchy-Schwarz inequality and then using \eqref{eq:Z2n_a}, \eqref{eq:Z2n_b}, we estimate,
	\begin{align*}
		&\abs{\iiiint \big\langle\hat\zeta_n(x,\xi)-x,\check\zeta_n(y,\eta)-y \big\rangle \,d\lambda_n}  \\
		& \qquad \leq 	\left(\iiiint \big|\hat\zeta_n(x,\xi)-x\big|^2\,d\lambda_n \right)^{\frac{1}{2}} \left( \iiiint \big|\check\zeta_n(y,\eta)-y\big|^2\,d\lambda_n \right)^{\frac{1}{2}}  \\
		& \qquad =	2\left(\iint \frac{1}{2}\big|\hat\zeta_n(x,\xi)-x\big|^2\,d\hat{\gamma}_n \right)^{\frac{1}{2}} \left( \iint \frac{1}{2} \big|\check\zeta_n(y,\eta)-y\big|^2\,d\check{\gamma}_n \right)^{\frac{1}{2}}  \\
		&\qquad\leq 2 \big(Z_2(\mu,\mu_n) \big)^{\frac{1}{2}} \big(Z_2(\nu,\nu_n) \big)^{\frac{1}{2}} \leq Z_2(\mu,\mu_n) +Z_2(\nu,\nu_n),
	\end{align*}
	where at the end we used Young's inequality. Upon plugging the above equalities and estimates into the chain of inequalities, every term in the bracket, apart from $Z_2(\mu,\mu_n), \, Z_2(\nu,\nu_n)$, cancels, which allows to conclude that,
	\begin{equation*}
		\limsup_{n \to +\infty} F_n(\gamma_n,q_n) \leq   \limsup_{n \to +\infty} \left\{\iint c(\xi,\eta) \,\gamma_n(d\xi d\eta)+ \frac{2}{\eps_n} \Big[ Z_2(\mu,\mu_n)+Z_2(\nu,\nu_n)\Big] \right\}.
	\end{equation*}
	To prove the convergence of the integral $\iint c \,d\gamma_n$ we note that
	the convergences $\mu_n \to \mu$, $\nu_n \to \nu$ in the metric $Z_2$ guarantee that the second moments $m_2(\mu_n)$, $m_2(\nu_n)$ converge, in particular they are bounded. In turn, this means that the sequence $\{\gamma_n\} \subset \Gamma(\mu_n,\nu_n)$ has uniformly bounded second moments as well. Then, the convergence of $\iint c \,d\gamma_n \to \iint c \,d\gamma$ follows by weak convergence of $\gamma_n \rightharpoonup \gamma$ and the $p$-growth of $c$ with $p<2$, see \cite[Section 5.1.1]{ambrosio2008}. Finally, the assumed rate of decrease of $\eps_n$ in \eqref{eq:rate_epsn} furnishes,
	\begin{equation*}
		\limsup_{n \to +\infty} F_n(\gamma_n,q_n) \leq  \iint c(x,y)\, \gamma(dxdy).
	\end{equation*}
	Since $(\gamma,q) \in \ov{\Gamma Q}(\mu,\nu)$, the latter integral is $F(\gamma,q)$ exactly, which concludes the proof.
\end{proof}

\section{Examples}

\label{sec:examples}

\subsection{Non-uniqueness of the Zolotarev projection}

The purpose of the first example is to show that the Zolotarev projection  can be non-unique, i.e. that the non-empty set $\Pi_{\succeq_c \mu}(\nu)$ may not be a singleton for some pair of probability distributions $\mu,\nu \in \mathcal{P}_2^\ba(\Rd)$. Recall that Proposition \ref{prop:Zol_proj} states that this set is identical to the set of convex dominants of minimum second moment. We can thus concentrate on proposing data such that \eqref{eq:minDom} has non-unique minimizers instead.

The following is the only example in this work with exact solution, not counting the 1D case in which the optimal convex dominance problem has a universal solution $\mu \vee \nu$, cf. Proposition \ref{prop:1D}. We refer to  \cite{bolbotowski2024kantorovich} where the quadratic case was solved analytically for  two type of data $\mu,\nu$: Gaussians (any dimensions $d \geq 1$) and two-point measures ($d=2$). For the latter case, Fig. \ref{fig:bimart} in the introduction herein depicts the optimal solution.

%\begin{example}[\textbf{Gaussian data}] For arbitrary dimension $d \geq 1$ assume a pair of centred Gaussians,
%	$\mu = \mathcal{N}(0,M)$, $\nu = \mathcal{N}(0,N)$, where $M, N$ are positive semi-definite symmetric $d \times d$ covariance matrices. Then, an optimal common convex dominant for the quadratic cost $f = \abs{\argu}^2$ -- and thus also the Zolotarev projection of $\mu$ onto the cone $\{\argu \succeq_c \nu\}$ and \textit{vice versa} -- reads,  
%	\begin{equation}
%		\ov\rho = \mathcal{N}(0,\ov{R}), \qquad \text{where} \quad  \ov{R} := \argmin_{R} \Big\{ \tr R \, : \, R \succeq M, \  R \succeq N \Big\}.
%	\end{equation}	
%	Above $R$ runs over the cone of positive symmetric  $d \times d$ matrices, and the inequality $R \geq M$ is understood in the sense that $R-M$ is positive semi-definite. The matrix $\ov{R}$ is unique, and it can be expressed in terms of the spectral decomposition of the  difference
%	$M-N = \sum_{i=1}^d \lambda_i\, a_i\otimes a_i$. The two equivalent formulas follow: $\ov{R} = N + \sum_i (\lambda_i)_+\, a_i\otimes a_i = M + \sum_i  (-\lambda_i)_+\, a_i\otimes a_i$, where $(\argu)_+$ is the positive part of a number. The same spectral decomposition furnishes the optimal potential in the Zolotarev problem,
%	\begin{equation}
%		\ov{u} = \frac{1}{2} \sum_{i=1}^d \mathrm{sgn}(\lambda_i) \pairing{a_i,x}^2,
%	\end{equation}
%	which also allows to conclude that the Zolotarev-2 distance itself is  half the first Schatten norm of $M-N$, namely $Z_2(\mu,\nu) = \frac{1}{2} \sum_{i=1}^d \abs{\lambda_i}$.
%	
%\end{example}

\begin{example}
	\label{ex:non-uniqueness}
	Consider  the planar optimal convex dominant problem for the quadratic cost $f = \abs{\argu}^2$ and the probability measures
	$\mu = \frac{1}{2} \mu_1 + \frac{1}{2} \mu_2$, \ $\nu = \frac{1}{2} \nu_1 + \frac{1}{2} \nu_2$ where:
	$\mu_1  = \frac{1}{2} \delta_{(-1,0)} + \frac{1}{2} \delta_{(1,0)}$, \ $\mu_2 = \frac{1}{2} \delta_{(-2,0)} + \frac{1}{2} \delta_{(2,0)}$, and $\nu_1  = \frac{1}{2} \delta_{(0,-1)} + \frac{1}{2} \delta_{(0,1)}$, \ $\nu_2 = \frac{1}{2} \delta_{(0,-2)} + \frac{1}{2} \delta_{(0,2)}$, see Fig. \ref{fig:non-unique}. The following plans are elements of $\Gamma(\mu,\nu)$,
	\begin{equation*}
		\ov\gamma = \frac{1}{2} \, \mu_1 \otimes \nu_1 + \frac{1}{2}\,  \mu_2 \otimes \nu_2, \qquad 	\ov\gamma' = \frac{1}{2} \, \mu_1 \otimes \nu_2 + \frac{1}{2}\,  \mu_2 \otimes \nu_1,
	\end{equation*}
	and they are bi-martingale with respect to the coupling map $\ov\zeta(x,y) = x+y$. Indeed, the conditions \eqref{eq:2xmartingale_condition} can be easily checked. For instance, $\ov\gamma(dxdy) = \mu(dx) \otimes \ov\gamma_x(dy)$ where $\ov\gamma_x = \nolinebreak \nu_i$ for $\mu_i$-a.e. $x$, $i=1,2$. Then, $	\int \ov\zeta(x,y) \, \gamma_x(dy)  = \int (x+y) \, \gamma_x(dy)= x$ since $[\nu_i] = 0$.  The second condition in \eqref{eq:2xmartingale_condition} is verified similarly, and the same reasoning works for $\ov\gamma'$ as well.
		\begin{figure}[h]
		\centering
		\subfloat[]{\includegraphics*[trim={0cm 0cm -0cm -0cm},clip,width=0.2\textwidth]{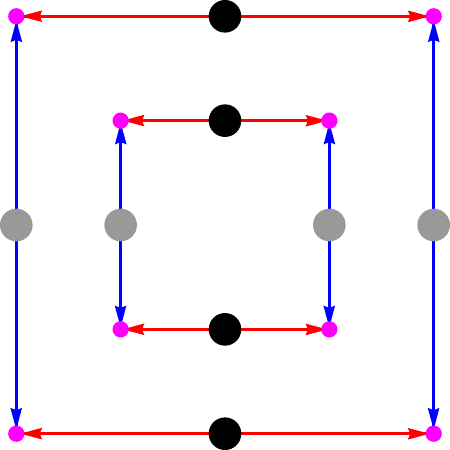}}\hspace{0.7cm}
		\subfloat[]{\includegraphics*[trim={0cm 0cm -0cm -0cm},clip,width=0.2\textwidth]{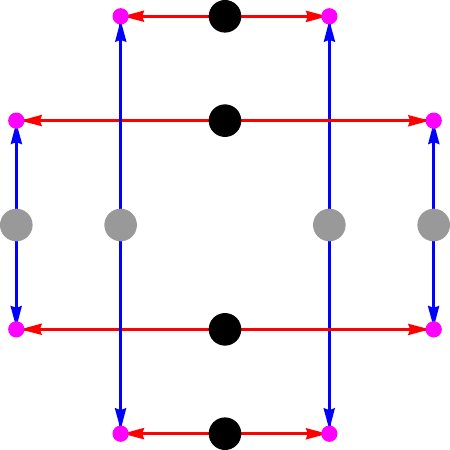}}\hspace{0.7cm}
		\subfloat[]{\includegraphics*[trim={0cm 0cm -0cm -0cm},clip,width=0.2\textwidth]{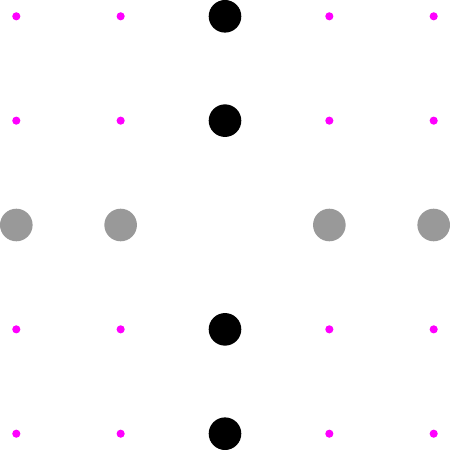}}\hspace{0.7cm}
		\subfloat[]{\includegraphics*[trim={0cm 0cm -0cm -0cm},clip,width=0.2\textwidth]{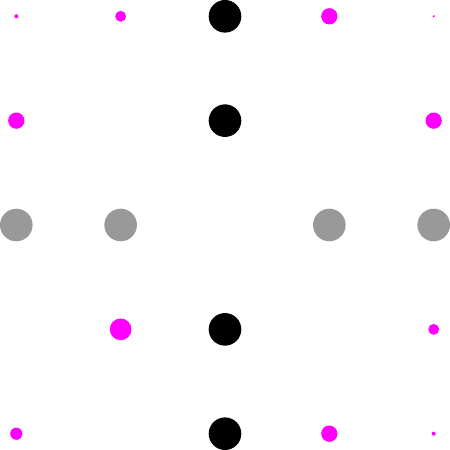}}
		
		\caption{Data $\mu$ (gray), $\nu$ (black), the optimal convex dominant $\rho$ (magenta), and the associated martingale plans $\gamma_1 \in \Gamma_{\mathrm{M}}(\mu,\rho)$ (blue), $\gamma_2 \in \Gamma_{\mathrm{M}}(\nu,\rho)$ (red). Subsequent figures display different solutions.}
		\label{fig:non-unique}      
	\end{figure}
	
	Let us now show that both the pairs $(\ov\gamma,\ov{q}) = (\ov\gamma,\ov\zeta \ov\gamma)$ and $(\ov\gamma,\ov{q}') = (\ov\gamma',\ov\zeta \ov\gamma')$ are solutions to the two equivalent problems in Theorem \ref{thm:RK_reinvented}. First, a straightforward computation of the energy $\iint \frac{1}{2} \big(\abs{\ov{\zeta}(x,y)-x}^2  + \abs{\ov{\zeta}(x,y)-y}^2 \big) \,\gamma(dxdy)$ for $\gamma = \ov\gamma, \ov\gamma'$ yields the same value $\frac{1}{2} m_2(\mu) + \frac{1}{2} m_2(\nu) = \frac{5}{2}$. In particular, this furnishes the upper bound $Z_2(\mu,\nu) \leq \frac{1}{2} m_2(\mu) + \frac{1}{2} m_2(\nu) $. On the other hand, consider the smooth potential,
	\begin{equation*}
		\label{eq:optu_orthogonal}
		\ov{u}(x) = \frac{1}{2} \pairing{e_1,x}^2 - \frac{1}{2} \pairing{e_2,x}^2,
	\end{equation*}
	where $e_1 = (1,0)$, $e_2 =(0,1)$. Observe that $\nabla \ov{u} : \Rd \to \Rd$ is an isometry, thus in particular $\mathrm{lip}(\nabla \ov{u}) = 1$, rendering $\ov{u}$ a competitor in the Zolotarev problem \eqref{eq:Z2}. The lower bound follows: $Z_2(\mu,\nu) \geq \int \ov{u} \,d(\mu-\nu)$. Since $\mu$ and $\nu$ are supported on the lines $\{t e_1 : t \in \R\}$ and  $\{te_2: t \in \R\}$, respectively, the latter integral equals $\frac{1}{2} m_2(\mu) + \frac{1}{2} m_2(\nu)$. Combining the two bounds shows that $Z_2(\mu,\nu) = \frac{1}{2} m_2(\mu) + \frac{1}{2} m_2(\nu)$, which establishes optimality of the pairs $(\ov\gamma,\ov\zeta \ov\gamma)$ and  $(\ov\gamma',\ov\zeta \ov\gamma')$, along with optimality of $\ov{u}$ in the Zolotarev problem.
	
	By virtue of Theorem \ref{thm:RK_reinvented}, the two optimal bi-martingale plans lead to two convex dominants of $\mu,\nu$ with the least second moment, i.e. solving the problem \eqref{eq:minDom},
	\begin{align*}
		&\ov\rho = \ov\zeta \# \ov\gamma = \frac{1}{2} \,\mu_1 * \nu_1 +  \frac{1}{2} \,\mu_2 * \nu_2 = \frac{1}{8} \sum_{\substack{i,j \in \{-1,1\} }} \big(\delta_{(i,j)} + \delta_{(2i,2j)}\big),  \\
		&\ov\rho' = \ov\zeta \# \ov\gamma' = \frac{1}{2} \,\mu_1 * \nu_2 +  \frac{1}{2} \,\mu_2 * \nu_1 = \frac{1}{8} \sum_{\substack{i,j \in \{-1,1\} }} \big(\delta_{(i,2j)} + \delta_{(2i,j)}\big),
	\end{align*}
	where $*$ is the convolution of measures. From the same theorem we also deduce that both of them are Zolotarev projections of $\mu$ or $\nu$ on the respective cones of dominating measures, namely $\ov\rho, \ov\rho' \in \Pi_{\succeq_c \mu}(\nu) = \Pi_{\succeq_c \nu}(\mu)$. We have thus showed that the Zolotarev projections, and the minimal convex dominant, are not uniquely determined, which was announced in Remark \ref{rem:conv_dom}(d) and Theorem \ref{thm:Zol_proj_again}.
	
	The probabilities $\ov\rho$ and $\ov\rho'$ are shown in Fig. \ref{fig:non-unique}(a) and Fig. \ref{fig:non-unique}(b), respectively, along with the associated martingale plans $\ov\gamma_i = (\pi_i,\ov\zeta) \# \ov\gamma$ and $\ov\gamma'_i = (\pi_i,\ov\zeta) \# \ov\gamma'$, where $i=1,2$. Naturally, every convex combination of $\ov\rho$ and $\ov\rho'$ yields another optimal dominant. In particular, the solution $\frac{1}{2}\ov\rho + \frac{1}{2} \ov\rho' = \mu * \nu$ is shown in Fig. \ref{fig:non-unique}(c).  In fact, the set of optimal dominants is even richer: it can be shown to be a convex set of dimension 4, cf. yet another solution in Fig. \ref{fig:non-unique}(d).
\end{example}

\begin{remark}
	The optimal coupling map $\ov\zeta(x,y) = x+y$ is in agreement with the condition (ii) in Theorem \ref{thm:RK_reinvented}. Indeed, for $\ov{u}$ in \eqref{eq:optu_orthogonal} the formula \eqref{eq:ovzeta} gives $\ov\zeta(x,y) = \pairing{e_1,x} e_1 + \pairing{e_2,y} e_2$. The latter expression equals to $x+y$ on the set $\spt \mu \otimes \spt \nu \subset  \{t e_1 : t \in \R\} \otimes  \{t e_2: t \in \R\}$.
\end{remark}

%\begin{remark}
%	Let us make two comments in respect with the above example:
%	\begin{enumerate}[leftmargin=0.9cm, labelsep=5pt]
%		\item[(a)] A similar structure of the solutions is exhibited whenever $\mu$ and $\nu$ have mutually orthogonal supports, e.g. $\spt \mu \subset \{t e_1 : t \in \R\}$	and $\spt \nu \subset \{t e_2 : t \in \R\}$ on the plane. In particular, $\ov{u}$ as in \eqref{eq:optu_orthogonal} is a universal optimal potential, while, provided that $\mu,\nu$ are centred,  the convolution $\mu * \nu$ is a universal optimal convex dominant for $f = \abs{\argu}^2$.
%		\item [(b)] The optimal coupling map $\ov\zeta(x,y) = x+y$ is in agreement with the condition (ii) in Theorem \ref{thm:RK_reinvented}. Indeed, for $\ov{u}$ in \eqref{eq:optu_orthogonal} the formula \eqref{eq:ovzeta} gives $\ov\zeta(x,y) = \pairing{e_1,x} e_1 + \pairing{e_2,y} e_2$. The latter expression equals to $x+y$ on the set $\spt \mu \otimes \spt \nu \subset  \{t e_1 : t \in \R\} \otimes  \{t e_2: t \in \R\}$.
%	\end{enumerate}
%\end{remark}

\subsection{Optimal convex dominance in two dimensions -- simulations}
\label{ssec:conic}

As advertised in the introduction, the bi-martingale formulation $(\mathrm{M^2OT})$ paves ways to efficient numerical methods. The simplest strategy is to use off-the-shelf software for non-linear convex optimization. In this subsection we shortly demonstrate the computational potential of $(\mathrm{M^2OT})$ for the optimal convex dominance problem with the power cost, i.e. in the scenario $c(x,y,z) = f(z) = \abs{z}^p$ for $p \in (1,+\infty)$. Assuming that $\mu,\nu$ are discrete,  we can rewrite $(\mathrm{M^2OT})$ as a finite dimensional convex conic optimization problem.
 
For $\mu = \sum_i \mu_i \delta_{x_i}$, $\nu = \sum_j \nu_j \delta_{y_j}$ on the plane $\R^2$, the variables in $(\mathrm{M^2OT})$ are simply 
$\gamma = \sum_{i,j} \gamma_{ij} \delta_{(x_i,y_j)}$ and $q = \sum_{i,j} \big(q_{ij}^{(1)}, q_{ij}^{(2)}\big) \delta_{(x_i,y_j)}$, where $\gamma_{ij}, q_{ij}^{(1)}, q_{ij}^{(2)}$ are numbers. Let us add the fourth variable $r_{ij}$ and, for each $i,j$, consider the convex cone in $\R^4$ satisfying,
\begin{equation*}
	(r_{ij})^{1/p} (\gamma_{ij})^{1-1/p} \geq \sqrt{\big(q_{ij}^{(1)}\big)^2 +\big(q_{ij}^{(2)}\big)^2}, \qquad r_{ij} \geq 0, \ \ \gamma_{ij} \geq 0.
\end{equation*}
Note that this constraint guarantees that $q \ll \gamma$. Naturally, the condition $(\gamma,q) \in \Gamma(\mu,\nu)$ becomes a linear system of algebraic equations. The total cost satisfies the inequality
$\iint |\frac{dq}{d\gamma}|^p d\gamma \leq \sum_{i,j} r_{ij}$, which becomes an equality at optimality. This conic form of the problem is ready to be tackled by standard convex optimization algorithms. For our simulations below, we have used the MOSEK solver \cite{mosek2024}. After we have found the optimal solution $\hat{\gamma}_{ij}, \hat{q}_{ij}^{(1)}\!, \,\hat{q}_{ij}^{(2)}$, the optimal convex dominant for the cost $f = \abs{\argu}^p$ can be computed as,
\begin{equation*}
	\hat\rho = \hat\zeta \# \hat\gamma = \sum_{i,j} \hat\gamma_{ij} \, \delta_{\left(\hat{q}_{ij}^{(1)}\!, \,\hat{q}_{ij}^{(2)}\right)/ \hat\gamma_{ij}} .
\end{equation*}

It must be stressed that, if $\mu,\nu$ are discrete and the finite dimensional conic problem is solved exactly, then $\hat{\rho}$ above is an exact solution of the optimal convex dominance problem \eqref{eq:conv_dom_intro}, see also Remark \ref{rem:conv_dom}(b). Hence, in the numerical simulations below, the solutions are as accurate as is the MOSEK's interior point algorithm.

\begin{example}
		\label{ex:conv_dom}
		Consider the data $\mu = \mathcal{L}^2 \mres Q$, being the Lebesgue measure restricted to a unit square $Q \subset \R^2$ centred at the origin, and $\nu = \sum_{j=1}^5 \frac{1}{5} \delta_{y_j}$, where $y_1 = (\frac{2}{5},0)$, $y_2 = (-\frac{2}{5},0)$, $y_3 = (0,\frac{2}{5})$, $y_4 =(0,-\frac{2}{5})$, $y_5 = (0,0)$, see Fig. \ref{fig:diff_p}. In the simulations, we shall discretize $\mu$ with a regular mesh of $121 \times 121$ points (we keep the symbol $\mu$).
		
		\begin{figure}[h]
			\centering
			\subfloat[$p=1 + 1\text{e-}3$]{\includegraphics*[trim={0cm 0cm -0cm -0cm},clip,width=0.3\textwidth]{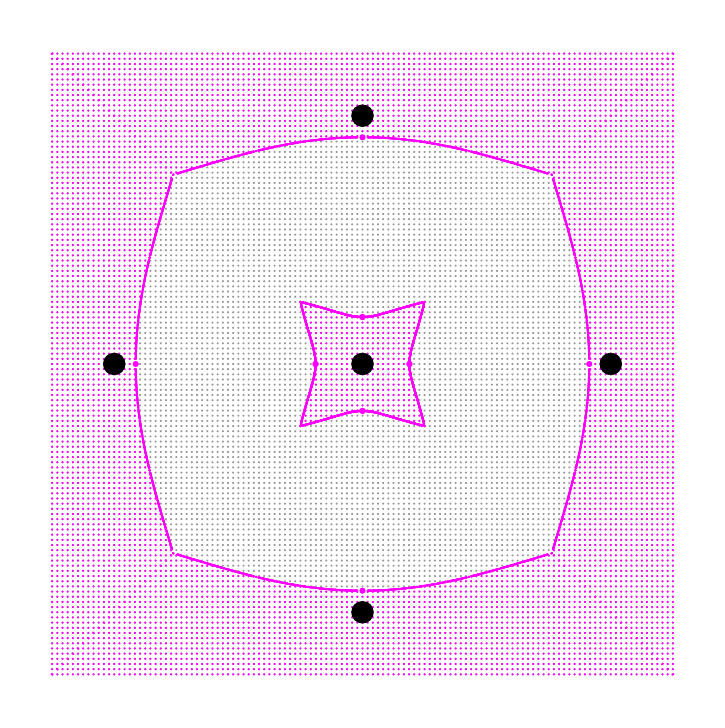}}\hspace{0.5cm}
			\subfloat[$p=1 + 1\text{e-}1$]{\includegraphics*[trim={0cm 0cm -0cm -0cm},clip,width=0.3\textwidth]{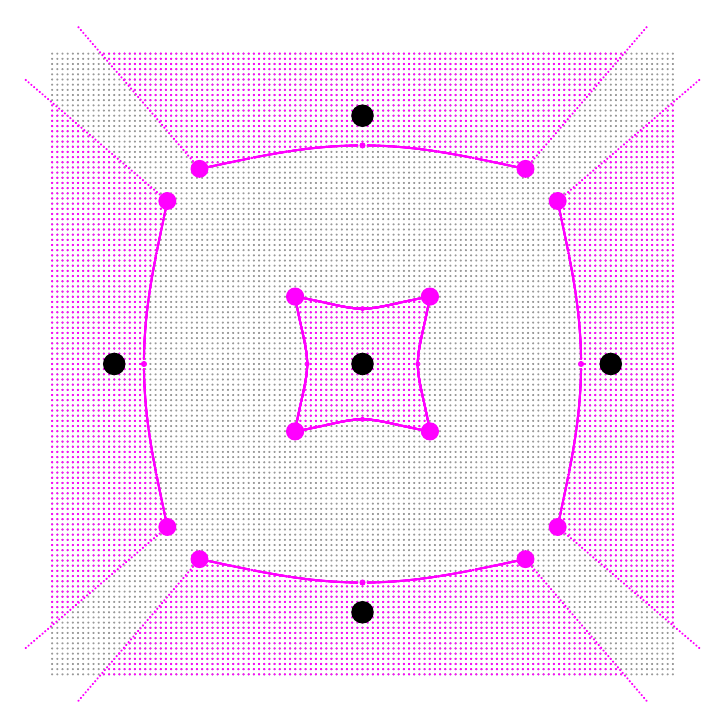}} \hspace{0.5cm}
			\subfloat[$p=2$]{\includegraphics*[trim={0cm 0cm -0cm -0cm},clip,width=0.3\textwidth]{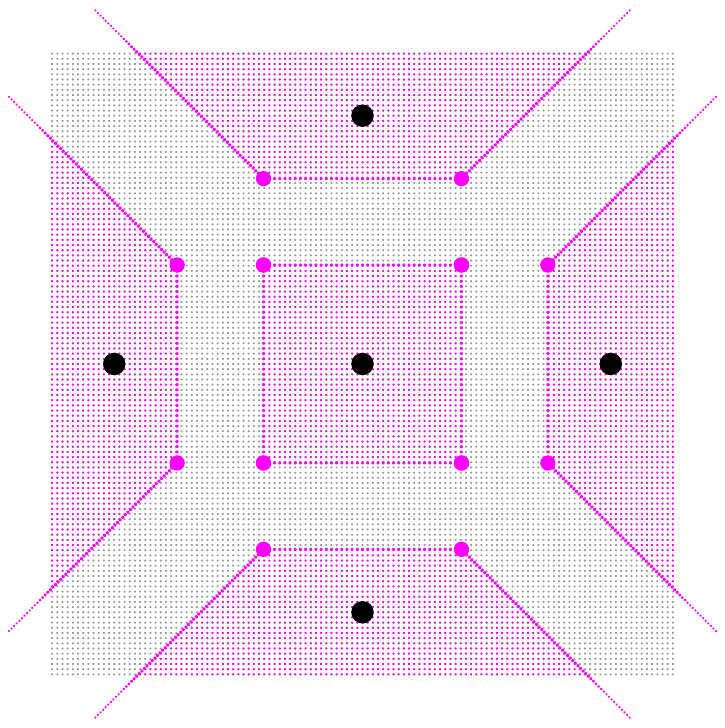}}\\
			\subfloat[$p=4$]{\includegraphics*[trim={0cm 0cm -0cm -0cm},clip,width=0.3\textwidth]{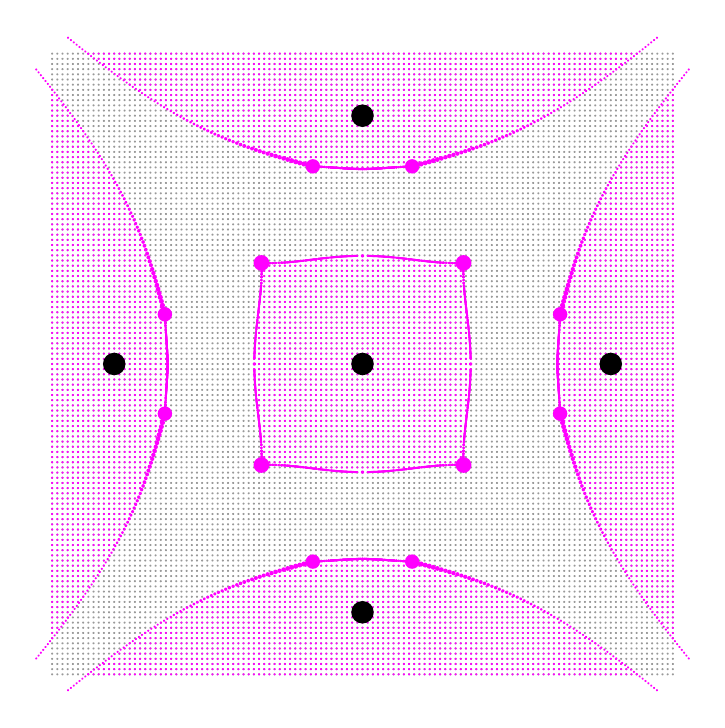}} \hspace{0.5cm}
			\subfloat[$p=11$]{\includegraphics*[trim={0cm 0cm -0cm -0cm},clip,width=0.3\textwidth]{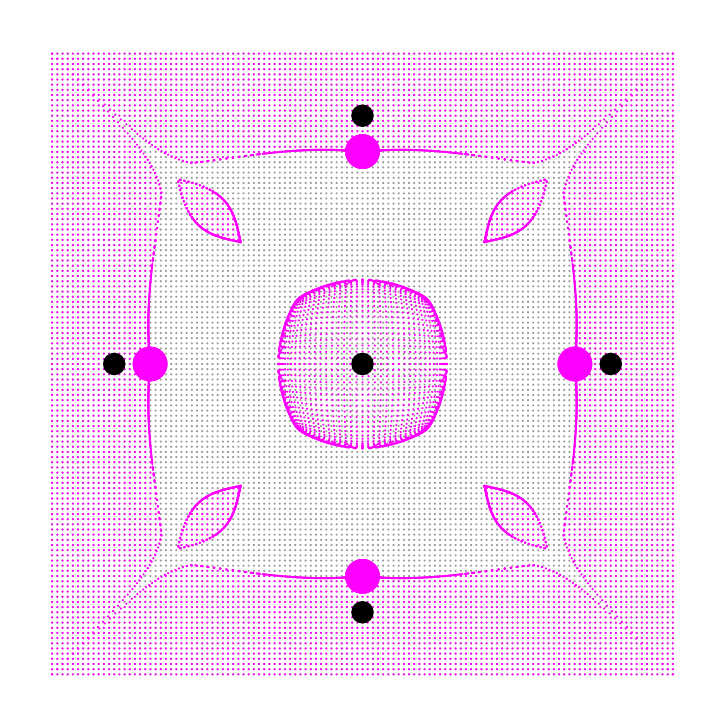}}
			\caption{Data $\mu$ (discretized, gray), $\nu$ (black), the optimal convex dominant $\ov\rho$ (magenta) for the cost $f = \abs{\argu}^p$. For visibility, a different and finer scale was applied to the measure $\nu$ in comparison to $\mu,\rho$.}
			\label{fig:diff_p}      
		\end{figure}
		
		Let us check for the convex order between  $\mu$ and $\nu$. Using the conic formulation as described above, we can compute the distance $Z_2(\mu,\nu) \approx 0.0233$, which is all that is required for calculating the convex-order index, 
		\begin{equation*}
				\alpha_{\succeq_c}(\nu\,|\,\mu)= \frac{m_2(\nu) - m_2(\mu)}{2 Z_2(\mu,\nu)} \approx  -0.8898.
		\end{equation*}
		As it is not $-1$ nor $1$, the two measures are not ordered in any direction. Since the index is much closer to $-1$, we deduce that '$\mu$ is not far from dominating $\nu$'. To put this in perspective, let us add that the convex order $\nu \preceq_c \mu$ would be revived if $\nu$ was contracted by moving the four masses towards the centre by less than $0.02$.
		
		Fig. \ref{fig:diff_p} shows the solutions for five different values of the exponent $p$. Based on these simulations, we may predict that for the exact $\mu$ being the Lebesgue measure, the solutions $\hat\rho$ consist of two-, one- and zero-dimensional parts. The 1D part charges curves that are partially the frontier of the two dimensional part, and partially they may extend beyond the square~$Q$. The Dirac delta parts seem to be positioned at the kinks of those curves. It should be emphasized that lower dimensional parts of $\hat\rho$ are to be expected even if both $\mu$ and $\nu$ are absolutely continuous and with smooth densities, cf. Remark \ref{rem:conv_dom}(a).
		The case where $p=2$ stands out with the simplicity of its geometry, see Fig. \ref{fig:diff_p}(c).  Recall that this $\hat\rho$ is precisely the Zolotarev projection $\ov\rho$ of $\mu$ onto the cone of convex dominants of $\nu$ (and \textit{vice versa}), cf. Proposition \ref{prop:Zol_proj}.
\end{example}

\subsection{Demonstration of the stabilization effect for the approximation of (MOT)}

In Section \ref{ssec:MOT_app} we have discussed the issue of the lack of stability of the martingale optimal transport problem $(\mathrm{MOT})$ with respect to the marginal probabilities $\mu,\nu$ in dimension $d\geq2$. This serious drawback of multidimensional $(\mathrm{MOT})$ has been exposed in \cite{bruckerhoff2022instability} with a simple two-dimensional counter-example. 

Below we use this example to demonstrate the essence of Theorem \ref{thm:approx_MOT_intro}, which by means of $\Gamma$-convergence guarantees that the sequence of plans $\gamma_n$ -- solving the approximating bi-martingale problems $(\mathrm{M^2OT}_n)$ posed for the marginals $\mu_n,\nu_n$ and for the penalty coefficient $\eps_n$ that tends to zero sufficiently slowly -- weakly converges to $\gamma \in \Gamma_{\mathrm{M}}(\mu,\nu)$ that solves $(\mathrm{MOT})$ for the limit data $\mu,\nu$. 
\captionsetup[subfigure]{labelformat=empty}
\begin{figure}[h]
	\centering
	\subfloat[\parbox{5cm}{$n=3, \ \  \iint \abs{x-y} \,d\tilde\gamma_n=1$}]{\includegraphics*[trim={0cm 0cm -0cm -0cm},clip,width=0.33\textwidth]{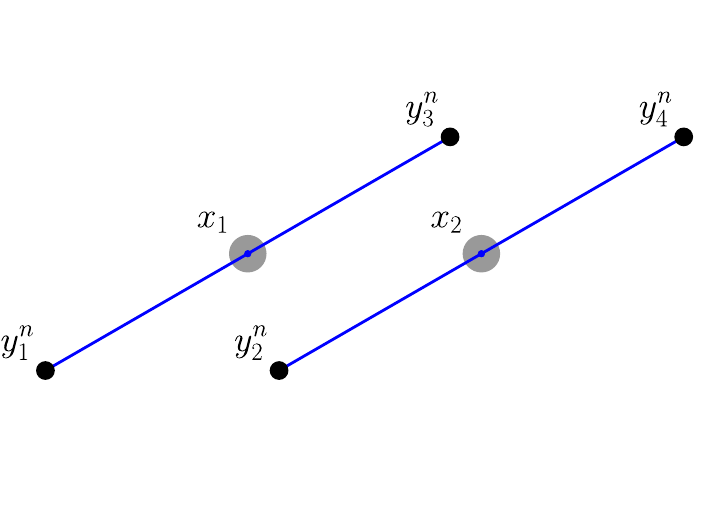}}\hspace{1.5cm}
	\subfloat[\parbox{5cm}{$n=3, \ \  \iint \abs{x-y} \,d\gamma_n \approx 0.9223$}]{\includegraphics*[trim={0cm 0cm -0cm -0cm},clip,width=0.33\textwidth]{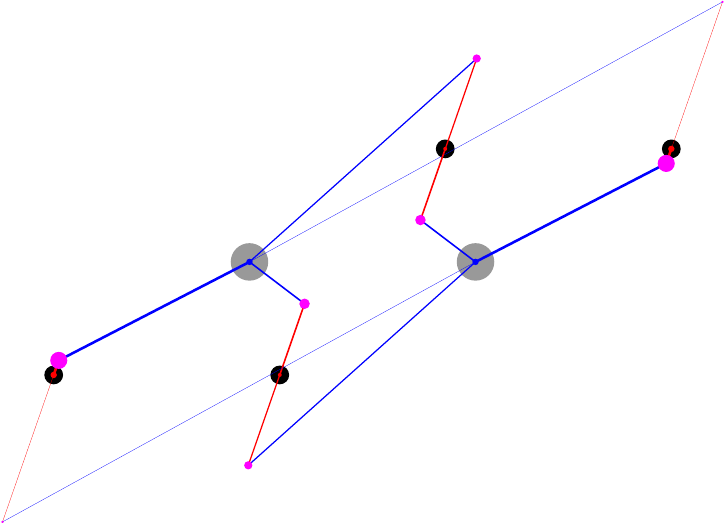}} \\
	\subfloat[\parbox{5cm}{$n=5, \ \  \iint \abs{x-y} \,d\tilde\gamma_n=1$}]{\includegraphics*[trim={0cm 0cm -0cm -0cm},clip,width=0.33\textwidth]{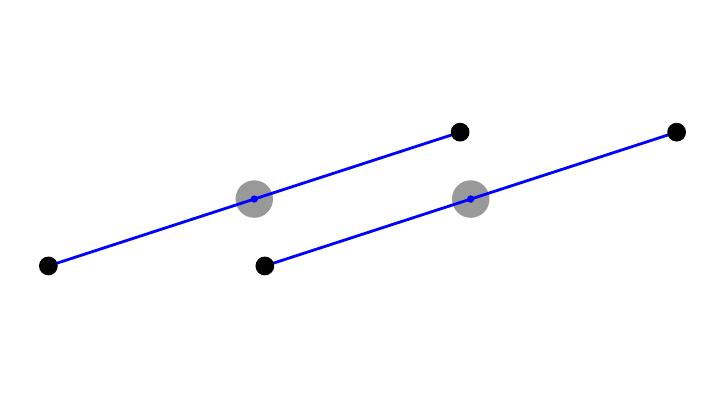}}\hspace{1.5cm}
	\subfloat[\parbox{5cm}{$n=5, \ \  \iint \abs{x-y} \,d\gamma_n \approx 0.8209$}]{\includegraphics*[trim={0cm 0cm -0cm -0cm},clip,width=0.33\textwidth]{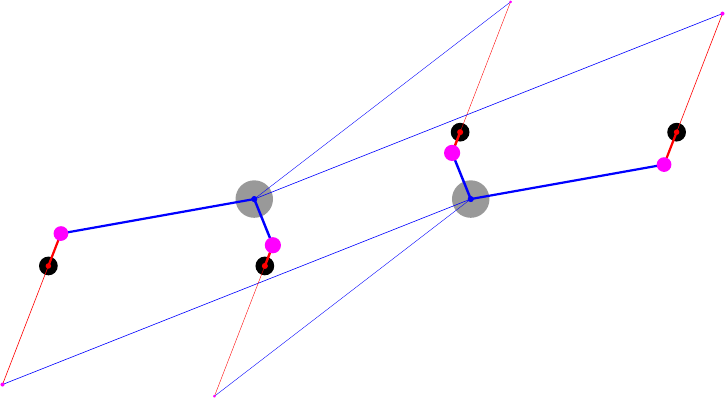}} \\
	\subfloat[\parbox{5cm}{$n=20, \ \  \iint \abs{x-y} \,d\tilde\gamma_n=1$}]{\includegraphics*[trim={0cm 0cm -0cm -0cm},clip,width=0.33\textwidth]{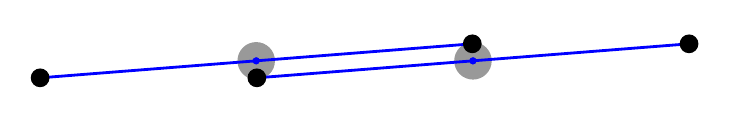}}\hspace{1.5cm}
	\subfloat[\parbox{5cm}{$n=20, \ \  \iint \abs{x-y} \,d\gamma_n \approx 0.6928$}]{\includegraphics*[trim={0cm 0cm -0cm -0cm},clip,width=0.33\textwidth]{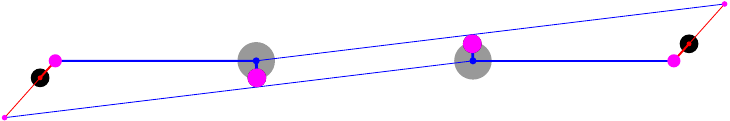}}
	\caption{Left column: optimal martingale transport (blue) between $\mu$ (gray) and  $\nu_n$ (black). Right column: bi-martingale transport (blue and red) solving the approximation $(\mathrm{M^2OT}_n)$ for the same data. In the right column the probability  $\rho_n = \zeta_n \# \gamma_n$ (magenta) is also displayed.}
	\label{fig:stability}       % Give a unique label
\end{figure}

\begin{figure}[h]
\centering

\subfloat[]{\includegraphics*[trim={0cm 0cm -0cm -0cm},clip,width=0.45\textwidth]{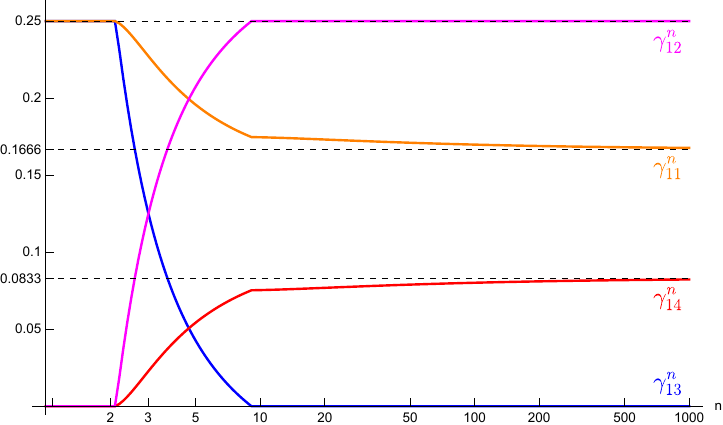}}\hspace{0.5cm}
\subfloat[]{\includegraphics*[trim={0cm 0cm -0cm -0cm},clip,width=0.45\textwidth]{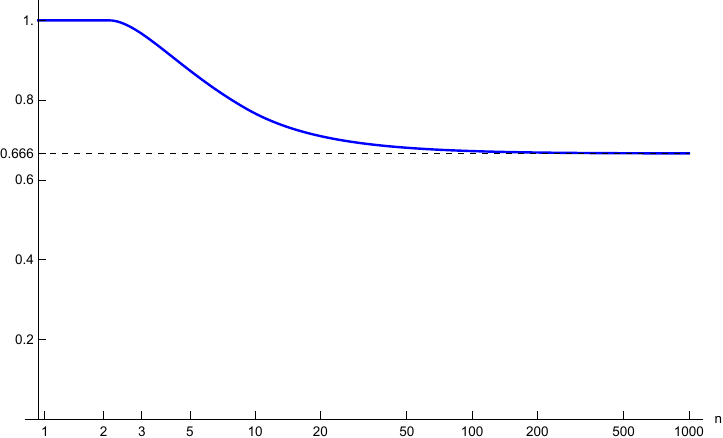}}
\caption{The sequence of solutions $\gamma_n = \sum_{i,j} \gamma^n_{ij}  \, \delta_{(x_i,y_j)} $ to the bi-martingale approximations $(\mathrm{M^2OT_n})$: (a) the transports emanating from the mass at $x_1$; (b) the total cost $\iint \abs{x-y} \,\gamma_n(dxdy)$  (without the penalizing term). }
\label{fig:conv}       % Give a unique label
\end{figure}

\begin{example}
	\label{ex:stability}
	On the plane $\R^2$, we will consider the $(\mathrm{MOT})$ problem for the distance cost $c(x,y) = \abs{x-y}$ and the data $\mu = \sum_{i=1}^2\frac{1}{2}\,\delta_{x_i}$, $\nu = \sum_{j=1}^4 \frac{1}{4} \, \delta_{y_j}$, where $x_1 = (-\frac{1}{2},0)$, $x_2 = (\frac{1}{2},0)$ and  $y_1 = (-\frac{3}{2},0)$, $y_2 = (-\frac{1}{2},0)$, $y_3 = (\frac{1}{2},0)$, $y_4 = (\frac{3}{2},0)$.
	The problem is essentially one dimensional. The convex order $\mu \preceq_c \nu$ is easily verified, and the set of martingale plans  $\Gamma_{\mathrm{M}}(\mu,\nu)$ contains more than one element. The unique martingale plan $\gamma$ solving $(\mathrm{MOT})$ for the distance cost is the one that does not move the shared mass, i.e. $\gamma = \frac{1}{6}\delta_{(x_1,y_1)} +  \frac{1}{4}\delta_{(x_1,y_2)} + \frac{1}{12}\delta_{(x_1,y_4)} +  \frac{1}{12}\delta_{(x_2,y_1)} +  \frac{1}{4}\delta_{(x_2,y_3)} + \frac{1}{6}\delta_{(x_2,y_4)}$, cf. \cite{bruckerhoff2022instability} for details. The minimal martingale transport cost equals $\iint \abs{x-y} \, \gamma(dxdy) = \frac{2}{3}$. 

	After \cite{bruckerhoff2022instability}, we will now build a sequence $\nu_n$ that weakly converges to $\nu$ (also in metrics $W_2$ or $Z_2$), cf. Fig. \ref{fig:stability}. With  $R_{\theta_n} e_1$ being the anti-clockwise rotation of $e_1=(1,0)$ by the angle $\theta_n:= \frac{\pi}{2n}$, we  define $ \nu_n = \sum_{j=1}^4 \frac{1}{4} \delta_{y_j^n}$, where,
	\begin{equation*}
	 y_1^n = x_1 - R_{\theta_n}e_1, \quad 	y_2^n = x_2 - R_{\theta_n}e_1, \quad y_3^n = x_1 + R_{\theta_n}e_1, \quad 	y_4^n = x_2 + R_{\theta_n}e_1,
	\end{equation*}
	see Fig. \ref{fig:stability}(a). We see that $y_j^n \to y_j$, hence $\nu_n \rightharpoonup \nu$. Moreover, the following is a martingale plan,
	\begin{equation*}
		\tilde{\gamma}_n = \frac{1}{4} \Big( \delta_{(x_1,y_1^n)} + \delta_{(x_1,y_3^n)} + \delta_{(x_2,y_2^n)} + \delta_{(x_2,y_4^n)} \Big) \ \in \  \Gamma_{\mathrm{M}}(\mu,\nu_n).
	\end{equation*}
	
	The core idea behind this example lies in the following observation \cite[Lemma 1.1]{bruckerhoff2022instability}: for any  $n \geq 1$ the set of martingale plans between $\mu$ and $\nu^\theta$ is a singleton, namely $\Gamma_{\mathrm{M}}(\mu,\nu_n) = \{ \tilde\gamma_n\}$, rendering the relevant martingale optimal transport problem trivial. The left column of Fig. \ref{fig:stability} present these trivial solutions $\tilde\gamma_n \in \Gamma_{\mathrm{M}}(\mu,\nu_n)$ for several $n$. The minimal transport cost is constant across all $n$, namely $\iint \abs{x-y} \, \tilde{\gamma}_n(dxdy) = 1$. Clearly, $\tilde{\gamma}_n$ weakly converges to $\tilde{\gamma} = \frac{1}{4} \big( \delta_{(x_1,y_1)} + \delta_{(x_1,y_3)} + \delta_{(x_2,y_2)} + \delta_{(x_2,y_4)} \big)$, being a martingale plan for the limit data, that is $\tilde{\gamma} \in \Gamma_{\mathrm{M}}(\mu,\nu)$. This is not the plan $\gamma$ which, as established above, is the unique optimal solution for $\mu,\nu$. In fact, $	\iint \abs{x-y} \,\tilde\gamma(dxdy) = 1 > \frac{2}{3} = 	\iint \abs{x-y} \,\gamma(dxdy).$
	
	\smallskip

	Let us now jump to the sequence of bi-martingale problems $(\mathrm{M^2OT_n})$ for the sequence of data $\mu_n \equiv \mu$, $\nu_n$. To define the problems, we must pick a sequence of positive numbers $\eps_n$ converging to zero. In order to meet the prerequisites of Theorem \ref{thm:approx_MOT_intro}, we have to ensure that $Z_2(\nu,\nu_n)/\eps_n \to 0$. Computing $Z_2(\nu,\nu_n)$ numerically can be by-passed via the $W_2$ upper bound, cf. \eqref{eq:rio} and \cite{bolbotowski2025},
	\begin{equation*}
		Z_2(\nu,\nu_n) \leq \frac{1}{2} (\sigma_{\nu}+ \sigma_{\nu_n}) \, W_2(\nu,\nu_n) \leq \frac{\sqrt{5}}{2} \big(2 \sin(\theta_n/2)\big) \leq   \frac{\sqrt{5}}{2} \theta_n =  \frac{\sqrt{5} \pi}{4n} =: \eps_n^2.
	\end{equation*}
	Above we computed the variances $\sigma_{\nu_n} = \sigma_ \nu = \frac{\sqrt{5}}{2}$, whilst the Wasserstein distance was estimated using the transport maps $T_n(y_j) = y_j^n$. The above choice, i.e. $\eps_n \approx 1.325\, n^{-1/2} \geq \sqrt{Z_2(\nu,\nu_n)}$ guarantees that  $Z_2(\nu,\nu_n)/\eps_n \to 0$.
	
	Readily, Theorem \ref{thm:approx_MOT_intro} states that, for the sequence of solution $({\gamma}_n,{q}_n) = ({\gamma}_n,\zeta_n{\gamma}_n) $ to $(\mathrm{M^2OT_n})$, the plans $\gamma_n$ weakly converge to a solution of $(\mathrm{MOT})$ for the limit data $\mu,\nu$. Since we know that for this data the solution is unique and equals $\gamma$, we deduce,
	\begin{align*}
		\gamma_n = \sum_{i=1}^2 \sum_{j=1}^4 \gamma^n_{ij}  \, \delta_{(x_i,y_j)} \quad \rightharpoonup  \quad \gamma =  \tfrac{1}{6}\delta_{(x_1,y_1)} &+  \tfrac{1}{4}\delta_{(x_1,y_2)} + \tfrac{1}{12}\delta_{(x_1,y_4)} \\
		&+  \tfrac{1}{12}\delta_{(x_2,y_1)} +  \tfrac{1}{4}\delta_{(x_2,y_3)} + \tfrac{1}{6}\delta_{(x_2,y_4)},
	\end{align*}
	as well as the convergence of the cost: $\iint \abs{x-y} \,\gamma_n(dxdy) \to \iint \abs{x-y} \,\gamma(dxdy) = \frac{2}{3}$. For each $n$ ranging to $10^3$, the problem $(\mathrm{M^2OT_n})$ was numerically solved by employing the conic formulation described in Section \ref{ssec:conic}. The plots in Fig. \ref{fig:conv} show: (a) the transports $\gamma_{1j}^n$ emanating from the mass $x_1$; and (b) the cost $\iint \abs{x-y} \,\gamma_n(dxdy)$. The foregoing convergences can be observed. In addition, for chosen $n$, the right column in Fig. \ref{fig:stability} presents the induced martingale plans $\gamma_{1}^n = (\pi_1,\zeta_n) \# \gamma_n$ (blue) and $\gamma_{2}^n = (\pi_2,\zeta_n) \# \gamma_n$ (red), as well as their common second marginal $\rho_n = \zeta_n \# \gamma_n$ (magenta).
\end{example}

\bibliographystyle{plain}

\begin{thebibliography}{10}
	
	\bibitem{agueh2011barycenters}
	M. Agueh, G. Carlier:
	\newblock Barycenters in the Wasserstein space.
	\newblock {\em SIAM J. Math. Anal.} 43:904--924, 2011.
	
	\bibitem{alfonsi2020sampling}
	A. Alfonsi, J. Corbetta, B. Jourdain:
	\newblock Sampling of probability measures in the convex order by Wasserstein
	projection.
	\newblock {\em Ann. inst. Henri Poincar\'{e} (B) Probab. Stat.}, 56:1706--1729, 2020.
	
	\bibitem{alibert2019}
	J.-J. Alibert, G. Bouchitt{\'e}, T. Champion:
	\newblock A new class of costs for optimal transport planning.
	\newblock {\em Eur. J. Appl. Math.}, 30:1229--1263,
	2019.
	
	\bibitem{ambrosio-fusco}
	L. Ambrosio, N. Fusco, D. Pallara: 
	\newblock {\em Functions of bounded variation and free
		discontinuity problems}.
	\newblock Clarendon Press, Oxford, 2000.
	
	
	\bibitem{ambrosio2008}
	L. Ambrosio, N. Gigli, G. Savar{\'e}.
	\newblock {\em Gradient flows: in metric spaces and in the space of probability
		measures}.
	\newblock Springer, 2005.
	
	\bibitem{mosek2024}
	Mosek ApS.
	\newblock Mosek optimization toolbox for MATLAB.
	\newblock {\em User’s Guide and Reference Manual, Release 10.1.27}, 2024.
	
	\bibitem{backhoff2022stability}
	J. Backhoff-Veraguas, G. Pammer:
	\newblock Stability of martingale optimal transport and weak optimal transport.
	\newblock {\em Ann. Appl. Probab.}, 32:721--752, 2022.
	
	\bibitem{beiglbock2013}
	M. Beiglb{\"o}ck, P. Henry-Labordere, F. Penkner:
	\newblock Model-independent bounds for option prices--a mass transport
	approach.
	\newblock {\em Finance Stoch.}, 17:477--501, 2013.
	
	\bibitem{beiglbock2016}
	M. Beiglb{\"o}ck, N. Juillet:
	\newblock On a problem of optimal transport under marginal martingale
	constraints.
	\newblock {\em Ann. Probab.}, 44:42--106, 2016.
	
	\bibitem{beiglbock2017complete}
	M. Beiglb{\"o}ck, M. Nutz, N. Touzi:
	\newblock Complete duality for martingale optimal transport on the line.
	\newblock {\em Ann. Probab.}, 45:3038--3074, 2017.
		
	\bibitem{belili2000}
	N. Belili, H. Heinich:
	\newblock Distances de Wasserstein et de Zolotarev.
	\newblock {\em C. R. Acad. Sci. Paris}, t. 330, Srie I, p. 811–814, 2000.
	
	\bibitem{benamou2000}
	J.-D. Benamou, Y. Brenier:
	\newblock A computational fluid mechanics solution to the Monge-Kantorovich
	mass transfer problem.
	\newblock {\em Numer. Math.}, 84:375--393, 2000.
	
	\bibitem{bogachev2007}
	V. I. Bogachev.
	\newblock {\em Measure theory}.
	\newblock Springer, 2007.
	
	\bibitem{bolbotowski2024kantorovich}
	K. Bo{\l}botowski, G. Bouchitt{\'e}:
	\newblock Kantorovich--Rubinstein duality theory for the Hessian.
	\newblock {\em arXiv preprint arXiv:2412.00516}, 2024 (to appear in \textit{Duke Math. J.} \href{https://projecteuclid.org/journals/dmj/duke-mathematical-journal/acceptedpapers}{https://projecteuclid.org/journals/dmj/duke-mathematical-journal/acceptedpapers}).
	
	\bibitem{bolbotowski2025}
	K. Bo{\l}botowski, G. Bouchitt{\'e}:
	\newblock Sharp inequalities between Zolotarev and Wasserstein distances in
	$\mathcal{P}_2(\Rd)$.
	\newblock {To be published soon.}
	
	\bibitem{bouchitte2001}
	G. Bouchitt{\'e}, G. Buttazzo:
	\newblock Characterization of optimal shapes and masses through
	Monge-Kantorovich equation.
	\newblock {\em J. Eur. Math. Soc.}, 3:139--168, 2001.
	
	\bibitem{brenier1991}
	Y. Brenier:
	\newblock Polar factorization and monotone rearrangement of vector-valued
	functions.
	\newblock {\em Commun. Pure Appl. Math.}, 44:375--417,
	1991.
	
	\bibitem{bruckerhoff2022instability}
	M. Br{\"u}ckerhoff, N. Juillet:
	\newblock Instability of martingale optimal transport in dimension $d \geq 2$.
	\newblock {\em Electron. Commun. Probab.}, 27:1--10, 2022.
	
	\bibitem{ciosmak2021}
	K.J. Ciosmak:
	\newblock Optimal transport of vector measures.
	\newblock {\em Calc. Var. Partial Differ. Equ.} 60 Article no. 230, 2021.
	
	\bibitem{ciosmak2025}
	K.J. Ciosmak:
	\newblock Characterisation of optimal solutions to second-order Beckmann
	problem through bimartingale couplings and leaf decompositions.
	\newblock {\em arXiv preprint arXiv:2502.05683}, 2025.
	
	\bibitem{cuturi2013}
	M. Cuturi:
	\newblock Sinkhorn distances: Lightspeed computation of optimal transport.
	\newblock {\em Adv. Neural Inf. Process. Syst.}, 26:2292–-2300, 2013.
	
	\bibitem{demarch2019}
	H. De March, N. Touzi:
	\newblock Irreducible convex paving for decomposition of multidimensional
	martingale transport plans.
	\newblock {\em Ann. Probab.}, 47:1726--1774, 2019.
	
	\bibitem{dentcheva2003optimization}
	D. Dentcheva, A. Ruszczyński:
	\newblock Optimization with stochastic dominance constraints.
	\newblock {\em 	SIAM J. Optim.}, 14:548--566, 2003.
	
	\bibitem{dentcheva2015}
	D. Dentcheva, E. Wolfhagen:
	\newblock Optimization with multivariate stochastic dominance constraints.
	\newblock {\em SIAM J. Optim.}, 25:564--588, 2015.
	
	\bibitem{dweik2019p}
	S. Dweik, F. Santambrogio:
	\newblock $L^p$ bounds for boundary-to-boundary transport densities, and $W^{1,p}$
	bounds for the BV least gradient problem in 2D.
	\newblock {\em Calc. Var. Partial Differ. Equ.},
	58 Article no. 31, 2019.
	
	\bibitem{fathi2021}
	M. Fathi:
	\newblock Higher-order Stein kernels for Gaussian approximation.
	\newblock {\em Stud. Math.}, 256:241--258, 2021.
	
	\bibitem{gangbo1998optimal}
	W. Gangbo, A. Święch:
	\newblock Optimal maps for the multidimensional Monge-Kantorovich problem.
	\newblock {\em Comm. Pure Appl. Math.}  51:23--45, 1998.
	
	\bibitem{gozlan2020}
	N. Gozlan, N. Juillet:
	\newblock On a mixture of Brenier and Strassen theorems.
	\newblock {\em Proc. Lond. Math. Soc.},
	120:434--463, 2020.
	
	\bibitem{gozlan2017}
	N. Gozlan, C. Roberto, P.-M. Samson, P. Tetali:
	\newblock Kantorovich duality for general transport costs and applications.
	\newblock {\em J. Funct. Anal.}, 273:3327--3405, 2017.
	
	\bibitem{guo2019}
	G. Guo, J. Ob{\l}{\'o}j:
	\newblock Computational methods for martingale optimal transport problems.
	\newblock {\em Ann. Appl. Probab.}, 29:3311--3347, 2019.
	
	\bibitem{gutjahr2016}
	W.J. Gutjahr, A. Pichler:
	\newblock Stochastic multi-objective optimization: a survey on non-scalarizing
	methods.
	\newblock {\em Ann. Oper. Res.}, 236:475--499, 2016.
	
	\bibitem{hanin1994}
	L.G. Hanin, S.T. Rachev.
	\newblock Mass-transshipment problems and ideal metrics.
	\newblock {\em 	J. Comput. Appl. Math.},
	56:183--196, 1994.
	
	\bibitem{hirsch2012}
	F. Hirsch, B. Roynette:
	\newblock A new proof of Kellerer’s theorem.
	\newblock {\em ESAIM - Probab. Stat.}, 16:48--60, 2012.
	
	\bibitem{juillet2016stability}
	N. Juillet:
	\newblock Stability of the shadow projection and the left-curtain coupling.
	\newblock In {\em Ann. inst. Henri Poincar\'{e} (B) Probab. Stat.}, 52:1823--1843, 2016.
	
	\bibitem{kim2024}
	Y.-H. Kim, Y. Ruan:
	\newblock Backward and forward Wasserstein projections in stochastic order.
	\newblock {\em J. Funct. Anal.}, 286 Article no. 110201, 2024.
	
	\bibitem{legruyer2009}
	E. Le~Gruyer.
	\newblock Minimal Lipschitz extensions to differentiable functions defined on a
	Hilbert space.
	\newblock {\em Geom. Funct. Anal.}, 19:1101--1118, 2009.
	
	\bibitem{muller2017}
	A. M{\"u}ller, M. Scarsini, I. Tsetlin, R.L. Winkler:
	\newblock Between first-and second-order stochastic dominance.
	\newblock {\em Manag. Sci.}, 63:2933--2947, 2017.
	
	\bibitem{rachev2000}
	S.T. Rachev, L. R{\"u}schendorf:
	\newblock {\em Mass transportation problems. Volume I: Theory}.
	\newblock Springer, 2000.
		
	\bibitem{rachev1991}
	S.T. Rachev:
	\newblock {\em Probability metrics and the stability of stochastic models.}
	\newblock Wiley, 1991.
	
	\bibitem{rio2009}
	E. Rio:
	\newblock Upper bounds for minimal distances in the central limit theorem.
	\newblock  {\em Ann. inst. Henri Poincar\'{e} (B) Probab. Stat.}, 45:802--817, 2009.
	
	\bibitem{santambrogio2015}
	F. Santambrogio.
	\newblock {\em Optimal transport for applied mathematicians.}
	\newblock Birkh\"{a}user, 2015.
	
	\bibitem{shaked2007}
	M. Shaked, J.G. Shanthikumar:
	\newblock {\em Stochastic orders}.
	\newblock Springer, 2007.
	
	  \bibitem{strassen1965}
	V. Strassen:
	\newblock The existence of probability measures with given marginals.
	\newblock {\em Ann. Math. Statist.}, 36:423--439, 1965.
	
	\bibitem{villani}
	C. Villani:
	\newblock {\em Topics in optimal transportation. Graduate
		Studies in Mathematics} 58.
	\newblock Amer. Math. Soc., Providence, RI, 2003.
	
	\bibitem{wiesel2023continuity}
	J. Wiesel:
	\newblock Continuity of the martingale optimal transport problem on the real
	line.
	\newblock {\em Ann. Appl. Probab.}, 33:4645--4692, 2023.
	
	\bibitem{zolotarev1978}
	V.M. Zolotarev:
	\newblock Ideal metrics in the problem of approximating distributions of sums
	of independent random variables.
	\newblock {\em Theory Probab. Appl.}, 22:433--449,
	1978.
	
\end{thebibliography}

\bigskip

\noindent
Karol Bo{\l}botowski\\
Faculty of Mathematics, Informatics and Mechanics, University of Warsaw\\
2 Banacha Street, 02-097 Warsaw - POLAND\\
{\tt k.bolbotowski@mimuw.edu.pl}

\end{document}